\documentclass[11pt, leqno, twoside]{article}

\usepackage{amssymb}
\usepackage{amsmath}
\usepackage{amsthm}
\usepackage{amsfonts}
\usepackage{mathrsfs}
\usepackage{indentfirst}
\usepackage{graphicx}
\usepackage{txfonts}

\usepackage{anysize}

\usepackage{fancyhdr}

\usepackage{color}
\usepackage{setspace}

\usepackage{hyperref}
\hypersetup{colorlinks=true, linkcolor=blue, citecolor=blue}

\allowdisplaybreaks

\usepackage{txfonts}

\def\rr{{\mathbb R}}
\def\rn{{{\rr}^n}}

\def\zz{{\mathbb Z}}

\def\nn{{\mathbb N}}

\def\cm{{\mathcal M}}

\def\ca{{\mathcal A}}

\def\ch{{\mathcal H}}

\def\cx{{\mathcal X}}
\def\cy{{\mathcal Y}}
\def\cz{{\mathcal Z}}

\def\fz{\infty}

\def\dz{\delta}

\def\lz{\lambda}

\def\lf{\left}
\def\r{\right}

\def\hs{\hspace{0.26cm}}

\def\ls{\lesssim}
\def\gs{\gtrsim}

\def\noz{\nonumber}
\def\wz{\widetilde}

\def\cx{{\mathcal X}}

\def\gfz{\genfrac{}{}{0pt}{}}

\def\loc{{\mathop\mathrm{\,loc\,}}}
\def\esinf{\mathop\mathrm{\,ess\,inf\,}}
\def\esup{\mathop\mathrm{\,ess\,sup\,}}

\def\q1{\wz q}
\def\Q1{q_1}

\def\loc{{\mathop\mathrm{loc\,}}}

\newtheorem{thm}{Theorem}[section]
\newtheorem{prop}[thm]{Proposition}
\newtheorem{lem}[thm]{Lemma}
\newtheorem{cor}[thm]{Corollary}

\theoremstyle{definition}

\newtheorem{rem}[thm]{Remark}

\numberwithin{equation}{section}

\begin{document}
\title{\bf\Large Hausdorff content \(A_p\) weights: weighted maximal inequalities and fundamental structural properties
\footnotetext{\hspace{-0.35cm}
2020 {\it Mathematics Subject Classification}.
{Primary 42B25; Secondary 28A12, 28A25, 31C15.} \endgraf
{\it Key words and phrases.} Hausdorff content \(A_p\) weight, weighted maximal inequality, Choquet integral, reverse H\"older inequality, self-improving property, Jones' factorization theorem. \endgraf}}
\author{Long Huang, Yangzhi Zhang and Ciqiang Zhuo\footnote{Corresponding author,
E-mail: cqzhuo87@hunnu.edu.cn
%/{\color{red}{October 9, 2026}}
}}
\date{ }
\maketitle

%\vspace{-0.6cm}

\vspace{-1.2cm}

\begin{center}
\begin{minipage}{13cm}
{\small {\bf Abstract}\quad
Let $\mathcal H_{\infty}^\delta$ denote the Hausdorff content of dimension $\delta\in(0,n]$ defined on subsets of $\mathbb R^n$.
The principal problem, considered in this paper, is to characterize the non-negative function $w$ for which there exists a positive constant $K$ such that
\begin{align}\label{0eq1}
\int_{\mathbb R^n} \left[\mathcal M_{\mathcal H_\infty^\delta}(f)(x)\right]^pw(x)\,d\mathcal H_\infty^\delta
\le K \int_{\mathbb R^n} |f(x)|^pw(x)\,d\mathcal H_\infty^\delta,
\end{align}
and to characterize the non-negative function $w$ such that
\begin{align}\label{0eq2}
w _{\mathcal H^{\delta}_\infty}\lf(\left\{x\in \mathbb R^n:\mathcal M_{\mathcal H^{\delta}_\infty}(f)(x)>t\right\}\r)
\le \frac K{t}\int_{\mathbb R^n} |f(x)|w(x)\,d\mathcal H_\infty^\delta,\quad \forall\,t\in(0,\infty),
\end{align}
where $p\in(1,\infty)$ and $\mathcal M_{\mathcal H_{\infty}^\delta}$ is the Hardy-Littlewood maximal operator associated with Hausdorff content
$$\mathcal M_{{\mathcal H}_\infty^\delta}(f)(x):=\sup_{Q\ni x}\frac1{{\mathcal H}_\infty^\delta(Q)}\int_Q |f(y)|\,d{\mathcal H}_\infty^\delta$$
with the supremum being taken over all cubes $Q$ containing $x$.
To achieve this, we introduce a class of Hausdorff content \(A_p\) weights, denoted as $\mathcal A_{p,\delta}$, and then show that, for any $p\in(1,\infty)$ and $\delta\in(0,n]$, the weighted maximal inequality
\eqref{0eq1} holds true if and only if $w\in\mathcal A_{p,\delta}$,
and \eqref{0eq2} holds true if and only if $w\in\mathcal A_{1,\delta}$
by proposing a novel approach.
As the second objective, applying this approach, the seminal structural properties of classical Muckenhoupt $A_p$ weights, such as the reverse H\"older inequality [R. R. Coifman and C. Fefferman, Studia Math. 51 (1974), 241-250], the self-improving property [B. Muckenhoupt,
Trans. Amer. Math. Soc. 165 (1972), 207-226], and Jones' factorization theorem [P. W. Jones, Ann. of Math. (2) 111 (1980), 511-530], are all established for $\mathcal A_{p,\delta}$.
}

\setlength{\parindent}{1.2em}
{\small The main novelties of this paper lies in: (i) we discover the Hausdorff content \(A_p\) weights $\mathcal A_{p,\dz}$
with rich structural properties, and prove that they are strictly increasing with respect to the dimension exponent \(\delta\). This unexpected monotonicity further yields a precise connection with the classical $A_p$ classes:
$$\mathcal A_{p,\dz} \subsetneqq \mathcal A_{p,\,n} \sim A_p, \quad \forall\,p\in[1,\infty)\ \text{and}\ \delta\in(0,n);$$
(ii) we propose a novel approach in the Hausdorff content setting to build up a full analogue of Muckenhoupt's theorem, reverse H\"older inequality and Jones' factorization theorem, etc., where the classical methods fail. This work naturally extends the classical theory beyond measure-theoretic frameworks at $\delta\in(0,n)$, while also provides novel insights---even when reduced to the classical $A_p$ weight case (i.e., $\delta=n$)---that avoids the use of Fubini's theorem, the countable additivity of measures and linearity of the integral.
}
\end{minipage}
\end{center}

\tableofcontents

\vspace{0.2cm}

\section{Introduction and Main Results}

Given a set $E\subset \rn$ and $\dz \in (0,n]$,
the \emph{Hausdorff content} $\mathcal{H}_{\infty }^{\delta } (E )$ is defined by setting
\begin{align*}
\mathcal{H}_{\infty }^{\delta }\lf (E \r )
:=\inf\lf \{\sum_i [l(Q_i)]^{\delta }:\
E\subset \bigcup_i Q_i \right \},
\end{align*}
where the infimum is taken over all finite or countable cubes from $\rn$ coverings
$\{Q_i\}_{i}$ of $E$ and $l(Q)$ denotes the edge length
of the cube $Q$. Here and in what follows, a cube $Q\subset \rn$ always denotes a left-open and right-closed cube in $\rn$ with side parallel to the coordinate axes. The quantity $\mathcal{H}_{\infty }^{\delta }\lf (E \r )$ is also referred to as the $\delta$-\emph{Hausdorff content}
or the $\delta$-\emph{Hausdorff capacity} or the \emph{Hausdorff content of $E$ of
dimension} $\delta$. The $\delta$-Hausdorff content
allows us to measure certain ``very small" subsets of $\rn$, such as lower dimensional subsets or fractal subsets of $\rn$; see Remark \ref{rem-HC} below.

Given $p\in(0,\fz)$ and any $\ch_\fz^\delta$-almost everywhere defined function $g$ on $\rn$, its
\emph{Choquet integral} with respect to Hausdorff content $\ch_\fz^\delta$ is defined to be
$$\int_\rn |g(x)|^p\,d\ch_\fz^\delta:=p\int_0^\fz t^{p-1}\ch_\fz^\delta(\{x\in\rn:\ |g(x)|>t\})\,dt.$$
The monotonicity of the set function $\ch_\fz^\delta$ implies that the Lebesgue integral on the right hand side is always well defined, even when $g$ is a non-measurable function in the sense of Lebesgue measures. Thus, the Choquet integral is well defined for all non-negative functions, without the assumption of measurability.
It is worth noting that Hausdorff content, originating from fractal geometry (see \cite{eg15,fe69}), play a crucial role in harmonic analysis and nonlinear potential theory (see \cite{AdH96}).
The Choquet integral also has many essential and profound applications in  quasilinear elliptic equations (see, for example, T. Kilpel\"ainen and J. Mal\'y \cite{km94} as well as D. Labutin \cite{l02}), in continuous time dynamic and coherent risk measures in finance (see, for example, L. Denis, M. Hu and S. Peng \cite{dhp11}) and in Bayesian decision theory, subjective probability and robust optimization (see, for example, D. Bertsimas, D. Brown and C. Caramanis \cite{bbc11}).
For the further theory of Choquet integrals with respect to Hausdorff content, we refer the reader to the original work of G. Choquet \cite{ch55}, the excellent survey by D. R. Adams \cite{0923-2} for an overview.

Let $L_{\loc}^1({{\mathbb R}^n})$ denote
the set of all locally integrable functions on
${{\mathbb R}^n}$.
The classical \emph{Hardy-Littlewood maximal operator} $\mathcal M$ is defined by setting
\[\mathcal M(f)(x):=\sup_{Q\ni x}\frac1{|Q|}{\int_Q|f(y)|\,dy}
,\quad \forall\,x\in\rn\ \text{and}\ f\in L_{\loc}^1(\rn).\]
Equivalently, the supremum may be
taken over all cubes $Q$ centered at $x$, or over all balls $B$ containing $x$.
Let $w\in L_{\loc}^1(\rn)$ and $w\in(0,\fz)$ almost everywhere.  Given $p\in[1,\fz)$ and a measurable set $E$, let $w(E):=\int_Ew(x)\,dx$ and let $f$ belong to weighted Lebesgue spaces $L^p_w(\rn)$, i.e.,
$$\int_\rn |f(x)|^pw(x)\,dx<\fz.$$
As one of fundamental results in harmonic analysis, the following
weighted norm inequality for maximal operator $\mathcal M$,
\begin{equation}\label{eq12-31-c}
\int_\rn \lf[\cm(f)(x)\r]^pw(x)\,dx
\le K\int_\rn |f(x)|^pw(x)\,dx,
\end{equation}
with $p\in(1,\fz)$ and $K$ being a positive constant independent of $f$, and the weak-type weighted norm inequality in the endpoint case $p=1$,
\begin{equation}\label{eq1-10-a}
w\lf(\lf\{x\in \mathbb R^n:\mathcal M(f)(x)>t\r\}\r)
\le  \frac K{t}\int_\rn |f(x)|w(x)\,dx,\quad \forall\,t\in(0,\fz)
\end{equation}
are characterized, respectively, by the Muckenhoupt class $A_p$,
\[A_p:=\lf\{w:\ [w]_{A_p}:=\sup_Q\lf[\frac{1}{|Q|}\int_Q w (x)\,dx\r]\lf[\frac{1}{|Q|}\int_Q w(x)^{-\frac{1}{p-1}}\,dx\r]^{p-1}<\fz\r\}\]
with the supremum taken over all cubes $Q$, and the Muckenhoupt class $A_1$,
$$A_1:=\lf\{ w:\ [w]_{A_1}:=\esup_{x\in\rn}\frac{\cm w (x)}{ w (x)}<\fz\r\}.$$
Namely, for $p\in(1,\fz)$, \eqref{eq12-31-c}  holds true if and only if $w\in A_p$, and
\eqref{eq1-10-a} holds true if and only if $w\in A_1$.
The original proofs were given by B. Muckenhoupt \cite{Muc72} based on an interpolation argument and a remarkable self-improving property of $A_p$ weights, which says that
$w\in A_p$ implies $w\in A_{p-\varepsilon}$ for some $\varepsilon>0$ (i.e., left-openness property). Later, R. R. Coifman and C. Fefferman \cite{CoFe74} gave a greatly simplified proof by a crucial reverse H\"older inequality for the $A_p$ weight. After that, an elementary proof for \eqref{eq12-31-c} avoiding the reverse H\"older inequality was provided by M. Christ and R. Fefferman \cite{cf83} based essentially on the Calder\'on-Zygmund decomposition. For further results concerning weighted norm inequalities on maximal operators and singular integral operators, we refer the reader to the series of contributions by T. Hyt\"onen \cite{th12,th18} and A. K. Lerner \cite{l07,l13,l20}.

In addition, there is an intimate connection between John-Nirenberg BMO functions, B. Muckenhoupt's $A_p$ weights, reverse H\"older inequalities and C. Fefferman's $H^1-\mathrm{BMO}$ duality theorem. This connection can be roughly described by saying that BMO consists of the logarithms of $A_p$ weights or BMO consists of the logarithms of weights that satisfies the reverse H\"older inequality,  which further gives a new proof of C. Fefferman's duality theorem; see \cite{gr85,s93}. Indeed, the $A_p$ weight theory has already reached a great level of perfection and has found applications in several branches of analysis, from complex function theory to partial differential equations; see, for example, \cite{cf76,fs71,gr85,HMW73,l06,l11,t12,t15}.

Due to the significance of Choquet integrals on Hausdorff content, $A_p$ weights and weighted norm inequality in analysis, in this paper we are interested in identifying the conditions on the function $w$ under which the above weighted norm inequalities \eqref{eq12-31-c} and \eqref{eq1-10-a} remain true when Lebesgue measure $dx$ is replaced by Hausdorff content $d\ch_\fz^\delta$, i.e., in the Choquet integral setting.
Observe that, the Hausdorff content $\ch_\fz^\delta$ not only coincides with the standard Lebesgue measure at $\delta=n$, but also enables the describing lower dimensional or even fractal subsets of $\rn$ when $\delta\in(0,n)$.
Furthermore, we are devoted to build up the most important properties of the class of these functions $w$, including their analogues of self-improving property as in  \cite{Muc72}, reverse H\"older inequality as in \cite{CoFe74} and P. W. Jones' factorization theorem as in \cite{j80}.

More precisely, the first purpose of this paper is to characterize the function $w$ on $\rn$
such that there exists a positive constant $K$ satisfying
\begin{equation}\label{eq12-31-b}
\int_\rn \lf[\cm_{\ch_\fz^\delta}(f)(x)\r]^pw(x)\,d\ch_\fz^\delta
\le K \int_\rn |f(x)|^pw(x)\,d\ch_\fz^\delta,
\end{equation}
when $p\in(1,\fz)$, and characterize the function $w$ on $\rn$
such that
\begin{equation}\label{eq1-9-a}
w _{\mathcal H^{\delta}_\infty}\lf(\lf\{x\in \mathbb R^n:\mathcal M_{\mathcal H^{\delta}_\infty}(f)(x)>t\r\}\r)
\le \frac K{t}\int_\rn |f(x)|w(x)\,d\ch_\fz^\delta,\quad \forall\,t\in(0,\fz).
\end{equation}
Here and thereafter, $\dz \in (0,n]$,
\begin{equation}\label{eq2-21-a}
w _{\mathcal H^{\delta}_\infty}(F):=\int_Fw(x)\,d\ch_\fz^\delta,\quad \forall\,F\subset \rn,
\end{equation}
and $\mathcal M_{\mathcal H^{\delta}_\infty}$ is the \emph{capacitary Hardy-Littlewood maximal operator} with respect to $\ch_\fz^\delta$ defined as
$$\cm_{\ch_\fz^\delta}(f)(x):=\sup_{Q\ni x}\frac1{\ch_\fz^\delta(Q)}\int_Q |f(y)|\,d\ch_\fz^\delta$$
with the supremum taken over all cubes $Q$ containing $x$ and $f\in L_{\loc}^1(\rn,{\mathcal H^{\delta}_\infty})$; see Section \ref{s2} for precise definition.

To settle these problems, it might seem intuitive and immediate to adapt the aforementioned methods for the classical $A_p$ weight theory to the setting of Choquet integrals with respect to Hausdorff content and capacitary maximal operators. However, upon closer examination, one can find that these approaches are not feasible and the thing is far away from immediate. 

A fundamental difficulty lies in the absence of linearity or even sublinear property for the Choquet integrals and the
lack of countable additivity for Hausdorff content, that is, for any non-negative functions $\{f_j\}_{j\in\nn}$ defined on $E$,
$$\int_{E}\sum_{j\in\nn}f_j(x)\,d\ch_\fz^\delta
\neq \sum_{j\in\nn}\int_{E}f_j(x)\,d\ch_\fz^\delta.$$
These two quantities are even not
equivalent (see \cite[p.\,148]{0923-7}). Indeed, for any $M>0$, there exist non-negative functions $\{f_j\}_{j\in\nn}$ and $E$, such that
$$M\int_{E}\sum_{j\in\nn}f_j(x)\,d\ch_\fz^\delta
< \sum_{j\in\nn}\int_{E}f_j(x)\,d\ch_\fz^\delta.$$
This makes that the aforementioned crucial techniques such as, self-improving property  $A_p \Rightarrow A_{p-\varepsilon}$, reverse H\"older inequality and the Calder\'on-Zygmund decomposition, can not be extended to the current Hausdorff content and Choquet integral setting. Consequently, new techniques and approaches are required to investigate the aforementioned weighted norm inequalities.

In this paper, roughly speaking, we are inspired by the idea used in \cite{l08} to address the characterization of weight $w$ satisfying the strong-type $(p,p)$ inequality \eqref{eq12-31-b} and the weak-type $(1,1)$ inequality \eqref{eq1-9-a}. Note that A. K. Lerner \cite{l08} gave an extremely simple proof for Muckenhoupt's theorem that completely
avoids additional ingredients and can be applied to maximal operators with respect to a general basis.
But, differing with \cite{l08} and also differing with \cite{CoFe74,Muc72}, we do not use the linear properties for integrals,
countable additivity for Hausdorff content or Fubini's theorem. Instead, we build
 a sort of Radon-Nikodym Theorem (Proposition \ref{them-0919-4}) for weighted Hausdorff content and
starting from covering arguments, we then develop a new approach which shows that, under some suitable conditions, one can interchange the order of infinite sum with the weighted Choquet integral.
Precisely, inspired by a wonderful idea from Calder\'on-Zygmund decomposition techniques
and some stopping time arguments, we formulate and prove a ``sparse covering property" in the context of Hausdorff content; see Proposition \ref{lem-1101-1} below.
Combining this crucial tools and several other tricks, we discover that
the weighted norm inequalities \eqref{eq12-31-b} and \eqref{eq1-9-a} are exactly characterized by
the Hausdorff content $A_p$ weight class $\mathcal A_{p,\delta}$ introduced in the present paper.

The first finding of this paper is the following theorem. 
In what follows, for any $p\in(1,\fz)$, we denote by $p'$ the H\"older conjugate exponent of $p$, namely $p':=\frac{p}{p-1}$.
We say a function $w$ is a $\ch_\fz^\delta$-\emph{weight} on $\rn$ if it satisfies $w\in L_{\loc}^1(\rn,\mathcal H^\delta_\infty)$ and $w\in(0,\fz)$ almost everywhere on $\rn$ with respect to $\ch_\fz^\delta$.
If a $\ch_\fz^\delta$-weight $w$ satisfies
 \begin{equation}\label{eq12-12-b}
[w]_{\mathcal A_{p,\delta}}:=\sup_{\text{cube}\,Q\,\subset \rn}\lf\{\frac{1}{\mathcal H^{\delta}_\infty(Q)}\int_Qw(x)\,d\mathcal H^{\delta}_\infty\r\}\lf\{\frac{1}{\mathcal H^{\delta}_\infty(Q)}\int_Qw(x)^{-\frac{1}{p-1}}\,d\mathcal H^{\delta}_\infty\r\}^{p-1}< \fz;
\end{equation}
with $p\in(1,\fz)$, then $w$ is called a \emph{Hausdorff content $A_p$ weight} with respect to the Hausdorff content $\ch_{\fz}^{\dz}$, denoted as $w\in \mathcal A_{p,\delta}$.

\begin{thm}\label{them-1013-1}
Let $\delta \in(0, n]$, $p\in(1,\infty)$ and $w$ be a $\ch_\fz^\delta$-weight. Then the following statements are equivalent
\begin{enumerate}
\item[{\rm(i)}]  $w\in \mathcal A_{p,\delta}$;

\item[{\rm(ii)}] there exists a positive constant $K=K(n,\delta, p)$ such that
\begin{align}\label{eq919-1}
\int_{\rn}\lf[\mathcal M_{\mathcal H^{\delta}_{\fz}}(f)(x)\r]^pw(x)\,d\mathcal H^{\delta}_{\fz}
\leq K[w]^{4p'+\frac{2}{p}}_{\ca_{p,\delta}}\int_{\rn}|f(x)|^pw(x)\,d\mathcal H^{\delta}_{\fz};
\end{align}

\item[{\rm(iii)}] there exists a positive constant $K=K(n,\delta, p)$ such that
\begin{equation*}
w _{\mathcal H^{\delta}_\infty}\lf(\lf\{x\in \mathbb R^n:\mathcal M_{\mathcal H^{\delta}_\infty}(f)(x)>t\r\}\r)
\le \frac {K[w]^{4p'+\frac{2}{p}}_{\ca_{p,\delta}}}{t^p}\int_\rn |f(x)|^pw(x)\,d\ch_\fz^\delta,\quad \forall\,t\in(0,\fz).
\end{equation*}
\end{enumerate}
\end{thm}

\begin{rem}\label{rem916-1}
\begin{enumerate}
\item[(i)] The dependence of the constants as functions of $[w]_{\mathcal A_{p,\delta}}$ in Theorem \ref{them-1013-1}
are not expected to be optimal compared with the case of standard Lebesgue measure due to S. M. Buckley \cite{buck93}:
\[
\int_{\mathbb R^n} [\cm (f)(x)]^p w(x)\,dx
\lesssim [w]_{A_p}^{p'}
\int_{\mathbb R^n} |f(x)|^p w(x)\,dx,\quad 1<p<\infty.
\]
There are at least two reasons to illustrate this point in the present nonlinear Choquet integrels framework.
In one aspect, to handle the lack of linearity of Choquet integrals, we need to repeatedly invokes the doubling property of the $\ch_\fz^\delta$-weights $\mathcal A_{p,\delta}$ (see Lemma \ref{lm-0919-1});
in another aspect, to fill the gap left by the absence of Fubini's theorem in the weighted Choquet integral setting,
we need to use Proposition \ref{them-0919-4}, which partly enlarges the growth of constants on $[w]_{\mathcal A_{p,\delta}}$.
However, such a growth automatically disappears in the setting of standard Lebesgue measure and, as being pointed out in Remark \ref{rem-0919-1} by constructing counterexamples, it is unavoidable within this nonlinear framework.

\item[(ii)] In the proof of Theorem \ref{them-1013-1}, the weak-type estimate (iii) is obtained directly from (ii) by Chebyshev’s inequality and thus inherits the same constant.
    However, if $w\in \ca_{p,\delta}$, then a more general argument in the proof of Theorem \ref{them-0919-1}
     improves the dependence of the constant to $[w]_{\ca_{p,\delta}}^{3+\frac1p}$; see \eqref{eq1017-3}.
\end{enumerate}
\end{rem}

For the endpoint $p=1$, we also obtain the following characterization for $\ch_\fz^\delta$-weight $w$
satisfying weak-type inequality \eqref{eq1-9-a}.
In what follows, a property 
is said to hold \emph{$\mathcal{H}^\delta_\infty$--almost everyehere}, shortly \emph{$\mathcal{H}^\delta_\infty$--a.\,e.}, if it holds everywhere except on a set $E$ with $\mathcal{H}^\delta_\infty(E)=0$.
A $\ch_\fz^\delta$-weight $w$ satisfying \begin{equation}\label{eq12-12-a}
[w]_{\ca_{1,\delta}}:=\inf\lf\{K\in(0,\fz):\ \cm_{\ch_\fz^\delta}(w)(x)\leq K w (x)\ {\rm for}\ \ch_\fz^\delta{\rm-a.\ e.}\ x\in\rn\r\}<\fz.
\end{equation}
is called the \emph{Hausdorff content $A_1$ weight} with respect to $\ch_\fz^\delta$, denoted as $w\in \mathcal A_{1,\delta}$.

\begin{thm}\label{them-0919-1}
Let $\delta \in(0, n]$ and $w$ be a $\ch_\fz^\delta$-weight. Then the following statements are equivalent
\begin{enumerate}
\item[{\rm(i)}] $w\in \mathcal A_{1,\delta}$;

\item[{\rm(ii)}] there exists a positive constant $K=K(n,\delta)$ such that
\begin{equation*}
w _{\mathcal H^{\delta}_\infty}\lf(\lf\{x\in \mathbb R^n:\mathcal M_{\mathcal H^{\delta}_\infty}(f)(x)>t\r\}\r)
\le \frac {K[w]^{4}_{\ca_{1,\delta}}}{t}\int_\rn |f(x)|w(x)\,d\ch_\fz^\delta,\quad \forall\,t\in(0,\fz).
\end{equation*}

\end{enumerate}
\end{thm}

\begin{rem}\label{rem-A1-delta}
The constant in the weak-type $(1,1)$ bound of Theorem \ref{them-0919-1} depends explicitly on the $\mathcal A_{1,\delta}$ characteristic $[w]_{\mathcal A_{1,\delta}}$ with exponent $4$. In the standard Lebesgue measure case, the best bound of weak-type $(p,p)$ estimate, with $p\in[1,\fz)$, is
$[w]_{A_p}^{1/p}$; see S. M.  Buckley \cite{buck93} and B. Muckenhoupt \cite{Muc72}. We now compare our bound of Hausdorff content weighted estimate against these classical Lebesgue measure case.
For $p\in(1,\infty)$, the relative growth factor of our strong-type exponent versus its classical counterpart satisfies
\[
4<\frac{\text{our strong exponent}}{\text{classical sharp strong exponent}}=\frac{4p'+\frac{2}{p}}{p'}\le 4.5,
\]
while for $p\in[1,\infty)$, the corresponding relative factor for weak-type bounds obeys
\[
3<\frac{\text{our weak exponent}}{\text{classical sharp weak exponent}}=3+\frac{1}{p}\le 4.
\]
This shows that our weak-type estimates exhibit milder relative growth with respect to the classical sharp dependence, when compared with our strong-type result. This discrepancy originates from differences in proof technique: the argument establishing weak-type $(p,p)$ boundedness invokes Proposition \ref{them-0919-4} fewer times than the strong-type $(p,p)$ proof, and therefore accumulates less constant inflation from the nonlinear framework ingredient.
\end{rem}

\begin{rem}
We point out that, when weights and functions being Lebesgue measurable, Theorems \ref{them-1013-1} and \ref{them-0919-1} revisit the classical Muckenhoupt theorem for the maximal operator $\mathcal M$.
Recall that when $\delta=n$, the Hausdorff content $\mathcal H_\infty^n$ is comparable with the Lebesgue measure when restricted to Lebesgue measurable sets. More precisely, there exist positive constants $K_1(n)$ and $K_2(n)$ such that, for every measurable subset $E\subset \rn$,
\[
K_1(n)\mathcal H_\infty^n(E)\le |E|\le K_2(n)\mathcal H_\infty^n(E),
\]
which follows essentially from the well-known relation between the $n$-dimensional Hausdorff content and Lebesgue measure;
see L. C. Evans and R. F. Gariepy \cite[Chapter 2.2]{eg15} for details.
Under the restrictions to Lebesgue measurable sets and Lebesgue measurable functions,
the Hardy-Littlewood maximal operator $\mathcal M_{\mathcal H_\infty^n}$ and the classical Hardy-Littlewood maximal operator $\mathcal M$ are comparable.
Moreover, restricting to Lebesgue measurable weights and subsets, the Hausdorff content $A_p$ weight class $\mathcal A_{p,n}$ is equivalent to the classical Muckenhoupt $A_p$ class, denoted by $\mathcal A_{p,\,n} \sim A_p$.
Therefore, under the additional restriction to Lebesgue measurable functions and weights, Theorems \ref{them-1013-1} and \ref{them-0919-1} imply the classical Muckenhoupt theorem:
the maximal operator $\mathcal M$
is of strong-type $(p,p)$ for $p\in(1,\infty)$ if and only if $w\in A_p$, and is of weak-type $(1,1)$ if and only if $w\in A_1$.
\end{rem}

\begin{rem}
To establish Theorems \ref{them-1013-1} and \ref{them-0919-1}, we develop a novel approach within the setting of Choquet integrals, where classical methods are no longer applicable.
Furthermore, even within the classical $A_p$ weight setting, the method developed in this paper for proving Theorems \ref{them-1013-1} and \ref{them-0919-1} provides novel insights, which avoids using the linearity of integrals, countable additivity of measures, or Fubini’s theorem.
\end{rem}

\begin{rem}
Analogously to the classical Muckenhoupt $A_p$ weights, the newly introduced Hausdorff content $A_p$ weight class $\mathcal A_{p,\delta}$ is monotone increasing with respect to the index $p$, i.e., $\mathcal A_{p_1,\delta}\subset\mathcal A_{p_2,\delta}$ whenever $1\le p_1\le p_2<\infty$.
One somewhat unexpected phenomenon is that the class $\mathcal A_{p,\delta}$ also enjoys strict monotonicity with respect to the dimension exponent $\delta$ of Hausdorff content. This further yields a precise relation with the classical $A_p$ weights:
\[
\mathcal A_{p,\delta} \subsetneqq \mathcal A_{p,\,n} \sim A_p, \quad \forall\,p\in[1,\infty),\ \delta\in(0,n);
\]
see Corollary \ref{cor25-12-14} below.
\end{rem}

One of the main novelties of Theorems \ref{them-1013-1} and \ref{them-0919-1} lie in the fact that the dimension $\delta$ of the Hausdorff content $\ch_\fz^\delta$ is allowed to be strictly less than $n$. Indeed, the case of $\delta\in(0,n)$ is more delicate. In this setting, the Hausdorff content $\ch_\fz^\delta$ can be used to describe lower dimensional or even fractal subsets of $\rn$. Moreover,  $\ch_\fz^\delta$ may be identified with the standard definition of an outer measure. However, the set function $\ch_\fz^\delta$, when $\delta\in(0,n)$, can not be restricted to a nontrivial sigma algebra so as to be an additive measure there. In fact, it fails to be, what in measure theory is called, a ``metric outer measure"; see D. R. Adams \cite{0923-2}.

Notice that the theory of classical $A_p$ weights has had quite a success,
which is, in part, a consequence of the theory of classical $A_p$ weights carries through to the situation in which Lebesgue measure $dx$ is replaced by a general doubling measure $w(x)\,dx$. That is, under the Lebesgue integral setting, via Fubini's theorem, for any non-negative measurable function $w$ and $f\in L^1_w(\rn)$, it is easy to see
\[\int_{\rn}|f(x)|w(x)\,dx=\int_{\rn}|f(x)| dw.\]
Then one can show that, when $w\in A_p$, $w (F):=\int_F\,dw$, with $F\subset\rn$, is a new measure satisfying doubling condition: $w(2B)\le Cw(B)$ for any ball $B\in \rn$.
But, in the case of Choquet integrals with $0<\dz<n$, the following kind of
Fubini's theorem
\[\begin{aligned}
\int_{\mathbb R^n}\int^{\infty}_{0}f(x,t)w (x)\,dt\,d\mathcal H^{\delta}_{\infty}\sim\int^{\infty}_0\int_{\mathbb R^n}f(x,t) w (x)\,d\mathcal H^{\delta}_{\infty}\,dt.
\end{aligned}\]
fails in general [see Counterexample 2 in Remark \ref{rem-0919-1}(ii) below], and hence the equivalence
\begin{align}\label{eq6-11}
\int_{\rn}|f(x)|w(x)\,d\ch_\fz^\delta
\sim\int_{\rn}|f(x)| dw_{\ch_\fz^\delta}
\end{align}
does not hold for a general weight $w$ [see Counterexample 1 Remark \ref{rem-0919-1}(i) below],
where $w_{\ch_\fz^\delta}$ is defined as in \eqref{eq2-21-a}.

Surprisingly, given $p\in[1,\infty),$ we discover, in Proposition \ref{them-0919-4} below,
that the condition $w\in\mathcal A_{p,\delta}$ ensures  the Radon-Nikodym-type Theorem \eqref{eq6-11}.
This is unexpected compared to the classical situation
and is also one of the key points in this work.
The main idea behind the proof of \eqref{eq6-11} can be summarized as follows.
Let $w\in \ca_{p,\delta}$ and
$$E_k:=\lf\{x\in \mathbb R^n:2^{k-2}<|f(x)|w(x)\leq2^k\r\},\ \ \ \forall\, k\in \zz.$$
Then we find that the sets $G_j:=\{x\in\rn:\ 2^{j-1}<w(x)\le 2^j\}$, $\forall j\in\zz$, satisfy
$E_k=\bigcup_{j\in\zz}(E_k\cap G_j)$ and, more importantly, we obtain the following inequality
\begin{equation*}
\sum_{j\in\nn}\ch_\fz^\delta(E_k\cap G_j)\ls \ch_\fz^\delta(E_k),
\end{equation*}
which may be invalid when $\delta<n$ for an arbitrary sequence $\{G_j\}_{j\in\zz}$.

As another obstacle for the case $\delta<n$, it is typically invalid that there exists a constant $K$ such that if $Q_1,\ldots,Q_m$ are non-overlapping dyadic cubes and $f\ge 0$, then
$$
\sum_{j=1}^m\int_{Q_{j}}|f(x)|w(x)\,d\mathcal H^{\delta}_{\infty}\le K\int_{\cup_{j=1}^m Q_{j}}|f(x)|w(x)\,d\mathcal H^{\delta}_{\infty}.
$$
This can be shown by subdividing the interval $(0,1]$ in $m$ equal left-open and right-closed intervals, with $m$ large enough and taking $f\equiv w\equiv1$. Nevertheless, this inequality, even in the Lebesgue integral setting, is also essential both in establishing the characterization for weighted strong-type and weak-type boundedness.
To overcome this difficulty, we propose a ``weighted packing condition" (see Proposition \ref{lm-1101-9} below)
and prove a ``sparse covering property" in the context of Hausdorff content (see Proposition \ref{lem-1101-1} below).
Applying these key techniques, in Proposition \ref{lm-1101-9},
 we show that when $w\in \ca_{p,\delta}$ there exists a positive constant $K$ such that
\begin{align}\label{04-5}
\sum_{j\in\nn}\int_{Q_{j}}|f(x)|w(x)\,d\mathcal H^{\delta}_{\infty}\le K\int_{\cup_{j\in\nn} Q_{j}}|f(x)|w(x)\,d\mathcal H^{\delta}_{\infty},
\end{align}
where $\{Q_j\}_{j\in\nn}$ is a family of non-overlapping dyadic cubes of $\rn$ satisfying the \emph{weighted packing condition}:
there exists a constant $\beta\in(0,\fz)$ such that, for each dyadic cube $Q\subset \rn$,
\[
\sum_{Q_{j}\subset Q}w_{\mathcal H^\delta_\infty}(Q_{j})\le \beta\,w_{\mathcal H^\delta_\infty}(Q).
\]
Indeed, the above weighted packing condition is the sufficient and necessary condition for \eqref{04-5} to hold; see
Remark \ref{rem917-1} below.
The inequality \eqref{04-5} plays a central role in partial substituting the linear property on the weighted Choquet integrals and hence we can prove Theorems \ref{them-1013-1} and \ref{them-0919-1}.
Here, we point out that the ``sparse covering property" established in Proposition \ref{lem-1101-1}
seems to be new even when reduced to the classical setting, i.e., $\delta=n$, with independent significance and potential applicability to other contexts.

As an application of Theorems \ref{them-1013-1} and \ref{them-0919-1}, we obtain the following weighted maximal inequalities
for classical Hardy-Littlewood maximal operators $\mathcal M$ on Choquet integrals by an interpolation argument.

\begin{cor}\label{Cor-0919-1}
Let $\delta\in (0, n]$ and $f\in L_{\loc}^1({{\mathbb R}^n})$.
\begin{enumerate}
\item[{\rm(i)}]
If $w\in \ca_{p,\delta}$ with $p\in (1,\fz)$, then there exists a positive constant $K=K(n, \dz, p)$ such that
\[\int_{\mathbb R^n}|\mathcal M(f)(x)|^{\frac{p\dz}{n}} w (x)\,d\mathcal H^{\delta}_\infty\leq K[w]^{4p'+\frac{2}{p}}_{\ca_{p,\dz}}\int_{\mathbb R^n}|f(x)|^{\frac{p\dz}{n}} w(x)\,d\mathcal H^{\delta}_\infty.\]
\item[{\rm(ii)}]
If $w\in \ca_{p,\delta}$ with $p\in [1,\fz)$ and $q\in (\frac{p\dz}{n},\fz )$, then there exists a positive constant $K=K(n, \dz, p, q)$ such that
\[\int_{\mathbb R^n}|\mathcal M(f)(x)|^q w (x)\,d\mathcal H^{\delta}_\infty\leq K[w]^{3+\frac{2}{p}}_{\ca_{p,\dz}}\int_{\mathbb R^n}|f(x)|^q w (x)\,d\mathcal H^{\delta}_\infty.\]
\item[{\rm(iii)}]
If $w\in \ca_{1,\delta}$  and $q=\frac{\dz}{n}$, then there exists a positive constant $K=K(n,\dz)$ such that
\[ w _{\mathcal H^{\delta}_\infty}\lf(\lf\{x\in \mathbb R^n:\mathcal M(f)(x)>t\r\}\r)
\leq \frac{K[w]^5_{1,\dz}}{t^q}\int_{\mathbb R^n}|f(x)|^q w (x)\,d\mathcal H^{\delta}_\infty,\quad t\in(0,\fz).\]
\end{enumerate}
\end{cor}

The second objective of this paper is applying the novel approach developed in this paper to further study the structural properties of Hausdorff content $A_p$ weight class $\mathcal A_{p,\delta}$ with full range $p\in[1,\infty)$
and $\delta\in (0, n]$.
To this end, we first establish the reverse H\"older inequality for $\ca_{p,\delta}$---the one of deepest and most significant parts of the whole theory---which naturally extends the seminal inequality of R. R. Coifman and C. Fefferman \cite[Theorem IV]{CoFe74}. Indeed, the reverse H\"older inequality appeared in the renowned work of F. W. Gehring \cite{g73} in the following context: If $F$ is a quasiconformal homeomorphism fron $\rn$ into itself, then $|{\rm det}(\nabla F)|$ satisfies a reverse H\"older inequality.
It should be mentioned that the reverse H\"older inequality also serves an important role in such diverse areas as quasiconformal mappings (see, for example, F. W. Gehring \cite{g73}) and certain refined estimates for elliptic partial differential equations (see, for example, E. Fabes, D. Jerison and C. Kenig \cite{fjk84} and also  R. Fefferman \cite{rf89}). 

\begin{thm}\label{lm-411-1}
(Reverse H\"older's inequality) Let $\delta\in(0, n]$, $p\in[1,\infty)$ and $w\in\ca_{p,\delta}$. Then there exists positive constants $K=K(n, \delta, p, [w]_{\ca_{p,\delta}})$ and $\gamma=\gamma(n, \delta, p, [w]_{\ca_{p,\delta}})\in(0,1)$ such that, for every cube $Q \subset \rn$,
\begin{align}\label{eq-411-1}
\lf[\frac{1}{\ch_{\infty}^{\delta}(Q)}\int_{Q}w(x)^{1+\gamma}\,d\ch_{\infty}^{\delta}\r]^{\frac{1}{1+\gamma}}\leq \frac{K}{\ch_{\infty}^{\delta}(Q)}\int_{Q}w(x)\,d\ch_{\infty}^{\delta}.
\end{align}
\end{thm}

We should point out that employing this reverse H\"older inequality, one can also show the above Theorems \ref{them-1013-1} and \ref{them-0919-1}
as done in \cite{CoFe74}. But, even along this route, one must still overcome the difficulties arising from the lack of Fubini's theorem and the nonlinearity of Choquet integrals, as discussed above.

Having established the crucial reverse H\"older inequality for the Hausdorff content $A_p$ weight class $\mathcal A_{p,\delta}$, we now proceed to some important applications. Among them, the first result yields that if $w\in \ca_{p,\delta}$, then automatically $w^{1+\gamma}\in \ca_{p,\delta}$ for some $\gamma\in(0,1)$.

\begin{cor}\label{cor-0714-1}
Let $\delta\in(0, n]$, $p\in [1,\infty)$ and $w\in \ca_{p,\delta}$. Then there exists a constant $\gamma\in(0,1)$ such that $w^{1+\gamma}\in \ca_{p,\delta}$.
\end{cor}

Building upon the reverse H\"older inequality, we infer the following self-improving property of $\ca_{p,\delta}$, which generalizes the left-openness property of the classical Muckenhoupt $A_p$ class originating from B. Muckenhoupt \cite{Muc72}.

\begin{thm}\label{Cor-411-1}
(Left-openness property)
Let $\delta\in(0, n]$, $p\in(1,\infty)$ and $w\in \ca_{p,\delta}$. Then there is a $q=q(n, \delta, p, [w]_{\ca_{p,\delta}})$ with $1<q<p$ such that $w\in \ca_{q,\delta}$.
%In short, we have
%\[\ca_{p,\delta}=\bigcup_{q\in (1,p)}\ca_{q,\delta}.\]
\end{thm}

An application of Theorem \ref{them-1013-1} enables us to build up a complete analogue of Jones' factorization theorem for $\ca_{p,\delta}$, which is a non-trivial extension of the celebrated result developed by P. W. Jones in \cite{j80}.

\begin{thm}\label{lm-410-3}
(Jones' factorization theorem)
Let $\delta\in(0, n]$ and $p\in [1,\infty)$. Then $w\in \ca_{p,\delta}$ if and only if there exist two weights $w_0,\,w_1\in \ca_{1,\delta}$ such that $w=w_0w_1^{1-p}$.
\end{thm}

\begin{rem}
We emphasize that the classical methods fail to establish the last three theorems due to ``bad" properties of the Hausdorff content. Instead, we employ the novel approach developed in the proofs of Theorems \ref{them-1013-1} and \ref{them-0919-1}. Therefore, these results naturally extend the classical theory of \cite{CoFe74,j80,Muc72} beyond measure-theoretic frameworks, and simultaneously provide novel insights even when reduced to the classical Muckenhoupt $A_p$ weight case. For further applications of Hausdorff content $A_p$ weights, we refer the reader to \cite{hzzbmo,hzzgood}.
\end{rem}

The paper is organized as follows. In Section \ref{s2}, we mainly give some fundamental properties of Choquet integrals, Hausdorff content $A_p$ weights and several required lemmas. In particular, we give an example of Hausdorff content $A_p$ weight function and show that the new weight class $\mathcal A_{p,\delta}$ enjoys the strict monotonicity on the dimension exponent $\delta$ of Hausdorff content in the end of Section \ref{s2}.
In Section \ref{s3}, we prove Theorem \ref{them-1013-1}, Theorem \ref{them-0919-1} and Corollary \ref{Cor-0919-1}.
Section \ref{s4} is devoted to showing the reverse H\"older inequality and self-improving property
for the $\ca_{p,\delta}$-weight. To achieve this, we establish an estimate for the capacity of a upper level set for weight $w\in \ca_{p,\delta}$ (see Lemma \ref{lm-410-7} below)
by using the Calder\'on-Zygmund decomposition corresponding to Hausdorff content.
As applications, in Section \ref{s5}, we build up Jones' factorization theorem within the $\ca_{p,\delta}$ framework and show Theorem \ref{them-411-1}.

Throughout the paper, the notation $f\ls g$ (resp. $f\gs g$) means $f\le Kg$ (resp. $f\ge Kg$) for some positive constant $K$ independent of the main parameters, and $f\sim g$ stands for $f\ls g\ls f$; we write $K(\alpha,\beta,\dots)$ to denote a positive constant depending only on the parameters $\alpha,\beta,\dots$, whose value may change from line to line. We denote by $\mathbf{1}_E$ the characteristic function of a set $E\subset\rn$, and for any $x\in\rn$ and $r\in(0,\fz)$, $Q(x,r)$ stands for the cube centered at $x$ with side length $r$ while $B(x,r)$ denotes the open ball centered at $x$ with radius $r$; given a cube $Q\subset\rn$ and $\alpha\in(0,\fz)$, the scaled cube $\alpha Q$ is defined as the cube sharing the same center as $Q$ and with side length scaled by factor $\alpha$. Unless otherwise specified, any cube $Q\subset\rn$ is understood to be an $n$-dimensional axis-aligned cube in $\rn$, and lower-dimensional embedded cubes will always be explicitly indicated by specifying their dimension.

\section{Preliminaries and Fundamental Tools}\label{s2}

In this section, we first recall the definition of Choquet integral for a general capacity
and then give some useful and basic properties with respect to Choquet integrals in Subsection \ref{s2.1}. Then, via weighted packing lemma and
sparse covering property,
a substitute for Fubini's theorem and a substitute for the linearity of Choquet integrals are established in Subsection \ref{s2.2} and Subsection \ref{s2.3}, respectively. Finally, we give an example of Hausdorff content $A_p$ weight function and show that the new weight class $\mathcal A_{p,\delta}$ enjoys the strict monotonicity both on the first index $p$ and the dimension exponent $\delta$ of Hausdorff content in the end of this section.

\subsection{Capacity and Choquet Integrals}\label{s2.1}
A real-valued set function $C$, defined on all subsets of $\rn$, is called a capacity, if it satisfies the following conditions:
\begin{enumerate}
\item[(i)] $C(\emptyset)=0$ and, for any set $E\subset \mathbb R^n$, $C(E)\ge0$;
\item[(ii)] If $E\subset F$, then $C(E)\le C(F)$;
\item[(iii)] If $E=\bigcup_{i=1}^{\fz}E_i$, then
$C(E)\le \sum_{i=1}^{\fz}C(E_i)$,
\end{enumerate}
i.e., $C$ is a non-negative, monotone and countably subadditive set function.

Let $C$ be a capacity. Then, for any subset $E\subset \rn$, the \emph{Choquet integral}
of a non-negative function $f$ on $E$ is defined by setting
\begin{align*}
\int_E f(x)\,dC:=\int_{0}^{\infty}
C(\{x\in E:\
f(x)>\lambda\})\,d\lambda.
\end{align*}
Since $C$ is monotone, then,
for any function $g$ defined on $E$, its distribution
function in the sense of the capacity
$$\lambda\mapsto C(\{x\in E:\ |g(x)|>\lambda\})$$
is decreasing on $\lambda\in[0,\fz)$.
Thus, we easily find that the above distribution function
is measurable with respect to the Lebesgue measure.
Therefore, $$\int_{0}^{\infty}C(\{x\in E:\ |g(x)|>\lambda\})\,d\lambda$$ is a well-defined
Lebesgue integral and hence the Choquet integral is also well defined.
Based on this, for any $p\in(0,\fz)$, the \emph{$p$-Choquet
integral with respect to the capacity} of a function $f$ on $E$ is defined by setting
\begin{align*}
\|f\|_{L^p(E,C)}
:=\lf[\int_E |f(x)|^p\,dC\r]^\frac1p.
\end{align*}
In what follows, for any $p\in(0,\fz)$ and $E\subset \rn$, we use $L^p(E,C)$ to denote
the space of all functions $f$ on $E$ such that the quasi-norm $\|f\|_{L^p(E,C)}<\fz$,
and use $L^{\fz}(E,C)$ to denote the space of all functions $f$ on $E$ such that the quasi-norm
$$\|f\|_{L^\infty(E,C)}:=\inf\lf\{t>0:\ C(\{x\in E:\ |f(x)|>t\})=0\r\}<\fz.$$

Denote by $L_{\loc}^1(\rn,C)$ the set of all functions satisfying that, for any compact set $E\subset\rn$, $f\in L^1(E,C)$.
For $f \in L_{\loc}^1(\rn,C)$, its \emph{Hardy-Littlewood
maximal function} $\cm_C(f)$ is defined by setting, for any $x\in\rn$,
\begin{equation*}
\cm_C(f)(x):=\sup_{Q\ni x}\frac{1}{C(Q)}\int_Q|f(y)|\,dC,
\end{equation*}
where the supremum is taken over all cubes $Q$ of $\rn$ containing $x$ (see \cite{0923-3}).

The following basic properties of the Choquet integral can be found in \cite{0923-2}.
\begin{rem}\label{48-2}
Let $C$ be a capacity.
\begin{enumerate}
\item[(i)]
(H\"{o}lder's inequality) Let $p\in(1,\infty)$.
Then, for any functions $f$ and $g$ on $\rn$, we have
    \[\int_{\rn}|f(x)g(x)|\,dC\leq 2\lf(\int_{\rn}|f(x)|^p\,dC\r)^{\frac{1}{p}}\lf(\int_{\rn}|g(x)|^{p'}\,dC\r)^{\frac{1}{p'}}.\]

\item[(ii)] Let $N\in \nn$, $E\subset \rn$, and $\{f_j\}_{j=1}^N$ be a sequence of functions defined on $E$. Then
\[\int_E\lf|\sum_{j=1}^Nf_j(x)\r|\,dC\leq N\sum_{j=1}^{N}\int_E\lf|f_j(x)\r|\,dC.\]
\end{enumerate}
\end{rem}

\begin{rem}\label{rem1113-1}
Let $C$ be a capacity and $E\subset \rn$.
For any $f\in L^\infty(E,C)$, it is not difficult to obtain
\[\|f\|_{L^\infty(E,C)}=\inf_{C(E_0)=0}\lf\{\sup_{x\in E\backslash E_0}|f(x)|\r\},\]
and, moreover, there exists a subset $E_f\subset E$ with $C(E_f)=0$ such that
\[\|f\|_{L^{\fz}(E,C)}=\sup_{x\in E\backslash E_f}|f(x)|.\]
\end{rem}

We point out that, compared with Riemann or Lebesgue integrals,
the Choquet integral with respect to the Hausdorff content $\ch_{\fz}^{\dz}$
has the following significant differences:
\begin{enumerate}
\item[{\rm(i)}]
The Choquet integral is a nonlinear integral, that is, for any non-negative
functions $f$ and $g$ on $E$,
\begin{align}\label{eq2-21-b}
\int_E f(x)\,d\mathcal{H}_{\infty }^{\delta }+\int_E g(x)\,d\mathcal{H}_{\infty }^{\delta }
&\neq\int_E [f(x)+g(x)]\,d\mathcal{H}_{\infty }^{\delta }\\
&\le 2\left[\int_E f(x)\,d\mathcal{H}_{\infty }^{\delta }
+\int_E g(x)\,d\mathcal{H}_{\infty }^{\delta }\right].\noz
\end{align}
\item[{\rm(ii)}] Choquet integrals are well defined for all non-negative functions without the assumption of measurability.
\end{enumerate}

It is well known that some common examples of capacities include the Hausdorff content, the Riesz capacity and the Bessel capacity
(see \cite{0923-2}). In this paper, we focus our attention on the Hausdorff content $\ch_\fz^\delta$.
Moreover, for any weight $w$ on $\rn$ and $\delta\in(0,n]$, it is not hard to see that
$w_{\ch_\fz^\delta}$ as in \eqref{eq2-21-a} is also a capacity.

To facilitate the understanding of the Hausdorff content $\ch_{\infty}^{\delta}$, we collect in the following remark some elementary facts concerning its behavior on cubes, balls, and lower dimensional embedded sets. These properties will be used later in our proofs and counterexamples.

\begin{rem}\label{rem-HC}
Let $\delta\in(0,n]$, $k\in\{1,2,\dots,n\}$ and suppose that $W$ is a $k$-dimensional cube embedded in $\rn$ with side length $l$. Then
\begin{align}\label{eq917-z}
\ch_{\infty}^{\delta}(W)=
\begin{cases}
l^{\delta}, &\text{when}\ \delta\in (0, k],\\
0, &\text{when}\ k\le n-1 \text{ and } \delta\in (k, n].
\end{cases}
\end{align}
Indeed, when $\delta\in (0, k]$, obviously $\ch_{\infty}^{\delta}(W)\leq l^{\delta}$. On the other hand, for any sequence $\{Q_j\}_{j}$ of cubes in $\rn$ satisfying $W\subset \bigcup_{j}Q_j$, we have $l^k\leq \sum_{j}[l(Q_j)]^{k}$ by definition of the $k$-dimensional Lebesgue outer measure, which further implies $l^{\delta}\leq \sum_{j}[l(Q_j)]^{\delta}$. Therefore, $\ch_{\infty}^{\delta}(W)\geq l^{\delta}$ and consequently $\ch_{\infty}^{\delta}(W)=l^{\delta}$. When $k\le n-1$ and $\delta\in (k, n]$, observing that, for any $m\in\nn$, there exists a sequence $\{P_j\}_{j=1}^{m^k}$ of cubes in $\rn$ with side length $\frac{l}{m}$ such that $W\subset \bigcup_{j=1}^{m^k}P_j$, we find that
$\ch_{\infty}^{\delta}(W)\leq m^{k-\delta}l^{\delta}$ and hence $\ch_{\infty}^{\delta}(W)=0$ by taking $m\to\infty$, which implies that \eqref{eq917-z} holds.

Moreover, if $Q$ is a $n$-dimensional cube in $\rn$ with side length $l$ and
$W\subset Q$ is a $k$-dimensional cube of $\rn$ with $k\in\{1,\dots,n-1\}$,
then, when $\delta\in(0,k]$, by the monotonicity of the Hausdorff content $\ch_{\infty}^{\delta}$ and \eqref{eq917-z}, we have,
for any subset $E\subset\rn$ with $W\subset E\subset Q$,
\[
\ch_{\infty}^{\delta}(E)=l^{\delta}.
\]
This property highlights an important conceptual distinction between the Hausdorff content and Lebesgue measure: the Hausdorff content does not necessarily increase under set theoretic enlargement.
 This counter intuitive behaviour poses technical challenges in the analysis of Choquet integrals with respect to $\ch_{\infty}^{\delta}$, and is employed in our subsequent constructions and counterexamples.

In addition, for any $n$-dimensional cube $Q\subset\rn$, by virtue of \eqref{eq917-z}, we have $\ch^{\delta}_{\fz}(Q)=[l(Q)]^{\delta}$. For any ball $B(x,r)$ with $x\in \rn$ and $r\in (0,\fz)$, since the ball $B(x, r)$ contains a cube of side length $\frac{2\sqrt n}{n}r$ and is contained in a cube of side length $2r$, by the monotonicity of the Hausdorff content $\ch_{\infty}^{\delta}$,
we have
\begin{align}\label{eq917-a}
\lf(\frac{2\sqrt n}{ n}\r)^{\dz}r^{\dz}\leq \ch^{\delta}_{\fz}(B(x, r))\leq 2^{\dz}r^{\delta}.
\end{align}
\end{rem}
Thus, compared with Lebesgue measure,  the reason Hausdorff content is capable of measuring the size of lower-dimensional sets is that it essentially ``flattens out" the redundant dimensions. This also reveals that the smaller 
$\delta$ is, the higher-dimensional the sets it can measure. Compared with Hausdorff measure, its advantage lies in the fact that one need not know the dimension of the set to be measured in advance; it suffices to take 
$\delta$ sufficiently small—smaller than the dimension of the set. This further suggests that for problems concerning Hausdorff content, the small $\delta$
case is more important, and indeed the smallness of $\delta<n$ is precisely where the difficulty lies.

\begin{rem}\label{12-4}
Let $\delta\in(0,n]$.
\begin{enumerate}
\item[(i)] Let $\widetilde{\ch}_\fz^{\delta}$ be the \emph{dyadic Hausdorff content}, which is defined by the same way as
 $\ch^{\delta}_{\fz}$ but with cubes coverings $\{Q_i\}_i$ replaced by dyadic cubes coverings. Let
$$\widetilde{\ch}_\fz^{\delta,0}(E):=\inf \lf\{\sum_{j}[l(Q_j)]^\delta:\ E\subset \lf(\bigcup_jQ_j\r)^\circ\r\},$$
  where the infimum is taken over all such finite or countable dyadic cubes coverings and $F^\circ$ denotes the interior of the set $F$. 
Then there exists a positive constant $K(n, \delta)$ such that, for any subset $E\subset \rn$,
    \begin{align}\label{eq918-5}
    \ch^{\delta}_{\fz}(E)\le \widetilde{\ch}_\fz^{\delta}(E)\le K(n, \delta)\ch^{\delta}_{\fz}(E).
    \end{align}
Moreover, $\widetilde{\ch}_\fz^{\delta,0}$ is also equivalent to $\widetilde{\ch}_\fz^{\delta}$ and is a capacity in the sense of Choquet for $\delta\in(0,n]$; see, for example, \cite{Ada15}.

\item[(ii)] By the definition of $\ch_\fz^\delta$ and the Choquet integral,
we have, for any sequence $\{E_j\}_{j\in\nn}$ of subset in $\rn$,
$$\int_{\bigcup_{j\in\nn}E_j}|f(x)|\,d\ch_\fz^\delta\le \sum_{j\in\nn}\int_{E_j}|f(x)|\,\ch_\fz^\delta.$$

\item[(iii)] By \eqref{eq2-21-b}, we know that the Choquet integral with respect to $\ch_\fz^\delta$ is
not sub-linear. However, using the equivalence of $\ch_\fz^\delta$ and $\widetilde{\ch}_\fz^\delta$, and the sub-linearity
of the Choquet integral with respect to $\widetilde{\ch}_\fz^\delta$ (see \cite[p.\,13, Theorem 1]{0923-2}),
we conclude that there exists a positive constant $K(n, \delta)$ such that, for any non-negative functions $\{f_j\}_{j\in\nn}$ defined on
$E$,
$$\int_{E}\sum_{j\in\nn}f_j(x)\,d\ch_\fz^\delta
\le K(n, \delta)\sum_{j\in\nn}\int_{E}f_j(x)\,d\ch_\fz^\delta.$$

\item[(iv)]  For any $\ch_\fz^\delta$-weight $w$ and subset $E\subset\rn$,
we define the \emph{weighted Choquet-Lebesgue space}
$L^p_w(E,\ch_\fz^\delta)$, with $p\in[1,\fz)$, as the set of all functions $f$ satisfying
\begin{align*}
\|f\|_{L^p_w(E,\ch_\fz^\delta)}
:=&\lf\{\int_E |f(x)|^pw(x)\,d\ch_\fz^\delta \r\}^{\frac{1}{p}}<\fz,
\end{align*}
and the \emph{weak Choquet-Lebesgue space} $L_w^{p,\fz}(E, \ch_{\fz}^{\dz})$ as the set of all functions $f$ satisfying
\begin{align*}
\|f\|_{L^{p,\fz}_w(E,\ch_\fz^\delta)}:=\sup_{t\in(0,\fz)} t\lf(w_{\mathcal H^{\delta}_\infty}(\{x\in E:\,|f(x)|>t\})\r)^{\frac{1}{p}}<\fz,
\end{align*}
where $w _{\mathcal H^{\delta}_\infty}$
is as in \eqref{eq2-21-a}. Also, we define the \emph{space}
$L^{\fz}_w(E,\ch_\fz^\delta)$ as the collection of all functions $f$ satisfying
\begin{align*}
\|f\|_{L^{\fz}_w(E,\ch_\fz^\delta)}:=\inf\lf\{t>0:\  w _{\mathcal H^{\delta}_\infty}(\{x\in E:\,|f(x)|>t\})=0\r\}<\fz.
\end{align*}

\end{enumerate}
\end{rem}

\subsection{Weighted Packing Condition and Substitute for Fubini's Theorem}\label{s2.2}

In this subsection, we mainly prove a weighted packing lemma (Lemma \ref{lm-0919-2})
and a substitute for Fubini's theorem (Proposition \ref{them-0919-4}), which says that the weighted Choquet integral can be
equivalently viewed as a Choquet integral with respect to the new capacity $w_{\ch_\fz^\delta}$ when $w\in \ca_{p,\delta}$.
We begin with the following doubling properties of the Hausdorff content $A_p$ weight.

\begin{lem}\label{lm-0919-1}(Doubling property)
Let $\delta\in(0,n]$ and $ p\in[1, \fz)$. If $ w \in \mathcal A_{p,\delta}$, then
for any cube $Q$ of $\rn$ and any subset $E\subset Q$,
\begin{equation}\label{eq-0929-1}
\lf[\frac{\mathcal H^\delta_\infty(E)}{\mathcal H^\delta_\infty(Q)}\r]^p
\le 2^p[w]_{\ca_{p,\delta}}\frac{ w _{\mathcal H^\delta_\infty}(E)}{ w _{\mathcal H^\delta_\infty}(Q)};
\end{equation}
moreover, for any $t\in[1,\fz)$, we have
$$w _{\mathcal H^{\delta}_\infty}(tQ)\le 2^p[w]_{\ca_{p,\delta}} t^{p\delta} w _{\mathcal H^{\delta}_\infty}(Q).$$
\end{lem}

\begin{proof}
Let $ w \in \mathcal A_{1,\delta}$. Then by definition, for any cube $Q$ of $\rn$ and $\mathcal H^{\delta}_\infty$-almost
everywhere $x\in Q$, we have
\begin{equation*}
\frac{ w _{\mathcal H^{\delta}_\infty}(Q)}{\mathcal H^{\delta}_\infty(Q)}\leq [w]_{\ca_{1,\delta}}w (x).
\end{equation*}
Then, for any subset $E\subset Q$, by taking Choquet integral over $E$ with respect to $\ch_\fz^\delta$, we obtain \eqref{eq-0929-1} in this case.

When $ w \in \mathcal A_{p,\delta}$ with $p\in(1,\fz)$, applying H\"{o}lder's inequality, i.e., Remark \ref{48-2}(i), we find that, for any function $f$ on $\rn$,
\begin{align}\label{2610-1}
&\lf(\frac{1}{\mathcal H^{\delta}_{\infty}(Q)}\int_Q|f(x)|\,d\mathcal H^{\delta}_{\infty}\r)^p\\
&\quad\leq 2^p\frac{1}{[\mathcal H^{\delta}_{\infty}(Q)]^p}\lf[\int_Q|f(x)|^pw(x)\,d\mathcal H^{\delta}_{\infty}\r]\lf[\int_Qw(x)^{-\frac{1}{p-1}}\,d\mathcal H^{\delta}_{\infty}\r]^{p-1}\noz\\
&\quad= 2^p\lf\{\frac{1}{w_{\mathcal H^{\delta}_{\infty}}(Q)}\int_Q|f(x)|^pw(x)\,d\mathcal H^{\delta}_{\infty}\r\}\lf[\frac{w_{\ch_{\fz}^{\delta}}(Q)}{\mathcal H^{\delta}_{\infty}(Q)}\r]\lf\{\frac{1}{\mathcal H^{\delta}_{\infty}(Q)}\int_Qw(x)^{-\frac{1}{p-1}}\,d\mathcal H^{\delta}_{\infty}\r\}^{p-1}\noz\\
&\quad\leq 2^p[w]_{\ca_{p,\delta}}\frac{1}{w_{\mathcal H^{\delta}_{\infty}}(Q)}\int_Q|f(x)|^pw(x)\,d\mathcal H^{\delta}_{\infty}.\noz
\end{align}
Applying this, we then obtain \eqref{eq-0929-1} by letting $f:=\mathbf{1}_E$ for any subset  $E\subset Q$.

On the other hand, observe that, for any cube $Q$ of $\rn$ and $t\in[1,\fz)$,
$\ch_\fz^\delta(tQ)=t^\delta \ch_\fz^\delta(Q)$ due to Remark \ref{rem-HC}. Then it follows from \eqref{eq-0929-1} that, when $w \in \mathcal A_{p,\delta}$ with $p\in[1,\fz)$ and $t\in [1, \fz)$,
$$ w _{\mathcal H^{\delta}_\infty}(tQ)\le 2^p[w]_{\ca_{p,\delta}} t^{p\delta} w _{\mathcal H^{\delta}_\infty}(Q).$$
This finishes the proof of Lemma \ref{lm-0919-1}.
\end{proof}

\begin{lem}\label{lm-0919-3}
Let $\delta\in(0,n]$, $\alpha\in(0,\fz)$, $ w \in \mathcal A_{1,\delta}$ and
$Q$ be a cube of $\rn$.
If $\{E_k\}_{k\in\nn}$ is a sequence of non-overlapping subsets of $Q$ satisfying
\begin{align}\label{6x19}
\sum_{k\in\nn} w _{\mathcal H^\delta_\infty}(E_k)\le \alpha w _{\mathcal H^\delta_\infty}(Q),
\end{align}
then it holds that
\[\sum_{k\in\nn}\mathcal H^{\delta}_{\infty}(E_k)\le \alpha[w]_{\ca_{1,\delta}}\mathcal H^{\delta}_{\infty}(Q).\]
\end{lem}

\begin{proof}
By the definition of $\ca_{1,\delta}$,
it is not difficult to know that, for the given cube $Q$,
\begin{align}\label{eq1014-3}
w _{\mathcal H^{\delta}_{\infty}}(Q)\| w ^{-1}\|_{L^{\infty}(Q,{\mathcal H^{\delta}_{\infty}})}\le[w]_{\ca_{1,\delta}}\mathcal H^{\delta}_{\infty}(Q).
\end{align}
Then, by \eqref{6x19}, we obtain
\[\begin{aligned}
\sum_{k\in \nn}\mathcal H^{\delta}_{\infty}(E_k)&=\sum_{k\in\nn} \int_{E_k}1\,d\mathcal H^{\delta}_{\infty}
\le\sum_{k\in\nn}\int_{E_k} w (x)\,d\mathcal H^{\delta}_{\infty}\| w ^{-1}\|_{L^{\infty}(E_k,{\mathcal H^{\delta}_{\infty}})}\\
&\le\sum_{k\in\nn}\int_{E_k} w (x) \,d\mathcal H^{\delta}_{\infty}\| w ^{-1}\|_{L^{\infty}(Q,{\mathcal H^{\delta}_{\infty}})}
\le\alpha w _{\mathcal H^{\delta}_{\infty}}(Q)\| w ^{-1}\|_{L^{\infty}(Q,{\mathcal H^{\delta}_{\infty}})}\\
&\le \alpha[w]_{\ca_{1,\delta}}\mathcal H^{\delta}_{\infty}(Q),
\end{aligned}\]
which completes the proof of Lemma \ref{lm-0919-3}.
\end{proof}

The weighted packing lemma is stated as follows. This kind of inequality in Lemma \ref{lm-0919-2} (i) is always called weighted packing condition. It is a generalization of J. Orobitg and J. Verdera \cite[Lemma 2]{0923-7} and is of vital importance in this paper.

\begin{lem}\label{lm-0919-2}
Let $\delta\in(0,n]$ and $w$ be a ${\ch_\fz^\delta}$-weight. Assume that $\{Q_j\}_{j\in\nn}$ is a family of non-overlapping dyadic cubes of $\rn$, then there exists a subfamily $\{Q_{j_v}\}_{v=1}^N$, with positive integer $N<\fz$ or $N=\fz$, satisfying
\begin{enumerate}
\item[{\rm(i)}]
for any dyadic cubes $Q$,
$$\sum_{Q_{j_v}\subset Q} w_{\ch_\fz^\delta}(Q_{j_v})\le 2 w_{\ch_\fz^\delta}(Q);$$

\item[{\rm(ii)}] the inequality
$$w _{\mathcal H^\delta_\infty} \lf(\bigcup_{j\in\nn}Q_j\r)
\le 2\sum_{v=1}^N w _{\mathcal H^\delta_\infty}(Q_{j_v}).$$
\end{enumerate}
Moreover, if $w\in \mathcal A_{p,\delta}$ with $p\in[1,\fz)$, then
\begin{enumerate}
\item[{\rm(iii)}] for a family $\{P_j\}_{j\in\nn}$ of non-overlapping dyadic cubes of $\rn$ satisfying $P_j\subset tQ_j$ with any $j\in\nn$ and some $t\in[1,\fz)$, it holds
$$ w _{\mathcal H^{\delta}_\infty} \lf(\bigcup_{j\in\nn}P_j\r)
\le 2^{p+1}[w]_{\ca_{p,\delta}}t^{p\delta}\sum_{v=1}^N w _{\mathcal H^{\delta}_\infty}(Q_{j_v}).$$
\end{enumerate}
\end{lem}

\begin{proof}
Let $j_1=1$. Then obviously we have, for any dyadic cube $Q$ of $\rn$ with $Q_{j_1}\subset Q$,
$$\sum_{Q_{j_v}\subset Q,v\le  1} w _{\mathcal H^\delta_\infty}(Q_{j_v})
=w _{\mathcal H^\delta_\infty}(Q_{j_1})
\le 2 w _{\mathcal H^\delta_\infty}(Q).$$
Now we assume that $\{j_1,j_2,\dots,j_k\}$ have been chosen such that, for any dyadic cube $Q$,
$$\sum_{Q_{j_v}\subset Q,v\le  k} w _{\mathcal H^\delta_\infty}(Q_{j_v})
\le 2 w _{\mathcal H^\delta_\infty}(Q).$$
Then $j_{k+1}$ is defined to be the first index (if that exists) from $\{j_{k}+1,j_k+2,\dots\}$ satisfying that, for any dyadic cube $Q$,
$$\sum_{Q_{j_v}\subset Q,v\le  k+1} w _{\mathcal H^\delta_\infty}(Q_{j_v})\le  2 w _{\mathcal H^\delta_\infty}(Q).$$
By induction, we obtain a subfamily $\{Q_{j_v}\}_{v=1}^N$ satisfying (i).

Next, we prove (ii).
For any $Q_k\in \{Q_j\}_{j=1}^\fz\backslash\{Q_{j_v}\}_{v=1}^N$,
we may assume that $j_m<k<j_{m+1}$ for some $m\in\nn$. Then, by the proof of (i), we find that
there exists a dyadic cube $Q_k^\ast$ such that $Q_k\subset Q_k^\ast$ and
\begin{equation*}
\sum_{Q_{j_v}\subset Q_k^\ast,v\le  m} w _{\mathcal H^\delta_\infty}(Q_{j_v})+ w _{\mathcal H^\delta_\infty}(Q_k)
> 2 w _{\mathcal H^\delta_\infty}(Q^*_k).
\end{equation*}
Thus,
\begin{equation*}
w _{\mathcal H^\delta_\infty}(Q^*_k)\le \sum_{Q_{j_v}\subset Q_k^\ast,v\le  m} w _{\mathcal H^\delta_\infty}(Q_{j_v})
\le \sum_{Q_{j_v}\subset Q_k^\ast} w _{\mathcal H^\delta_\infty}(Q_{j_v}).
\end{equation*}

Consider the maximal dyadic cubes of the family $\{Q_k^\ast\}_{k\in\nn\backslash \{j_v\}_{v=1}^N}$,
denoted by $\{\widetilde{Q}_k^\ast\}_{k}$.
Then
\begin{equation}\label{eq12-4-a}
\bigcup_{j\in\nn} Q_j\subset \lf(\bigcup_{v=1}^NQ_{j_v}\r)
\bigcup\lf(\bigcup_{k}\widetilde{Q}_k^\ast\r);
\end{equation}
moreover, by the maximality of $\{\widetilde{Q}_k^\ast\}_k$ and the properties of dyadic cubes, we know that
\begin{equation}\label{eq12-5-a}
\sum_{k} w _{\mathcal H^\delta_\infty}(\widetilde{Q}_k^\ast)
\le \sum_{k}\sum_{Q_{j_v}\subset \widetilde{Q}_k^\ast} w _{\mathcal H^\delta_\infty}(Q_{j_v})
\le \sum_{v=1}^N w _{\mathcal H^\delta_\infty}({Q}_{j_v}).
\end{equation}
From this and \eqref{eq12-4-a}, we infer that
\[\begin{aligned}
w _{\mathcal H^\delta_\infty}\lf(\bigcup_{j\in\nn} Q_j\r)
&\leq\sum_{v=1}^N w _{\mathcal H^\delta_\infty}(Q_{j_v})
+\sum_{k} w _{\mathcal H^\delta_\infty}(\widetilde{Q}_k^\ast)
\leq2\sum_{v=1}^N w _{\mathcal H^\delta_\infty}(Q_{j_v}).
\end{aligned}\]
Therefore, (ii) is proved.

For (iii), we observe from the selection of $\{Q_{j_v}\}_{v=1}^N$ and $\{\widetilde{Q}_k^\ast\}_{k}$ that, for any $j\in\nn$, either there exists a cube $Q_{j_v}$ such that $Q_j=Q_{j_v}$ or there exists a cube $\widetilde{Q}_k^\ast$ such that $Q_j\subset \widetilde{Q}_k^\ast$. Therefore for $t\in [1,\infty)$, we have
\[\bigcup_{j\in\nn} tQ_j\subset \lf(\bigcup_{v=1}^NtQ_{jv}\r)
\bigcup\lf(\bigcup_{k}t\widetilde{Q}_k^\ast\r).\]
Using this fact, and applying Remark \ref{12-4}(ii), Lemma \ref{lm-0919-1} and \eqref{eq12-5-a}, we have
\[\begin{aligned}
 w _{\mathcal H^{\delta}_\infty}\lf(\bigcup_{j\in \nn}P_j\r)&\leq w _{\mathcal H^{\delta}_\infty}\lf(\bigcup_{j\in\nn}tQ_j\r)\leq\sum_{v=1}^{N} w _{\mathcal H^{\delta}_\infty}(tQ_{j_v})+\sum_k w _{\mathcal H^{\delta}_\infty}(t\widetilde{Q}_k^\ast)\\
&\le 2^p[w]_{\ca_{p,\delta}}t^{p\delta}\lf(\sum_{v=1}^N w _{\mathcal H^{\delta}_\infty}(Q_{j_v})+\sum_k w _{\mathcal H^{\delta}_\infty}(\widetilde{Q}_k^\ast)\r)\\
&\le 2^{p+1}[w]_{\ca_{p,\delta}}t^{p\delta}\sum _{v=1}^N w _{\mathcal H^{\delta}_\infty}(Q_{j_v}).
\end{aligned}\]
This finishes the proof of Lemma \ref{lm-0919-2}.
\end{proof}

The following proposition  plays a pivotal role throughout the paper,
which fills the gap left by the absence of Fubini's theorem in the weighted Choquet integral setting. This can be regarded as a sort of Radon-Nikodym Theorem for weighted Hausdorff content.

\begin{prop}\label{them-0919-4}
Let $\delta\in(0,n]$, $p\in[1,\fz)$ and $ w \in \mathcal A_{p,\delta}$. Then there exists a positive constant $K(p)$ such that,
for any $f\in L_w^1(\rn,\ch_\fz^\delta)$,
\[\frac{1}{4}\int_{\mathbb R^n}|f(x)|w(x)\,d\mathcal H^{\delta}_\infty
\leq\int_\rn |f(x)|\,dw_{\mathcal H^{\delta}_\infty}
\leq K(p)[w]_{\ca_{p,\delta}}^{\frac{1}{p}}\int_{\mathbb R^n}|f(x)|w(x)\,d\mathcal H^{\delta}_\infty.\]
\end{prop}

To show Proposition \ref{them-0919-4}, we begin with the following useful lemma.

\begin{lem}\label{them-0919-3}
Let $\delta\in(0,n]$ and $E\subset\rn$.
Then,
 for any $f\in L^1(E,\ch^{\delta}_{\fz})$,
\[\int_{E}|f(x)|\,d\ch^{\delta}_{\fz}\leq\sum_{k\in\zz}\int_{E_k}|f(x)|\,d\ch^{\delta}_{\fz}
\leq 4\int_{E}|f(x)|\,d\ch^{\delta}_{\fz},\]
where, for any $k\in\zz$,
\begin{equation*}\label{eq5-30}
E_k:=\lf\{x\in E:\ 2^{k-1}<|f(x)|\leq2^k\r\}.
\end{equation*}
\end{lem}

\begin{proof}
Note that, for any $t\in(0,\fz)$,
$$\lf\{x\in E:\ |f(x)|>t\r\}=\bigcup_{k\in\zz}\lf\{x\in E_k:\ |f(x)|>t\r\}\bigcup\lf\{x\in E:\ |f(x)|=\infty\r\}.$$
Since $f\in L^1(E,\ch^{\delta}_{\fz})$, it follows that $\ch^{\delta}_{\fz}(\{x\in E:\ |f(x)|=\infty\})=0$.
Then, by the definition of Choquet integral and countably subadditive of $\ch^{\delta}_{\fz}$, we find that
\[\begin{aligned}
\int_{E}|f(x)|\,d\ch^{\delta}_{\fz}
&=\int^{\infty}_{0}\ch^{\delta}_{\fz}(\{x\in E:\ |f(x)|>t\})\,dt \leq\int^{\infty}_{0}\sum_{k\in \zz}\ch^{\delta}_{\fz}\lf(\{x\in E_k:|f(x)|>t\}\r)\,dt\\
&=\sum_{k\in \zz}\int^{\infty}_{0}\ch^{\delta}_{\fz}(\{x\in E_k:|f(x)|>t\})\,dt
=\sum_{k\in \zz}\int_{E_k}|f(x)|\,d\ch^{\delta}_{\fz}.
\end{aligned}\]
Conversely,
\[\begin{aligned}
\sum_{k\in \zz}\int_{E_k}|f(x)|\,d\ch^{\delta}_{\fz}
&\leq\sum_{k\in \zz}2^k \ch^{\delta}_{\fz}(E_k)
=4\sum_{k\in \zz}\int^{2^{k-1}}_{2^{k-2}}\ch^{\delta}_{\fz}(\{x\in E_k:\ |f(x)|>t\})\,dt\\
&\leq4\sum_{k\in \zz}\int^{2^{k-1}}_{2^{k-2}}\ch^{\delta}_{\fz}(\{x\in E:\ |f(x)|>t\})\,dt=4\int_{E}|f(x)|\,d\ch^{\delta}_{\fz}.
\end{aligned}\]
This proves Lemma \ref{them-0919-3}.
\end{proof}

\begin{rem}\label{rem-48-1}
Similarly, the following more general result holds. Let $\delta\in(0,n]$, $E\subset\rn$ and $f\in L^p(E,\ch^{\delta}_{\fz})$ with $p\in \rr$. Then we have
\[\int_{E}|f(x)|^p\,d\ch^{\delta}_{\fz}\leq\sum_{k\in\zz}\int_{E_k}|f(x)|^p\,d\ch^{\delta}_{\fz}
\leq \frac{2^{2|p|}}{2^{|p|}-1}\int_{E}|f(x)|^p\,d\ch^{\delta}_{\fz},\]
where, for any $k\in\zz$, $E_k:=\{x\in E:\ 2^{k-1}<|f(x)|\leq2^k\}$. Moreover, the corresponding conclusion also holds true with $\ch_\fz^\delta$ replaced by $w_{\ch_\fz^\delta}$ as in \eqref{eq2-21-a}.
\end{rem}

We now show Proposition \ref{them-0919-4}.

\begin{proof}[Proof of Proposition \ref{them-0919-4}]
We prove this proposition by two steps.
For any $j\in\zz$, let $$F_j:=\{x\in \mathbb R^n:\ 2^{j-1}<|f(x)|\leq2^j\}.$$

\emph{Step 1)} We first show that
\begin{equation}\label{eq11-4-1106}
\int_{\mathbb R^n}|f(x)| w (x)\,d\mathcal H^{\delta}_\infty\le\sum_{j\in\zz}\int_{F_j}|f(x)| w (x)\,d\mathcal H^{\delta}_\infty\leq K(p)[w]_{\ca_{p,\delta}}^{\frac{1}{p}}\int_{\mathbb R^n}|f(x)|w(x)\,d\mathcal H^{\delta}_\infty.
\end{equation}
This is a consequence of
\begin{equation}\label{eq11-4-e}\sum_{j\in\zz}\int_{F_j}|f(x)| w (x)\,d\mathcal H^{\delta}_\infty
\le K(p)[w]_{\ca_{p,\delta}}^{\frac{1}{p}}\int_{\mathbb R^n}|f(x)| w (x)\,d\mathcal H^{\delta}_\infty,
\end{equation}
since the first inequality of \eqref{eq11-4-1106} is obvious by the facts that
$$\rn=\bigcup_{j\in\zz}F_j\bigcup\lf\{x\in \rn:\ |f(x)|=\infty\r\}\bigcup\lf\{x\in \rn:\ |f(x)|=0\r\},$$
and $\ch^{\delta}_{\fz}(\{x\in \rn:\ |f(x)|=\infty\})=0$.

Define
$$G_i:=\{x\in \mathbb R^n:2^{i-1}< w (x)\leq2^i\}, \ \ \ \forall i\in \zz$$
and
$$E_k:=\{x\in \mathbb R^n:2^{k-2}<|f(x)|w(x)\leq2^k\},\ \ \ \forall k\in \zz.$$
Then we have
\begin{align}\label{eq11-4-x}
\sum_{k\in\zz}2^k\mathcal H^{\delta}_{\infty}(E_k)
&=8\sum_{k\in\zz}\int_{2^{k-3}}^{2^{k-2}}\ch_\fz^\delta\lf(\{x\in E_k:\ |f(x)|w(x)>t\}\r)\,dt\\
&\le 8
\int_{\mathbb R^n}|f(x)|w(x)\,d\mathcal H^{\delta}_\infty.\noz
\end{align}
Moreover, for any $j\in\zz$, notice that
$$F_j\subset \bigcup_{k\in\zz}\lf(E_k\cap G_{k-j}\r)\bigcup\lf\{x\in\rn: w(x)=0\r\}\bigcup\lf\{x\in \rn: w(x)=\infty\r\}$$
and
\[\ch^{\delta}_{\fz}(\{x\in\rn: w(x)=0\})=0=\ch^{\delta}_{\fz}(\{x\in\rn: w(x)=\infty\}).\]
Thus, we have
\begin{align}\label{eq11-4-y}
\sum_{j\in\zz}\int_{F_j}|f(x)| w (x)\,d\mathcal H^{\delta}_\infty
&\leq\sum_{j\in\zz}2^j w _{\mathcal H^{\delta}_\infty}(F_j)
\leq\sum_{j\in\zz}2^j \sum_{k\in\zz} w _{\mathcal H^{\delta}_\infty}(E_k\cap G_{k-j}) \\
&=\sum_{k\in\zz}\sum_{j\in\zz}\int_{E_k\cap G_{k-j}}2^j  w (x)\,d\mathcal H^{\delta}_\infty\noz\\
&\leq\sum_{k\in\zz}\sum_{j\in\zz}\int_{E_k\cap G_{k-j}}2^j\cdot 2^{k-j}\,d\mathcal H^{\delta}_\infty\noz\\
&=\sum_{k\in\zz}2^k\sum_{j\in\zz}\mathcal H^{\delta}_\infty(E_k\cap G_{j}).\noz
\end{align}
Therefore, to prove \eqref{eq11-4-e}, combining \eqref{eq11-4-x} and \eqref{eq11-4-y}, we only need to verify that, for any $k\in\zz$,
\begin{equation}\label{eq11-4-d}
\sum_{j\in\zz}\mathcal H^{\delta}_\infty(E_k\cap G_{j})\le K(p)[w]_{\ca_{p,\delta}}^{\frac{1}{p}}\mathcal H^{\delta}_\infty(E_k).
\end{equation}

For any $k\in\zz$, by \eqref{eq11-4-x}, we have $\mathcal H^{\delta}_{\infty}(E_k)<\infty$. Without loss of generality, we may assume that  $\mathcal H^{\delta}_{\infty}(E_k)>0$. Then there exists a family $\{Q_{k,m}\}_{m\in\nn}$
of cubes of $\rn$ such that
\begin{equation}\label{eq11-4-c}
E_k\subset\bigcup_{m\in\nn}Q_{k,m}
\end{equation} and
\begin{equation}\label{eq12-5-b}
\sum_{m\in\nn}\mathcal H^{\delta}_{\infty}(Q_{k,m})
\leq 2\mathcal H^{\delta}_{\infty}(E_k).
\end{equation}
Furthermore, there exists a positive constant $K(p)$ such that, for any $k\in\zz$ and $m\in\nn$,
\begin{equation}\label{eq11-4-z}
\sum_{j\in\zz}\mathcal H^{\delta}_{\infty}(Q_{k,m}\cap G_{j})
\leq K(p)[w]_{\ca_{p,\delta}}^{\frac{1}{p}} \mathcal H^{\delta}_{\infty}(Q_{k,m}).
\end{equation}
Indeed, when $w\in \mathcal A_{1,\delta}$, since
\[\sum_{j\in\zz}\int_{Q_{k,m}\cap G_{j}} w (x)\,d\mathcal H^{\delta}_\infty\leq 4\int_{Q_{k,m}} w (x)\,d\mathcal H^{\delta}_\infty\]
due to Proposition \ref{them-0919-3},
it follows from Lemma \ref{lm-0919-3} that \eqref{eq11-4-z} is true for $p=1$.

When $w\in \mathcal A_{p,\delta}$ with $p\in(1,\fz)$, by Proposition \ref{them-0919-3},
we first have, for any $k\in\zz$ and $m\in\nn$,
\begin{equation}\label{eq11-4-a}
\sum_{j\in\zz}2^{j}\mathcal H^{\delta}_\infty(Q_{k,m}\cap G_{j})
\leq 2\sum_{j\in\zz}\int_{Q_{k,m}\cap G_{j}}w(x)\,d\mathcal H^{\delta}_\infty
\leq 8\int_{Q_{k,m}} w (x)\,d\mathcal H^{\delta}_\infty
\end{equation}
and, applying Remark \ref{rem-48-1}, we find that
\begin{equation*}
\sum_{j\in\zz}2^{-\frac{j}{p-1}}\mathcal H^{\delta}_\infty(Q_{k,m}\cap G_{j})
\leq \sum_{j\in\zz}\int_{Q_{k,m}\cap G_{j}}w(x)^{-\frac{1}{p-1}}\,d\mathcal H^{\delta}_\infty\leq \frac{2^{\frac{2}{p-1}}}{2^{\frac{1}{p-1}}-1}
\int_{Q_{k,m}} w (x)^{-\frac{1}{p-1}}\,d\mathcal H^{\delta}_\infty.
\end{equation*}
Then, from this, H\"older's inequality and \eqref{eq11-4-a}, we deduce that,
for any $k\in\zz$ and $m\in\nn$,
\[\begin{aligned}
&\sum_{j\in\zz}\mathcal H^{\delta}_{\infty}(Q_{k,m}\cap G_{j})\\
&\hs\leq \lf[\sum_{j\in\zz}2^{j}\mathcal H^{\delta}_\infty(Q_{k,m}\cap G_{j})\r]^{\frac{1}{p}}
\lf[\sum_{j\in\zz}2^{-\frac{j}{p-1}}\mathcal H^{\delta}_\infty(Q_{k,m}\cap G_{j})\r]^{\frac{1}{p'}}\\
&\hs\leq K(p)\lf[\int_{Q_{k,m}} w (x)\,d\mathcal H^{\delta}_\infty\r]^{\frac{1}{p}}
\lf[\int_{Q_{k,m}} w(x)^{-\frac{1}{p-1}}\,d\mathcal H^{\delta}_\infty\r]^{\frac{1}{p'}}\\
&\hs\leq K(p)[w]_{\ca_{p,\delta}}^{\frac{1}{p}}\mathcal H^{\delta}_{\infty}(Q_{k,m}),
\end{aligned}\]
which means that \eqref{eq11-4-z} also holds true for $p\in(1,\fz)$.

Combining \eqref{eq11-4-c}, \eqref{eq11-4-z} and \eqref{eq12-5-b}, we further conclude that, for any $k\in\zz$,
\[
\begin{aligned}
\sum_{j\in\zz}\mathcal H^{\delta}_\infty(E_k\cap G_{j})
&\le  \sum_{m\in\nn}\sum_{j\in\zz}\mathcal H^{\delta}_\infty({Q_{k,m}\cap G_{j}})\\
&\le K(p)[w]_{\ca_{p,\delta}}^{\frac{1}{p}}\sum_{m\in\nn}\mathcal H^{\delta}_{\infty}(Q_{k,m})
\le K(p)[w]_{\ca_{p,\delta}}^{\frac{1}{p}} \mathcal H^{\delta}_{\infty}(E_k).
\end{aligned}
\]
This is just \eqref{eq11-4-d}. Therefore, \eqref{eq11-4-e} holds and hence \eqref{eq11-4-1106} is true.

\emph{Step 2)} We show that
\begin{equation}\label{eq12-5-c}
\sum_{j\in\zz}\int_{F_j}|f(x)| w (x)\,d\mathcal H^{\delta}_\infty
\le 4\int_{\rn}|f(x)|\,dw _{\mathcal H^{\delta}_\infty}
\le 8\sum_{j\in\zz}\int_{F_j}|f(x)| w (x)\,d\mathcal H^{\delta}_\infty.
\end{equation}
Obviously, we have
$$\sum_{j\in\zz}\int_{F_j}|f(x)| w (x)\,d\mathcal H^{\delta}_\infty
\leq\sum_{j\in\zz}2^j w _{\mathcal H^{\delta}_\infty}(F_j)\leq 2\sum_{j\in\zz}\int_{F_j}|f(x)| w (x)\,d\mathcal H^{\delta}_\infty.$$
Thus, to verify \eqref{eq12-5-c}, we only need to prove that
\begin{equation}\label{eq12-5-d}
\int_{\rn}|f(x)|\,dw _{\mathcal H^{\delta}_\infty}\leq \sum_{j\in\zz}2^j w _{\mathcal H^{\delta}_\infty}(F_j)\leq 4\int_{\rn}|f(x)|\,dw _{\mathcal H^{\delta}_\infty}.
\end{equation}
Since
\[w_{\ch^{\delta}_{\fz}}(\{x\in \rn:\ |f(x)|=\infty\})=0,\]
it follows that
\[\begin{aligned}
\int_{\rn}|f(x)|\,dw _{\mathcal H^{\delta}_\infty}\leq\sum_{j\in\zz} \int_{F_j}|f(x)|\,dw _{\mathcal H^{\delta}_\infty}\leq \sum_{j\in\zz}2^j w _{\mathcal H^{\delta}_\infty}(F_j).
\end{aligned}\]
On the other hand, it is easy to know that
\[\begin{aligned}
\sum_{j\in\zz}2^j w _{\mathcal H^{\delta}_\infty}(F_j)\leq4\sum_{j\in\zz}\int^{2^{j-1}}_{2^{j-2}} w _{\mathcal H^{\delta}_\infty}(\{x\in F_j:|f(x)|>t\})\,dt
\leq4\int_{\rn}|f(x)|\,dw _{\mathcal H^{\delta}_\infty},
\end{aligned}\]
which implies \eqref{eq12-5-d}. This completes the proof of Proposition \ref{them-0919-4}.
\end{proof}

\begin{rem}\label{rem-0919-1}
We point out that, in the case of $\delta=n$, it is easy to find that
\begin{equation}\label{eq416-1}
\int_\rn|f(x)|\,dw_{\ch_\fz^\delta}\le K\int_\rn |f(x)|w(x)\,d\ch_\fz^\delta
\end{equation}
for any non-negative measurable function $w$ and any measurable function $f$, where $K$ is a positive constant independent of $f$ and $w$.
However, when $\delta\in(0,n)$, it is impossible to find a positive constant $K$ such that
\eqref{eq416-1} holds uniformly even for any Hausdorff content $A_p$ weight $w$.
\begin{enumerate}
\item[(i)] Here we construct a counterexample to understand this.

\textbf{Counterexample 1}. Let $\{E_j\}_{j=1}^\fz$ be a sequence of non-overlapping left-open and right-closed cubes of $\rn$
 with side length $1$ satisfying $\bigcup_{j=1}^\fz E_j=\rn$. For any $m\in\nn$, divide each $E_j$ into $m^n$ congruent left-open and right-closed subcubes $E^j_{m, k}$ $(k=1,2,\dots, m^n)$ with side length $\frac{1}{m}$. For any $x\in\rn$, let
\[f_m(x):=\sum_{k=1}^{m^n}2^{k-1}\mathbf{1}_{E^1_{m,k}}(x)\quad{\rm and}\quad w _m(x):=\sum_{j\in\nn}\sum_{k=1}^{m^n}2^{-(k-1)}\mathbf{1}_{E^j_{m,k}}(x).\]
Then
$$\int_{\rn}f_m(x) w_m(x)\,d\ch^{\delta}_{\infty}=\int_{\rn}\mathbf{1}_{E_1}(x)\,d\ch^{\delta}_{\infty}=\ch_{\fz}^{\dz}(E_1)=1$$
and, for any $m\in\nn$, it is easy to see that $2^{1-m^n}\leq w_m(x)\leq 1$.
Consequently, $w_m\in \ca_{p,\delta}$ for all $p\in[1,\fz)$ and $[w_m]_{\ca_{p,\dz}}\leq 2^{m^n-1}$. Now, for any $k\in\zz$, let
$$F_{m, k}:=\lf\{x\in \mathbb R^n:2^{k-1}<|f_m(x)|\leq2^k\r\}.$$
Then by Remark $\ref{them-0919-3}$, we find that
\[\begin{aligned}
\int_{\rn}|f_m(x)|\,d(w_m)_{\mathcal H^{\delta}_\infty}&\geq \frac{1}{4}\sum_{k\in \zz}\int_{F_{m,k}}|f_m(x)|\,d(w_m)_{\mathcal H^{\delta}_\infty}\\
&\geq\frac{1}{8}\sum_{k=0}^{m^n-1}2^{k}(w_{m})_{\ch_{\infty}^{\delta}}(F_{m, k})=\frac{1}{8}\sum_{k=0}^{m^n-1}\ch_{\infty}^{\delta}(E^1_{m, k+1})\\ &=\frac{1}{8}\sum_{k=0}^{m^n-1}(\frac{1}{m})^{\delta}=\frac{1}{8}m^{n-\delta},
\end{aligned}\]
which tends to infinity as $m\to\infty$ due to $0<\dz<n$. Hence, there is no constant $K$ satisfying \eqref{eq416-1}
for $\delta\in(0,n)$.

\item[(ii)] The reason why \eqref{eq416-1} does not hold uniformly for weights $w$ in the case $\delta\in(0,n)$
seems to be the lack of Fubini's theorem in the framework of Choquet integrals with respect to Hausdorff content. To better understand this, we construct the following counterexample.

\textbf{Counterexample 2}. For $n=1$, let $\delta\in(0,1)$ and $\beta\in(0,1]$. For any $m\in \nn$, let $E_{m,k}=(\frac{k-1}{m}, \frac{k}{m}]$, $\forall\, k\in \{1, 2,\dots, m\}$, and, for any $x\in\mathbb R$, let
\[f_m(x):=\sum_{k=1}^{m}2^{k-1}\mathbf{1}_{E_{m,k}}(x)\quad{\rm and}\quad  w_m(x):=\sum_{k=1}^{m}2^{-(k-1)\beta}\mathbf{1}_{E_{m,k}}(x).\]
Additionally, we define
$$F_m(x,y):=
\begin{cases}
w_m(x)\mathbf{1}_{\{(x,y)\in \mathbb R^2:\ f_m(x)>y\}}(x,y),\quad &\text{when}\  x\in \mathbb R,\  y\in (0, \infty),\\
 0,\quad &\text{when} \ x\in \mathbb R,\  y\in (-\infty, 0].
\end{cases}
$$
Then
\[\int_{\mathbb R}\lf[\int_{\mathbb R}F_m(x,y)\,d\ch_{\infty}^{\beta}(y)\r]d\ch_{\infty}^{\delta}(x)=1,\]
and
\[\int_{\mathbb R}\lf[\int_{\mathbb R}F_m(x,y)\,d\ch_{\infty}^{\delta}(x)\r]d\ch_{\infty}^{\beta}(y)\geq\frac{2^{\beta}-1}{2^{2\beta}}m^{1-\delta},\]
which tends to infinity as $m\to \infty$. Therefore, the following two Choquet integrals
\[\int_{\mathbb R}\lf[\int_{\mathbb R}F_m(x,y)\,d\ch_{\infty}^{\beta}(y)\r]d\ch_{\infty}^{\delta}(x)
\quad{\rm and}\quad \int_{\mathbb R}\lf[\int_{\mathbb R}F_m(x,y)\,d\ch_{\infty}^{\delta}(x)\r]d\ch_{\infty}^{\beta}(y)\]
are not equal even when a constant multiple is permitted.

We point out that, in the following special case, the order of integrals corresponding to the Hausdorff content $\ch_\fz^\delta$ can be interchanged.
Let $ w\in\mathcal A_{p,\delta}$. Then, by Proposition \ref{them-0919-4}, we have, for any function $f$ on $\rn$,
\[\begin{aligned}
\int_{\mathbb R^n}\int^{\infty}_{0}\mathbf{1}_{\{x\in \mathbb R^n:|f(x)|>t\}}(x)w (x)\,dt\,d\mathcal H^{\delta}_{\infty}&=\int_{\mathbb R^n}|f(x)| w (x)\,d\mathcal H^{\delta}_{\infty}\\
&\sim\int^{\infty}_0 w _{\mathcal H^{\delta}_\infty}(\{x\in \mathbb R^n:|f(x)|>t\})\,dt\\
&\sim\int^{\infty}_0\int_{\mathbb R^n}\mathbf{1}_{\{x\in \mathbb R^n:|f(x)|>t\}}(x) w (x)\,d\mathcal H^{\delta}_{\infty}\,dt.
\end{aligned}\]
However, this holds true obviously in the Lebesgue integral setting for measurable function $f$ and non-negative measurable function $w$, and hence
$\int_{\rn}|f(x)|w(x)\,dx=\int_{\rn}|f(x)| dw$ obviously.

\item[(iii)]
We remark that \eqref{eq416-1} may still not hold even for a fixed $\mathcal H^{\delta}_{\infty}$-weight $w$. To show this, we give a counterexample as follows.

\textbf{Counterexample 3}. Let $\delta \in(0,n-1]$ with $n\geq 2$. For any $k\in \nn$, define
$$E_{k}:=\underbrace{(0,1] \times (0,1]\times \cdots \times (0,1]}_{n-1 }\times \lf(\frac{2^{k-1}-1}{2^{k-1}}, \frac{2^k-1}{2^k}\r]$$
and
$$E:=\underbrace{(0,1] \times (0,1] \times\cdots \times (0,1]}_{n-1}\times (0,1).$$
Then, by Remark \ref{rem-HC}, we have $\ch^{\delta}_{\infty}(E)=1$ and, for any $k\in\nn$, $\ch^{\delta}_{\infty}(E_k)=1$.
Now, let
$$w(x):=
\begin{cases}
\sum_{k\in\nn}2^{-k}\mathbf{1}_{E_k}(x), \quad &\text{when}\  x\in E,\\
 1, \quad &\text{when} \ x\in \rn\backslash E
\end{cases} $$
and
$$ f(x):=\sum_{k\in\nn}2^{k}\mathbf{1}_{E_k}(x),\ \ \  \forall x\in\rn.$$
Then it is not difficult to find that
$$\int_{\rn}f(x) w (x)\,d\ch^{\delta}_{\infty}=\int_{\rn}\mathbf{1}_{E}(x)\,d\ch^{\delta}_{\infty}=1.$$
On the other hand, for any $k\in\zz$, let $$F_{k}:=\left\{x\in \mathbb R^n:\ 2^{k-1}<|f(x)|\leq2^k\right\}.$$
Then $F_k=E_k$ for $k\in\nn$ and, $F_k=\emptyset$ otherwise. Thus, by Remark $\ref{them-0919-3}$, we have
\[\begin{aligned}
&\int_{\rn}|f(x)|\,dw_{\mathcal H^{\delta}_\infty}\geq \frac{1}{4}\sum_{k\in \zz}\int_{F_{k}}|f(x)|\,dw_{\mathcal H^{\delta}_\infty}
=\frac{1}{4}\sum_{k\in\nn}2^{k}w_{\ch_{\infty}^{\delta}}(E_{ k})=\frac{1}{4}\sum_{k\in\nn}1=\infty.
\end{aligned}\]
\item[(iv)] Let $\delta\in(0,n)$ and $w$ be a given $\mathcal H^{\delta}_{\infty}$-weight. Then \eqref{eq416-1} holds for any $f\in L_w^1(\rn,\ch_{\fz}^{\dz})$ \emph{if and only if} there exists a positive constant $K$ such that, for any cube $Q\subset \rn$,
    \begin{equation}\label{eq417-2}
    \int_{Q}\frac{1}{w(x)}\,dw_{\ch_{\infty}^{\delta}}\leq K\ch_{\infty}^{\delta}(Q),
    \end{equation}
    which is also equivalent to that there exists a positive constant $K$ such that for any $E\subset \rn$,
    \begin{equation}\label{eq417-3}
    \sum_{k\in\zz}\ch_{\infty}^{\delta}(E\cap E_k)\leq K\ch_{\infty}^{\delta}(E)
    \end{equation}
where, for any $k\in\zz$, $E_{k}:=\{x\in \mathbb R^n:2^{k-1}< w(x)\le 2^k\}$.
\begin{proof}
We complete the proof by showing \eqref{eq416-1}$\Longrightarrow$\eqref{eq417-2}$\Longrightarrow$\eqref{eq417-3}$\Longrightarrow$\eqref{eq416-1}.
Indeed, if \eqref{eq416-1} holds true, then for any cube $Q\subset \rn$, by letting $f(x):=\frac{1}{w(x)}\mathbf{1}_{Q}(x)$, we obtain \eqref{eq417-2}.
If \eqref{eq417-2} holds true, then by Remark \ref{rem-48-1}, we know that, for any cube $Q\subset\rn$,
\begin{align}\label{eq417-4}
\sum_{k\in\zz}\ch_{\infty}^{\delta}(Q\cap E_k)&=\sum_{k\in\zz}\ch_{\infty}^{\delta}(Q\cap E_{1-k})\leq 2\sum_{k\in\zz}2^{k-1}w_{\ch_{\infty}^{\delta}}(Q\cap E_{1-k})\\
&\leq 2\sum_{k\in\zz}\int_{Q\cap E_{1-k}}\frac{1}{w(x)}\,dw_{\mathcal H^{\delta}_\infty}\leq 8\int_{Q}\frac{1}{w(x)}\,dw_{\ch_{\infty}^{\delta}}\ls\ch_{\infty}^{\delta}(Q)\noz.
\end{align}
For any $E\subset \rn$, without loss of generality, we may assume that  $\mathcal H^{\delta}_{\infty}(E)\in(0,\infty)$. Then there exists a family $\{Q_{j}\}_{j\in\nn}$
of cubes in $\rn$ such that $E\subset\bigcup_{j\in\nn}Q_{j}$ and
\begin{equation*}
\sum_{j\in\nn}\mathcal H^{\delta}_{\infty}(Q_{j})
\leq 2\mathcal H^{\delta}_{\infty}(E).
\end{equation*}
Combining this with \eqref{eq417-4}, we conclude that
\[\sum_{k\in\zz}\ch_{\infty}^{\delta}(E\cap E_k)\leq \sum_{k\in\zz}\sum_{j\in\nn}\ch_{\infty}^{\delta}(Q_j\cap E_k)\ls\sum_{j\in\nn}\mathcal H^{\delta}_{\infty}(Q_{j})
\ls\mathcal H^{\delta}_{\infty}(E).\]
Therefore, \eqref{eq417-3} holds true.
Finally, if \eqref{eq417-3} holds, then by \eqref{eq11-4-d} and an argument similar to that used in the proof of Proposition \ref{them-0919-4}, we obtain \eqref{eq416-1}.
\end{proof}
\end{enumerate}
\end{rem}

\subsection{Sparse Covering Property and Substitute for the Linearity of Integrals}\label{s2.3}

Inspired by the idea from Calder\'on-Zygmund decomposition techniques and  stopping time arguments, in this subsection, we formulate  and prove a ``sparse covering property" in the context of Hausdorff content
 in Proposition \ref{lem-1101-1}, which is of central importance to this paper, and may have its own independent significance and potential applicability to other problems.

Using the sparse covering property, with a weighted packing condition, we then build up Proposition \ref{lm-1101-9}. It serves as a partial substitute for the linearity property, which generally fails to hold for weighted Choquet integrals.

\begin{prop}(Sparse covering property)\label{lem-1101-1}
Let $\delta\in(0, n]$ and $E$ be a subset of $\rn$ satisfying $0<\widetilde{\ch}_\fz^{\delta}(E)<\fz$. Then there exists a subset $F\subset \rn$ and a family $\{Q_j\}_{j\in\nn}$ of non-overlapping dyadic cubes in $\rn$ such that
\begin{enumerate}
\item[{\rm(i)}]
$E\subset (\bigcup_{j\in\nn} Q_j)\cup F$ and $\widetilde{\ch}_\fz^{\delta}(F)=0$;
\item[{\rm(ii)}]
$\sum_{j\in\nn}\widetilde{\ch}_\fz^{\delta}(Q_j)\le 2\widetilde{\ch}_\fz^{\delta}(E)$;
\item[{\rm(iii)}]
for any $j\in\nn$, we have
$\widetilde{\ch}_\fz^{\delta}(Q_j)\le 3\widetilde{\ch}_\fz^{\delta}(Q_j\cap E)$.
\end{enumerate}
\end{prop}

\begin{proof}
We first claim that, for any $m\in\nn$, there exist $\{P_{j_1}\}_{j_1\in\nn}$, $\{P_{j_1,j_2}\}_{j_2\in\nn}, \ldots, \{P_{j_1,j_2,\ldots,j_m}\}_{j_m\in\nn}$ of dyadic cubes in $\rn$ satisfying the following four properties:
\begin{enumerate}
\item[{\rm(i)}]
$$E\subset \lf(\bigcup_{k=1}^{m}\bigcup_{(j_1, j_2,\ldots,j_k)\in \cx_k}P_{j_1,j_2,\ldots,j_k}\r)\bigcup F_m\bigcup G_m ;$$
\item[{\rm(ii)}]
$$\sum_{(j_1, j_2,\ldots,j_m)\in \cy_m}\widetilde{\ch}_\fz^{\delta}(P_{j_1,j_2,\ldots,j_m}\cap E)\leq \lf(\frac{2}{3}\r)^{m}\widetilde{\ch}_\fz^{\delta}(E);$$
\item[{\rm(iii)}]
$$\sum_{k=1}^{m-1}\sum_{(j_1,j_2,\ldots,j_k)\in \cx_k}\widetilde{\ch}_\fz^{\delta}(P_{j_1,j_2,\ldots,j_k})+\sum_{\substack{(j_1, j_2,\dots,j_{m-1})\in \cy_{m-1}\\ j_{m}\in\nn}}
\widetilde{\ch}_\fz^{\delta}(P_{j_1,j_2,\ldots,j_m})\leq 2\widetilde{\ch}_\fz^{\delta}(E);$$
\item[{\rm(iv)}]
For any $1\le k\le m$, $F_k\subset F_{k-1}$ and $G_{k-1}\subset G_{k}$.
\end{enumerate}
Here, indicator sets $\{\cy_k\}_{k=1}^m$, $\{\cx_k\}_{k=1}^m$, $\{\cz_k\}_{k=1}^m$, and subsets $\{F_k\}_{k=1}^m$, $\{G_k\}_{k=1}^m$
 are defined as follows. 
For any $k\in\nn$, let
$$A_k:=\lf\{j_k\in\nn:\ \widetilde{\ch}_\fz^{\delta}(P_{j_1, j_2,\ldots, j_k}\cap E)\ge \frac{1}{3}\widetilde{\ch}_\fz^{\delta}(P_{j_1, j_2,\ldots, j_k})\r\},$$
$$B_k:=\lf\{j_k\in\nn:\ 0<\widetilde{\ch}_\fz^{\delta}(P_{j_1, j_2,\ldots, j_k}\cap E)<\frac{1}{3}\widetilde{\ch}_\fz^{\delta}(P_{j_1, j_2,\ldots, j_k})\r\}$$
and
$$C_k:=\lf\{j_k\in\nn:\ \widetilde{\ch}_\fz^{\delta}(P_{j_1, j_2,\ldots, j_k}\cap E)=0\r\}.$$
Then
$$\cy_k:=\lf\{(j_1,j_2,\ldots,j_k):\ j_1\in B_1, j_2\in B_2,\ldots, j_{k-1}\in B_{k-1}, j_k\in B_k\r\}$$
with $B_0=\emptyset$,
$$\cx_k:=\lf\{(j_1,j_2,\ldots,j_k):\ (j_1,\ldots,j_{k-1})\in \cy_{k-1}, j_k\in A_k\r\}$$and
$$\cz_k:=\lf\{(j_1,j_2,\ldots,j_k):\ (j_1,\ldots,j_{k-1})\in \cy_{k-1}, j_k\in C_k\r\}.$$
In addition, 
$$F_0=\rn,\quad F_k:=\bigcup_{(j_1,j_2,\ldots,j_k)\in\cy_k}P_{j_1, j_2,\ldots, j_k}\cap E$$
and
$$G_0=\emptyset,\quad G_k:=G_{k-1}\bigcup\lf(\bigcup_{(j_1,j_2,\ldots,j_k)\in\cz_k}P_{j_1, j_2,\ldots,j_k}\cap E\r).$$

We now prove this claim by induction on $m$. For $m=1$, according to the definition of the Hausdorff content $\widetilde{\ch}_\fz^{\delta}(E)$,
we know that there exists a family $\{P_{j_1}\}_{j_1\in\nn}$ of dyadic cubes in $\rn$ such that
\begin{equation}\label{eq11-5-a}
E\subset \bigcup_{j_1\in\nn} P_{j_1}\quad {\rm and} \quad \sum_{j_1\in\nn} \widetilde{\ch}_\fz^{\delta}(P_{j_1})\leq 2 \widetilde{\ch}_\fz^{\delta}(E).
\end{equation}
This immediately yields condition (iii) for $m=1$. Moreover, from \eqref{eq11-5-a} and the definitions of $A_1$, $B_1$ and $C_1$, we have
\begin{align}\label{eq914-1}
E=\bigcup_{j_1\in\nn}P_{j_1}\cap E\subset \lf(\bigcup_{j_1\in A_1}P_{j_1}\r)\bigcup\lf(\bigcup_{j_1\in B_1}P_{j_1}\cap E\r)\bigcup\lf(\bigcup_{j_1\in C_1}P_{j_1}\cap E\r).
\end{align}
Observe that $\cx_1=A_1$, $\cy_1=B_1$ and $\cz_1=C_1$. Combining this fact with \eqref{eq914-1} and the definitions of $F_1$ and $G_1$, we readily obtain condition (i) for $m=1$. By \eqref{eq11-5-a} and the definition of $B_1$, we have
\begin{equation}\label{eq1101-1}
\sum_{j_1\in \cy_1} \widetilde{\ch}_\fz^{\delta}(P_{j_1}\cap E)=\sum_{j_1\in B_1} \widetilde{\ch}_\fz^{\delta}(P_{j_1}\cap E)
<\frac{1}{3}\sum_{j_1\in B_1}\widetilde{\ch}_\fz^{\delta}(P_{j_1})\le \frac{2}{3}\widetilde{\ch}_\fz^{\delta}(E),
\end{equation}
which implies condition (ii) holds for $m=1$. Conditions (iv)  hold trivially. Therefore, the claim is valid for $m=1$.

Now assume that the claim holds for $m=i$. We aim to prove that it also holds for $m=i+1$. Indeed, for any $(j_1, j_2,\ldots, j_i)\in\cy_i$, by the definition of the Hausdorff content $\widetilde{\ch}_\fz^{\delta}(P_{j_1, j_2,\ldots, j_i}\cap E)$, there exists a family $\{P_{j_1, j_2,\ldots, j_{i+1}}\}_{j_{i+1}\in\nn}$ of dyadic cubes such that
\begin{equation}\label{eq11-5-c}
P_{j_1, j_2,\ldots, j_i}\cap E \subset \bigcup_{j_{i+1}\in\nn}P_{j_1, j_2,\ldots, j_{i+1}},
\end{equation}
and
\begin{equation}\label{eq914-2}
\sum_{j_{i+1}\in\nn} \widetilde{\ch}_\fz^{\delta}(P_{j_1, j_2,\ldots, j_{i+1}})\leq 2 \widetilde{\ch}_\fz^{\delta}(P_{j_1, j_2,\ldots, j_{i}} \cap E).
\end{equation}
Furthermore, for each $j_{i+1}$, we may assume that $P_{j_1, j_2,\ldots, j_{i+1}}\cap(P_{j_1, j_2,\ldots, j_{i}} \cap E)\neq \emptyset$. Then by \eqref{eq11-5-c},
\begin{align*}
F_{i}
&=\bigcup_{(j_1,j_2,\ldots,j_i)\in\cy_i}P_{j_1, j_2,\ldots, j_i}\cap E
\subset \bigcup_{(j_1,j_2,\ldots,j_i)\in\cy_i} \bigcup_{j_{i+1}\in\nn}P_{j_1, j_2,\ldots, j_{i+1}}\cap E\\
&\subset\lf(\bigcup_{\gfz{(j_1,\ldots,j_{i})\in\cy_{i}}{j_{i+1}\in A_{i+1}}}P_{j_1,\ldots, j_{i+1}}\r)
\bigcup \lf(\bigcup_{\gfz{(j_1,\ldots,j_{i})\in\cy_{i}}{j_{i+1}\in B_{i+1}}}P_{j_1,\ldots, j_{i+1}}\cap E\r)
\bigcup \lf(\bigcup_{\gfz{(j_1,\ldots,j_{i})\in\cz_{i}}{j_{i+1}\in A_{i+1}}}P_{j_1,\ldots, j_{i+1}}\cap E\r).
\end{align*}
Combining this inclusion with condition (i) for $m=i$, together with the definitions of $F_{i+1}$ and $G_{i+1}$, we immediately obtain condition (i) for $m=i+1$. 
Then, by \eqref{eq914-2}, condition (ii) for $m=i$, and the definition of $B_{i+1}$, we find that
\begin{align*}
\sum_{(j_1, j_2,\ldots,j_{i+1})\in \cy_{i+1}}\widetilde{\ch}_\fz^{\delta}(P_{j_1,j_2,\ldots,j_{i+1}}\cap E)&\leq \frac{1}{3}\sum_{(j_1, j_2,\ldots,j_{i+1})\in \cy_{i+1}}\widetilde{\ch}_\fz^{\delta}(P_{j_1,j_2,\ldots,j_{i+1}})\\
&\leq\frac{2}{3}\sum_{(j_1, j_2,\ldots,j_{i})\in \cy_{i}}\widetilde{\ch}_\fz^{\delta}(P_{j_1,j_2,\ldots,j_{i}}\cap E)\leq\lf(\frac{2}{3}\r)^{i+1}\widetilde{\ch}_\fz^{\delta}(E),
\end{align*}
which yields condition (ii) for $m=i+1$.
Again using \eqref{eq914-2}, condition (iii) for $m=i$, and the definition of $B_{i+1}$, we obtain
\begin{align*}
&\sum_{k=1}^{i}\sum_{(j_1,j_2,\ldots,j_k)\in \cx_k}\widetilde{\ch}_\fz^{\delta}(P_{j_1,j_2,\ldots,j_k})+\sum_{\substack{(j_1, j_2,\ldots,j_{i})\in \cy_{i}\\ j_{i+1}\in\nn}}
\widetilde{\ch}_\fz^{\delta}(P_{j_1,j_2,\dots,j_{i+1}})\\
&\quad\leq\sum_{k=1}^{i-1}\sum_{(j_1,\ldots,j_k)\in \cx_k}\widetilde{\ch}_\fz^{\delta}(P_{j_1,\ldots,j_k})+\sum_{\substack{(j_1,\ldots,j_{i-1})\in \cy_{i-1}\\ j_{i}\in A_i}}
\widetilde{\ch}_\fz^{\delta}(P_{j_1,\ldots,j_{i}})+2\sum_{(j_1,\ldots,j_{i})\in \cy_{i}}
\widetilde{\ch}_\fz^{\delta}(P_{j_1,\ldots,j_{i}}\cap E)\\
&\quad\leq\sum_{k=1}^{i-1}\sum_{(j_1,\ldots,j_k)\in \cx_k}\widetilde{\ch}_\fz^{\delta}(P_{j_1,\ldots,j_k})+\sum_{\substack{(j_1,\ldots,j_{i-1})\in \cy_{i-1}\\ j_{i}\in A_i}}
\widetilde{\ch}_\fz^{\delta}(P_{j_1,\ldots,j_{i}})+\frac{2}{3}\sum_{\substack{(j_1,\ldots,j_{i-1})\in \cy_{i-1}\\ j_{i}\in B_i}}
\widetilde{\ch}_\fz^{\delta}(P_{j_1,\ldots,j_{i}})\\
&\quad\leq \sum_{k=1}^{i-1}\sum_{(j_1,j_2,\ldots,j_k)\in \cx_k}\widetilde{\ch}_\fz^{\delta}(P_{j_1,j_2,\ldots,j_k})+\sum_{\substack{(j_1, j_2,\ldots,j_{i-1})\in \cy_{i-1}\\ j_{i}\in\nn}}
\widetilde{\ch}_\fz^{\delta}(P_{j_1,j_2,\dots,j_{i}})\leq 2\widetilde{\ch}_\fz^{\delta}(E),
\end{align*}
which verifies condition (iii) for $m=i+1$. 

We next prove condition (iv) for $m=i+1$. For $1\le k\le i-1$, the inclusion $F_{k}\subset F_{k-1}$ follows directly from the induction hypothesis. It remains only to show $F_{i+1}\subset F_{i}$. 
To this end, we show that, for any given $(j_1, j_2, \ldots, j_{i+1})\in \cy_{i+1}$, $P_{j_1, j_2,\ldots, j_{i+1}}\subset P_{j_1, j_2,\ldots, j_{i}}$. Indeed, since $P_{j_1, j_2,\ldots, j_{i+1}}\cap(P_{j_1, j_2,\ldots, j_{i}} \cap E)\neq \emptyset$, we have $P_{j_1, j_2,\ldots, j_{i+1}}\cap P_{j_1, j_2,\ldots, j_{i}}\neq\emptyset$. By the dyadic cube property, either $P_{j_1, j_2,\ldots, j_{i+1}}\subset P_{j_1, j_2,\ldots, j_{i}}$ or $P_{j_1, j_2,\ldots, j_{i}}\subset P_{j_1, j_2,\ldots, j_{i+1}}$. If $P_{j_1, j_2,\ldots, j_{i}}\subset P_{j_1, j_2,\ldots, j_{i+1}}$, then by \eqref{eq914-2} and the fact that $j_i\in B_i$ whenever $(j_1, j_2, \ldots, j_{i+1})\in \cy_{i+1}$, we have
\begin{align*}
\widetilde{\ch}_\fz^{\delta}(P_{j_1, j_2,\ldots, j_{i}})&\le \widetilde{\ch}_\fz^{\delta}(P_{j_1, j_2,\ldots, j_{i+1}})\le \sum_{j_{i+1}\in\nn}\widetilde{\ch}_\fz^{\delta}(P_{j_1, j_2,\ldots, j_{i+1}})\\
&\le 2\widetilde{\ch}_\fz^{\delta}(P_{j_1, j_2,\ldots, j_{i}}\cap E)<\frac23\widetilde{\ch}_\fz^{\delta}(P_{j_1, j_2,\ldots, j_{i}}),
\end{align*}
which is a contradiction. Therefore $P_{j_1, j_2,\ldots, j_{i+1}}\subset P_{j_1, j_2,\ldots, j_{i}}$.
By this, we obtain
\[F_{i+1}=\bigcup_{(j_1,j_2,\ldots,j_{i+1})\in\cy_{i+1}}P_{j_1, j_2,\ldots, j_{i+1}}\cap E\subset \bigcup_{(j_1,j_2,\ldots,j_{i})\in\cy_{i}}P_{j_1, j_2,\ldots,j_{i}}= F_{i},\]
which establishes condition (iv) for $m=i+1$. 
Thus, the claim also holds for $m=i+1$ and hence the claim is proved.

We now complete the proof with the help of the above claim. Let $\{Q_j\}_{j\in\nn}$ denote the collection of maximal dyadic cubes among $\{P_{j_1, j_2,\ldots, j_m}: m\in\nn, \ (j_1, j_2,\ldots, j_m)\in \cx_m\}$. Define
$$ F^*:=\bigcap_{m\in\nn}F_m\quad {\rm and} \quad G:=\bigcup_{m\in\nn} G_m.$$
Fix $m\in\nn$ and take $(j_1, j_2,\ldots, j_m)\in \cx_m$. Then $j_m\in A_m$ and hence 
$$\widetilde{\ch}_\fz^{\delta}(P_{j_1, j_2,\ldots, j_m}\cap E)\ge \frac{1}{3}\widetilde{\ch}_\fz^{\delta}(P_{j_1, j_2,\ldots, j_m}).$$
Consequently, for each $j\in\nn$, we have $\widetilde{\ch}_\fz^{\delta}(Q_j\cap E)\ge \frac{1}{3}\widetilde{\ch}_\fz^{\delta}(Q_j).$
Moreover, by condition (iii), we know that 
\begin{align*}
\sum_{j\in\nn}\widetilde{\ch}_\fz^{\delta}(Q_j)
&\leq\sum_{k\in\nn}\sum_{(j_1,j_2,\ldots,j_k)\in \cx_k}\widetilde{\ch}_\fz^{\delta}(P_{j_1,j_2,\ldots,j_k})\\
&=\lim_{m\to\fz}\sum_{k=1}^{m}\sum_{(j_1,j_2,\ldots,j_k)\in \cx_k}\widetilde{\ch}_\fz^{\delta}(P_{j_1,j_2,\ldots,j_k})\\
&\leq\lim_{m\to\fz}\lf\{ \sum_{k=1}^{m-1}\sum_{(j_1,j_2,\ldots,j_k)\in \cx_k}\widetilde{\ch}_\fz^{\delta}(P_{j_1,j_2,\ldots,j_k})+\sum_{\substack{(j_1, j_2,\dots,j_{m-1})\in \cy_{m-1}\\ j_{m}\in\nn}}
\widetilde{\ch}_\fz^{\delta}(P_{j_1,j_2,\ldots,j_m})\r\}\\
&\leq 2\widetilde{\ch}_\fz^{\delta}(E).
\end{align*}
For any $m\in \nn$, by condition (ii),
$$\widetilde{\ch}_\fz^{\delta}(F^*)\leq\widetilde{\ch}_\fz^{\delta}(F_m)\leq\sum_{(j_1, j_2,\ldots,j_m)\in \cy_m}\widetilde{\ch}_\fz^{\delta}(P_{j_1,j_2,\ldots,j_m}\cap E)\leq \lf(\frac{2}{3}\r)^{m}\widetilde{\ch}_\fz^{\delta}(E),$$
which forces $\widetilde{\ch}_\fz^{\delta}(F^*)=0$.

On the other hand, for any $m\in\nn$,
$$\widetilde{\ch}_\fz^{\delta}(G_m)\leq\widetilde{\ch}_\fz^{\delta}(G_{m-1})+\sum_{(j_1, j_2,\ldots,j_m)\in \cz_m}\widetilde{\ch}_\fz^{\delta}(P_{j_1,j_2,\ldots,j_m}\cap E)=\widetilde{\ch}_\fz^{\delta}(G_{m-1}),$$
which further implies
$$\widetilde{\ch}_\fz^{\delta}(G_m)\leq\widetilde{\ch}_\fz^{\delta}(G_1)\leq \sum_{j_1\in \cz_1}\widetilde{\ch}_\fz^{\delta}(P_{j_1})=0.$$
Therefore
$$0\leq\widetilde{\ch}_\fz^{\delta}(G)\leq \sum_{m\in\nn}\widetilde{\ch}_\fz^{\delta}(G_m)=0.$$
Set $F:=F^*\cup G$. Then
$$0\le\widetilde{\ch}_\fz^{\delta}(F)\leq \widetilde{\ch}_\fz^{\delta}(F^*)+\widetilde{\ch}_\fz^{\delta}(G)=0,$$
and hence $\widetilde{\ch}_\fz^{\delta}(F)=0$. 

Moreover, $E\subset \lf(\bigcup_{j\in\nn} Q_j\r)\bigcup F.$ To verify this inclusion, take any $x\in E$. If $x\notin F$, then there exists $m_0\in\nn$ such that $x\notin F_{m_0}$ and $x\notin G_m$ for all $m\in \nn$. By condition (i),
$$x\in \lf(\bigcup_{k=1}^{m_0}\bigcup_{(j_1, j_2,\ldots,j_k)\in \cx_k}P_{j_1,j_2,\ldots,j_k}\r)\subset \bigcup_{j\in\nn} Q_j, $$
and hence $E\subset \lf(\bigcup_{j\in\nn} Q_j\r)\bigcup F.$ This finally completes the proof of Proposition \ref{lem-1101-1}.
\end{proof}

\begin{rem}
We point out that the constants $2$ and $3$ in (ii) and (iii) of Proposition \ref{lem-1101-1} can be improved to $\alpha$ and $\beta$ with $1<\alpha<\beta$.
\end{rem}

Applying the ``sparse covering property", we obtain the following conclusion, which realizes the interchange of summation and integration with respect to $\ch_{\fz}^{\dz}$.

\begin{lem}\label{lem-1101-2}
Let $\delta\in(0,n]$, $p\in [1,\fz)$, $E\subset \rn$ and $w\in \ca_{p,\delta}$ satisfy $\int_{E}w(x)\,d\ch_{\fz}^{\dz}<\infty$. Then there exist a family $\{Q_j\}_{j\in\nn}$ of non-overlapping dyadic cubes and a subset $F\subset \rn$ with $\ch_\fz^\delta(F)=0$ such that
$$E\subset\lf(\bigcup_{j\in\nn} Q_j\r) \bigcup F\quad {\rm and} \quad \sum_{j\in\nn} \int_{Q_j}w(x)\,d\ch_{\fz}^{\delta}\leq K(n, \delta, p)[w]_{\ca_{p,\delta}} \int_Ew(x)\,d\ch_{\fz}^{\delta},$$
where $K(n, \delta, p)$ is a positive constant independent of $E$ and $w$.
\end{lem}

\begin{proof}
For any $k\in\zz$, let
$$E_k:=\{x\in E:\ 2^{k-1}<w(x)\le 2^k\}.$$
 Then, applying Lemma \ref{them-0919-3}, we have
\begin{equation}\label{eq1101-2}
\sum_{k\in\zz}2^{k}\ch_{\fz}^{\delta}(E_k)\leq 2\sum_{k\in\zz}\int_{E_k}w(x)\,d\ch_{\fz}^{\delta}\leq 8\int_{E}w(x)\,d\ch_{\fz}^{\delta}
\end{equation}
and hence, for any $k\in \zz$, $\widetilde{\ch}_\fz^{\delta}(E_k)\leq K(n, \delta) \ch_{\fz}^{\delta}(E_k)<\fz$ due to Remark \ref{12-4}(i). By this and Proposition \ref{lem-1101-1}, we know that, for any $k\in\zz$, there exist a subset $F_k\subset \rn$ and a family $\{Q_{k,j}\}_{j\in\nn}$ of
non-overlapping dyadic cubes of $\rn$ such that
\begin{enumerate}
\item[(i)] $E_k\subset (\bigcup_{j\in\nn} Q_{k,j})\cup F_k$, and $\ch_\fz^{\delta}(F_k)=0$;

\item[(ii)] $\sum_{j\in\nn}\ch_\fz^{\delta}(Q_{k,j})\leq 2K(n, \delta)\ch_\fz^{\delta}(E_k)$;

\item[(iii)] for any $j\in\nn$, $\ch_\fz^{\delta}(Q_{k,j})\leq 3K(n, \delta) \ch_\fz^{\delta}(Q_{k,j}\cap E_k)$.
 \end{enumerate}
Therefore, if we let $F:=\bigcup_{k\in\zz}F_k$, then $\ch_\fz^{\delta}(F)=0$ and
$$E=\bigcup_{k\in\zz}E_k\subset \lf(\bigcup_{k\in\zz}\bigcup_{j\in\nn}Q_{k,j}\r)\bigcup F.$$
Moreover, by $w\in\ca_{p,\dz}$, Lemma \ref{lm-0919-1} and \eqref{eq1101-2}, we deduce that
\[\begin{aligned}
\sum_{k\in\zz}\sum_{j\in\nn}\int_{Q_{k,j}}w(x)\,d\ch_{\fz}^{\delta}
&\leq 2^p[w]_{\ca_{p,\delta}}\sum_{k\in\zz}\sum_{j\in\nn} \lf[\frac{\ch_{\fz}^{\delta}(Q_{k,j})}{\ch_{\fz}^{\delta}
(E_k\cap Q_{k,j})}\r]^p\int_{E_k\cap Q_{k,j}}w(x)\,d\ch_{\fz}^{\delta}\\
&\leq K(n, \delta, p)[w]_{\ca_{p,\delta}}\sum_{k\in\zz}\sum_{j\in\nn}2^k\ch_{\fz}^{\delta}\lf(E_k\cap Q_{k,j}\r)\\
&\leq K(n, \delta, p)[w]_{\ca_{p,\delta}}\sum_{k\in\zz}2^k\sum_{j\in\nn}\ch_{\fz}^{\delta}\lf(Q_{k,j}\r)\\
&\leq K(n, \delta, p)[w]_{\ca_{p,\delta}}\sum_{k\in\zz}2^k\ch_{\fz}^{\delta}\lf(E_k\r)\\
&\leq K(n, \delta, p)[w]_{\ca_{p,\delta}}\int_Ew(x)\,d\ch_{\fz}^{\delta}.
\end{aligned}\]
Finally, we finish the proof of Lemma \ref{lem-1101-2} by rearranging $\{Q_{k,j}\}_{k\in\zz,j\in\nn}$ as $\{Q_j\}_{j\in\nn}$.
\end{proof}

Moreover, the following conclusion, with a weighted packing condition for a family $\{Q_j\}_j$ of non-overlapping cubes,  becomes essential. It serves as a partial substitute for the linearity property, which generally fails to hold for weighted Choquet integrals.

\begin{prop}\label{lm-1101-9}
Let $\delta\in(0,n]$, $p\in [1, \fz)$ and $w\in \ca_{p,\delta}$. Let $\{Q_j\}_{j\in\nn}$ be a family of non-overlapping dyadic cubes of $\rn$.
If there exists a constant $\beta\in(0,\fz)$ such that, for each dyadic cube $Q$,
\begin{equation}\label{eq11-11-b}
\sum_{Q_{j}\subset Q}w_{\mathcal H^\delta_\infty}(Q_{j})\le \beta\,w_{\mathcal H^\delta_\infty}(Q),
\end{equation}
then there exists a positive constant $K(n, \delta, p)$ such that, for any $f\in L_w^{1}({\cup_{j\in\nn} Q_{j}}, \ch_{\fz}^{\delta})$,
\begin{equation}\label{eq11-11-a}
\sum_{j\in\nn}\int_{Q_{j}}|f(x)|w(x)\,d\mathcal H^{\delta}_{\infty}\le K(n, \delta, p)\max\{1, \beta\}[w]^{1+\frac{1}{p}}_{\ca_{p,\delta}}\int_{\cup_{j\in\nn} Q_{j}}|f(x)|w(x)\,d\mathcal H^{\delta}_{\infty}.
\end{equation}
\end{prop}

\begin{proof}
To prove \eqref{eq11-11-a}, according to Proposition \ref{them-0919-4}, we only need to show
\begin{equation}\label{eq12-5-y}
\sum_{j\in\nn}\int_{Q_{j}}|f(x)|\,dw_{\mathcal H^{\delta}_{\infty}}\leq K(n, \delta, p)\max\{1, \beta\}[w]_{\ca_{p,\delta}} \int_{\bigcup_{j\in\nn} Q_{j}}|f(x)|\,dw_{\mathcal H^{\delta}_{\infty}}.
\end{equation}
For any given $t\in(0,\fz)$,
applying Lemma \ref{lem-1101-2}, we know that there exist a family $\{P_i\}_{i}$ of non-overlapping dyadic cubes and a subset $F\subset \rn$ with $w_{\mathcal H^{\delta}_{\infty}}(F)=0$ such that
\[\lf\{x\in \bigcup_{j\in\nn} Q_j:\ |f(x)|>t\r\}\subset \lf(\bigcup_iP_i\r)\bigcup F\]
and
\begin{equation}\label{eq11-11-x}
\sum_iw_{\mathcal H^{\delta}_{\infty}}(P_i)
\leq K(n, \delta, p)[w]_{\ca_{p,\delta}}w_{\mathcal H^{\delta}_{\infty}}\lf(\lf\{x\in \bigcup_{j\in\nn} Q_j:\ |f(x)|>t\r\}\r).
\end{equation}
Moreover, we may assume, for any $i$, $(\{x\in \bigcup_{j\in\nn} Q_j:\ |f(x)|>t\}\backslash F)\cap P_i\neq \emptyset$. Split the index set of $\{P_i\}_i$ into three disjoint parts:
\begin{itemize}
\item $A_1$: all $i$ such that $P_i$ contains exactly one $Q_j$;
\item $A_2$: all $i$ such that $P_i$ contains at least two distinct $Q_j$;
\item $A_3$: all $i$ such that no $Q_j$ is contained in $P_i$.
\end{itemize}
For $i\in A_3$, since $P_i\cap Q_j\neq\emptyset$ for some $j$ and $Q_j\not\subset P_i$, dyadic property forces $P_i\subset Q_j$.
Because $\{Q_j\}_j$ are non-overlapping, and each $P_i$, for $i\in A_3$, lies inside exactly one $Q_j$.

Let
\[
\begin{aligned}
B_1&:=\{j\in\nn:\ \exists\, i\in A_1\text{ such that }Q_j\subset P_i\},\\
B_2&:=\{j\in\nn:\ \exists\, i\in A_2\text{ such that }Q_j\subset P_i\},\\
B_3&:=\{j\in\nn:\ Q_j\not\subset P_i\text{ for any }i\}.
\end{aligned}
\]
Since $\{P_i\}_i$ are non-overlapping, these three sets form a disjoint partition of $\mathbb N$.

Denote $E_j:=\{x\in Q_j\setminus F:|f(x)|>t\}$. For any $j\in B_1$, the corresponding $i\in A_1$ is unique and $Q_j\subset P_i$, and by definition of $B_1$ and $A_1$, the assignment $j\mapsto i$ gives a bijection between $B_1$ and $A_1$, hence
\begin{align}\label{eq98-1}
\sum_{j\in B_1}w_{\mathcal H^\delta_\infty}(E_j)\leq\sum_{j\in B_1}w_{\mathcal H^\delta_\infty}(Q_j)
\le \sum_{i\in A_1}w_{\mathcal H^\delta_\infty}(P_i).
\end{align}

For any $j\in B_2$, we have $E_j\subset Q_j$. Since $\{P_i\}_i$ are non-overlapping, each such $j$ satisfies $Q_j\subset P_i$ for exactly one index $i\in A_2$.
We may therefore reorganize the summation over $j\in B_2$ as a double sum over $i\in A_2$ and those $Q_j$ contained in $P_i$:
\begin{align}\label{eq98-2}
\sum_{j\in B_2}w_{\mathcal H^\delta_\infty}(E_j)
\le \sum_{j\in B_2}w_{\mathcal H^\delta_\infty}(Q_j)
=\sum_{i\in A_2}\sum_{Q_j\subset P_i}w_{\mathcal H^\delta_\infty}(Q_j)\le \beta\sum_{i\in A_2}w_{\mathcal H^\delta_\infty}(P_i),
\end{align}
where the last inequality follows from \eqref{eq11-11-b} applied to the dyadic cube $P_i$.

For any $j\in B_3$, any $P_i$ intersecting $E_j$ cannot contain $Q_j$.
By the dyadic cube property, $P_i\cap Q_j\neq\emptyset$ forces $P_i\subsetneq Q_j$.
We claim that $i\notin A_1\cup A_2$.
Indeed, if $i\in A_1\cup A_2$, then there exists some $j_0\in\mathbb N$, $j_0\neq j$, such that $Q_{j_0}\subset P_i$.
This would imply $Q_j\cap Q_{j_0}=Q_{j_0}\neq\emptyset$, contradicting the fact that $\{Q_j\}_{j\in\mathbb N}$ are non-overlapping.
Consequently $i\in A_3$. Since $\{Q_j\}_j$ are non-overlapping, each $P_i$ with $i\in A_3$ lies inside exactly one $Q_j$. Thus
\begin{align}\label{eq98-3}
\sum_{j\in B_3}w_{\mathcal H^\delta_\infty}(E_j)
&=\sum_{j\in B_3}w_{\mathcal H^\delta_\infty}\lf(\bigcup_{\{i:\ E_j\cap P_i\neq\emptyset\}}E_j\cap P_i\r)\\
&\leq\sum_{j\in B_3}\sum_{\substack{i\in A_3\\ P_i\subsetneq Q_j}} w_{\mathcal H^\delta_\infty}(P_i)\le \sum_{i\in A_3}w_{\mathcal H^\delta_\infty}(P_i)\noz.
\end{align}
From \eqref{eq98-1}, \eqref{eq98-2} and \eqref{eq98-3}, we obtain
\begin{align*}
\sum_{j\in\nn} w_{\mathcal H^\delta_\infty}(E_j)
&=\sum_{j\in B_1}w_{\mathcal H^\delta_\infty}(E_j)+\sum_{j\in B_2}w_{\mathcal H^\delta_\infty}(E_j)+\sum_{j\in B_3}w_{\mathcal H^\delta_\infty}(E_j)\\
&\le
\sum_{i\in A_1} w_{\mathcal H^\delta_\infty}(P_i)
+\beta\sum_{i\in A_2} w_{\mathcal H^\delta_\infty}(P_i)
+\sum_{i\in A_3} w_{\mathcal H^\delta_\infty}(P_i)\\
&\le \max\{1,\beta\}\sum_i w_{\mathcal H^\delta_\infty}(P_i).
\end{align*}
Combining this, \eqref{eq11-11-x} and the fact that $w_{\mathcal H^{\delta}_{\infty}}(F)=0$, we further conclude that
\[\begin{aligned}
\sum_{j\in\nn} \int_{Q_j}|f(x)|\,dw_{\mathcal H^{\delta}_{\infty}}
&=\sum_{j\in\nn}\int^{\fz}_{0}w_{\mathcal H^{\delta}_{\infty}}\lf(\lf\{x\in Q_j:\ |f(x)|>t\r\}\r)\,dt\\
&=\int^{\fz}_{0}\sum_{j\in\nn}w_{\mathcal H^{\delta}_{\infty}}\lf(\lf\{x\in Q_j\backslash F:\ |f(x)|>t\r\}\r)\,dt\\
&\leq K(n, \delta, p)\max\{1, \beta\}[w]_{\ca_{p,\delta}}\int^{\fz}_{0}w_{\mathcal H^{\delta}_{\infty}}\lf(\lf\{x\in \bigcup_{j\in\nn}Q_j:\ |f(x)|>t\r\}\r)\,dt\\
&=K(n, \delta, p)\max\{1, \beta\}[w]_{\ca_{p,\delta}}\int_{\bigcup_{j\in\nn} Q_j}|f(x)|\,dw_{\mathcal H^{\delta}_{\infty}},
\end{aligned}\]
which is \eqref{eq12-5-y}. This completes the proof of Proposition \ref{lm-1101-9}.
\end{proof}

\begin{rem}\label{rem917-1}
We point out that, in Proposition \ref{lm-1101-9}, \eqref{eq11-11-a} also implies \eqref{eq11-11-b} by taking  $f=\mathbf{1}_Q$ for any dyadic cube. 
Moreover, if  $w\in \ca_{p,\delta}$ with $p\in[1,\fz)$ and $\delta\in(0,n]$,
then the inequality \eqref{eq11-11-b} holds true if and only if 
\begin{equation}\label{eq922-a}
\sum_{Q_{j}\subset Q}\mathcal H^{\delta}_{\infty}(Q_j)\lesssim \mathcal H^{\delta}_{\infty}(Q),\quad \forall\, {\rm dyadic\ cubes}\ Q.
\end{equation}
Indeed, applying \cite[(6) and (7)]{0923-7}, we obtain \eqref{eq922-a}$\Longrightarrow$\eqref{lm-1101-9}.
By taking $f=\frac{1}{w}\mathbf{1}_Q$ for any dyadic cube $Q$,
we have \eqref{lm-1101-9} $\Longrightarrow$ \eqref{eq922-a}.
In addition, we point out that when $p=1$, according to Lemma \ref{lm-0919-3}, condition \eqref{eq11-11-b} directly implies condition \eqref{eq922-a}.
\end{rem}

\subsection{An Example and Monotonicity of Hausdorff Content \(A_p\) Weight $\mathcal A_{p,\delta}$}\label{s2.4}

In this subsection, we first show an example of Hausdorff content $A_p$ weight function as follows.

\begin{prop}\label{lm-1210-1}
Let $\delta\in(0, n]$. For any given $\alpha\in\rr$, define $w(x):=|x|^\alpha$, $\forall\,x\in\rn$.
\begin{enumerate}
	\item[{\rm(i)}] $w\in \ca_{1,\delta}$ if and only if $\alpha\in(-\delta,0]$, and 
	\begin{align}\label{eq917-w}
		[w]_{\ca_{1,\dz}}\leq \max \lf\{\frac{\dz(4\sqrt{n})^{\dz}}{\alpha+\dz}, 2^{-\alpha}\r\}.
	\end{align}
	
 \item[{\rm(ii)}] $w\in \ca_{p,\delta}$ with $p\in(1,\fz)$ if and only if $\alpha\in(-\delta, \delta(p-1))$, and 
 \begin{align}\label{eq917-q}
 	[w]_{\ca_{p,\dz}}\leq\max\lf\{\frac{(\delta2^{\delta})^p(p-1)^{p-1}(2\sqrt{n})^{\delta p}}{(\alpha+\dz)[\delta(p-1)-\alpha]^{p-1}}, 1\r\}.
 \end{align}
\end{enumerate}
\end{prop}

\begin{proof}
For $\alpha=0$, it is obvious that $w\in \ca_{p,\dz}$ and $[w]_{\ca_{p,\dz}}=1$. In what follows, we exclude the case $\alpha=0$. We first prove the sufficiency for $p=1$. To this end, we only need to show that, for any given cube $Q(y,r)$ with $y\in\rn$
and $r\in(0,\fz)$,
\begin{equation}\label{eq12-23a}
\frac{1}{\ch^{\delta}_{\fz}\lf(Q(y, r)\r)}\int_{Q(y, r)}w(z)\,d\ch^{\delta}_{\fz} \leq \max \lf\{\frac{\dz(4\sqrt{n})^{\dz}}{\alpha+\dz}, 2^{-\alpha}\lf(\sqrt{n}\r)^{1-\alpha}\r\}w(x),\quad \forall\,x\in Q(y,r).
\end{equation}

If $|y|\le \frac{3\sqrt{n}}{2}r$, then $Q(y, r)\subset B(0, 2\sqrt{n}r)$ and, for any $z\in Q(y,r)$,
$w(z)=|z|^{\alpha}\ge (2\sqrt{n}r)^{\alpha}$. Thus, by \eqref{eq917-a}, we find that
\begin{align*}
\int_{Q(y, r)}w(z)\,d\ch^{\delta}_{\fz}
&\le\int_{B(0, 2\sqrt{n}r)}w(z)\,d\ch^{\delta}_{\fz}\\
&=\int_0^{\fz}\ch^{\delta}_{\fz}\lf(B\lf(0, 2\sqrt{n}r\r)\cap B\lf(0, t^{\frac{1}{\alpha}}\r)\r)\,dt\noz\\
&\le\int_0^{(2\sqrt{n}r)^{\alpha}}(4\sqrt{n}r)^{\delta}\,dt+
\int_{(2\sqrt{n}r)^{\alpha}}^{\fz}2^{\dz}t^{\frac{\delta}{\alpha}}\,dt\noz\\
&=\frac{\dz2^{\dz}}{\alpha+\dz}\lf(2\sqrt{n}r\r)^{\alpha+\dz} \leq \frac{\dz}{\alpha+\dz}\lf(4\sqrt{n}\r)^{\dz} \ch_\fz^\delta(Q(y,r)) w(x),\quad \forall\,x\in Q(y,r),\noz
\end{align*}
which implies that \eqref{eq12-23a} is true in this case.
If $|y|>\frac{3\sqrt{n}}{2}r$, then, for any $z\in Q(y, r)$, we have
\begin{align*}
|z|\ge |y|-|y-z|\ge \sqrt{n}r
\end{align*}
and
\begin{align*}
|z|\le |y|+|y-z|\le 2\sqrt{n}r.
\end{align*}
Thus
\[\int_{Q(y, r)}w(z)\,d\ch^{\delta}_{\fz}
\leq \int_{Q(y,r)}(\sqrt{n}r)^{\alpha}\,d\ch^{\delta}_{\fz}
=(\sqrt{n}r)^{\alpha}r^\dz
\leq 2^{-\alpha}\ch_\fz^\delta(Q(y,r)) w(x),\quad \forall\,x\in Q(y,r),\]
namely, \eqref{eq12-23a} also holds true in this case. Therefore, $w\in \ca_{1,\delta}$ and \eqref{eq917-w} holds.

Now we prove sufficiency part for $p\in(1,\fz)$.
When $\alpha\in(-\delta, 0)$, then by an argument similar to that used in the proof of the case $p=1$, we find that,
for any cube $Q(y,r)$ with $y\in\rn$ and $r\in(0,\fz)$,
\begin{equation}\label{eq1211-3}
\int_{Q(y, r)}w(z)\,d\ch^{\delta}_{\fz}=\int_{Q(y, r)}|z|^{\alpha}\,d\ch^{\delta}_{\fz}\leq
\left\{\begin{matrix}\frac{\dz2^{\dz}(2\sqrt{n})^{\alpha+\dz}}{\alpha+\dz}r^{\alpha+\dz}, \ \ &\text{when}\  |y|\le \frac{3\sqrt{n}}{2}r,\\
(\sqrt{n}r)^{\alpha}r^\dz, \ \ \ \ \ \ &\text{when}\ |y|> \frac{3\sqrt{n}}{2}r.\end{matrix}\right.
\end{equation}
When $\alpha\in(0,\delta(p-1))$, then $-\frac{\alpha}{p-1}\in (-\delta, 0)$. Therefore, similar to \eqref{eq1211-3}, we have
\[
\lf\{\int_{Q(y, r)}[w(z)]^{-\frac{1}{p-1}}\,d\ch^{\delta}_{\fz}\r\}^{p-1}\leq \left\{\begin{matrix}\frac{[\dz(p-1)2^{\dz}]^{p-1}(2\sqrt{n})^{\dz(p-1)-\alpha}}{[\dz(p-1)-\alpha]^{p-1}}r^{\dz(p-1)-\alpha}, \ \ &\text{when}\  |y|\le \frac{3\sqrt{n}}{2}r,\\
(\sqrt{n}r)^{-\alpha}r^{\dz(p-1)}, \ \ \ \ \ \ &\text{when}\ |y|> \frac{3\sqrt{n}}{2}r.\end{matrix}\right.
\]
Consequently,
\begin{align*}
\int_{Q(y,r)}w(z)\,d\ch_{\fz}^{\delta}
\lf\{\int_{Q(y,r)}[w(z)]^{-\frac{1}{p-1}}\,d\ch_{\fz}^{\delta}\r\}^{p-1}\leq K \lf[\ch_{\fz}^{\delta}(Q(y,r))\r]^p,
\end{align*}
where
\[K=\max\lf\{\frac{(\dz2^{\dz})^p(p-1)^{p-1}(2\sqrt{n})^{\dz p}}{(\alpha+\dz)[\dz(p-1)-\alpha]^{p-1}}, 1\r\},\]
which means $w\in \ca_{p,\delta}$, and \eqref{eq917-q} holds.

We next prove the necessity part for $p=1$.
It is enough to show that, when $\alpha\notin (-\delta, 0]$, $w\notin \ca_{1,\delta}$.
If $\alpha\in (-\fz, -\delta]$, then
\begin{align}\label{eq1211-7}
&\frac{1}{\ch^{\delta}_{\fz}\lf(Q(0, 2)\r)}\int_{Q(0, 2)}w(z)\,d\ch^{\delta}_{\fz}\gs \int_{B(0, 1)}w(z)\,d\ch^{\delta}_{\fz}\\
&=\int_0^{\fz}\ch^{\delta}_{\fz}\lf(B\lf(0, 1\r)\cap B\lf(0, t^{\frac{1}{\alpha}}\r)\r)\,dt\noz\\
&\sim \int_0^{1}1\,dt+
\int_{1}^{+\fz}t^{\frac{\delta}{\alpha}}\,dt=\fz\noz.
\end{align}
Therefore, in this case, $w\notin \ca_{1,\delta}$.
If $\alpha\in (0,\fz)$, then
\[\frac{1}{\ch^{\delta}_{\fz}\lf(Q(0, 1)\r)}\int_{Q(0, 1)}w(z)\,d\ch^{\delta}_{\fz}=\frac{1}{\ch^{\delta}_{\fz}\lf(Q(0, 1)\r)}\int_{Q(0, 1)}|z|^{\alpha}\,d\ch^{\delta}_{\fz}=:\lambda\in(0,1).\]
Hence, for any positive constant $N$, when $z\in B(0,\frac{1}{2})\cap B(0,(\frac{\lambda}{N})^{\frac{1}{\alpha}})$,
\[\frac{1}{\ch^{\delta}_{\fz}\lf(Q(0, 1)\r)}\int_{Q(0, 1)}w(z)\,d\ch^{\delta}_{\fz}> N|z|^\alpha,\]
which implies that, in this case, $w\notin \ca_{1,\delta}$.
Finally, we consider the necessity part for $p\in(1,\fz)$. To this end, we only need to prove that, when
$\alpha\notin (-\delta,\delta(p-1))$, $|x|^{\alpha}\notin \ca_{p,\delta}$.
Indeed, if $\alpha\in (-\fz, -\delta]$, then by \eqref{eq1211-7}, we have
\[\int_{Q(0, 2)}w(z)\,d\ch^{\delta}_{\fz}=\fz,\]
but, obviously,
\[\int_{Q(0, 2)}w(z)^{-\frac{1}{p-1}}\,d\ch_{\fz}^{\delta}>0.\]
Therefore, in this case,
\begin{align}\label{eq1211-8}
\int_{Q(0, 2)}w(z)\,d\ch_{\fz}^{\delta}\lf(\int_{Q(0, 2)}w(z)^{-\frac{1}{p-1}}\,d\ch_{\fz}^{\delta}\r)^{p-1}=\fz.
\end{align}
If $\alpha\in  [\delta(p-1),\fz)$, then $-\frac{\alpha}{p-1}\in (-\fz, -\delta]$ and hence \eqref{eq1211-8} also holds true in this case.
This completes the proof of the necessity part and hence of Proposition \ref{lm-1210-1}.
\end{proof}

From the definition of $\mathcal A_{p,\delta}$, it is easy to verify that the new weight class $\mathcal A_{p,\delta}$ is monotonically increasing with respect to the first parameter $p$, i.e., $\mathcal A_{p_1,\delta}\subset\mathcal A_{p_2,\delta}$ when $1\le p_1\le p_2<\infty$. Then, a natural question is whether the new weight class $\mathcal A_{p,\delta}$ also enjoys monotonicity with respect to the dimensional parameter $\delta$. We end this subsection by showing the strict monotonicity of Hausdorff content $A_p$ weight class $\mathcal A_{p,\delta}$ on the dimension $\delta$ of Hausdorff content.

\begin{prop}\label{rem-0710-1}
Let $0<\dz<\beta\le n$.
\begin{enumerate}
\item[{\rm(i)}] If $p\in [1,\fz)$ and $w\in\ca_{p,\dz}$, then $w^{\frac{\beta}{\dz}}\in \ca_{p, \beta}$ and $[w^{\frac{\beta}{\dz}}]_{\ca_{p,\beta}}\le (\frac{\beta}{\dz})^{p} [w]_{\ca_{p,\dz}}^{\frac{\beta}{\dz}}.$

\item[{\rm(ii)}] If $p\in(1,\fz)$, then $\ca_{p,\dz}\subsetneqq \ca_{\frac{p\dz+\beta-\dz}{\beta},\,\beta}\subsetneqq \ca_{p,\,\beta}.$ Moreover, \[[w]_{\ca_{\frac{p\dz+\beta-\dz}{\beta},\beta}}\le 2\lf(\frac{\beta}{\dz}\r)^{\frac{p\dz}{\beta}} [w]_{\ca_{p,\dz}}\quad {\rm and}\quad [w]_{\ca_{p,\beta}}\le 2^{p-1} [w]_{\ca_{\frac{p\dz+\beta-\dz}{\beta},\beta}}.\]

\item[{\rm(iii)}] If $p=1$, then $\ca_{1,\dz}\subsetneqq \ca_{1,\,\beta}$ and $[w]_{\ca_{1,\beta}}\le 2(\frac{\beta}{\dz})^{\frac{\dz}{\beta}} [w]_{\ca_{1,\dz}}$
\end{enumerate}
\end{prop}

To prove Proposition \ref{rem-0710-1}, we recall the following lemma from \cite[Lemma 2.2]{0923-3}.

\begin{lem}\label{lm918-1}
Let $0<\dz<\beta\le n$. Then, for any $E\subset \rn$ and $f\in L^{1}(\rn, \ch_{\fz}^{\dz})$,
\begin{align*}
\ch_{\fz}^{\beta}(E)\leq \lf(\ch_{\fz}^{\dz}(E)\r)^{\frac{\beta}{\dz}}\quad{\rm{and}} \quad \int_{\rn}|f(x)|^{\frac{\beta}{\dz}}\,d\ch_{\fz}^{\beta}\leq \frac{\beta}{\dz}\lf(\int_{\rn}|f(x)|\,d\ch_{\fz}^{\beta}\r)^{\frac{\beta}{\dz}}.
\end{align*}
\end{lem}

\begin{proof}[Proof of Proposition \ref{rem-0710-1}]
Let $w\in\ca_{p,\dz}$. Then for any compact set $E\subset \rn$ and $\beta\in (\dz, n]$, Lemma \ref{lm918-1} yields that $w^{\frac{\beta}{\dz}}\in(0,\fz)$ $\ch_\fz^\beta$--a. e. on $\rn$  and
\begin{align}\label{eq918-3}
\int_E w(x)^{\frac{\beta}{\dz}}\,d\ch_{\fz}^{\beta}\le \frac{\beta}{\dz}\lf(\int_Ew(x)\,d\ch_{\fz}^{\dz}\r)^{\frac{\beta}{\dz}}<\fz,
\end{align}
which implies that $w^{\frac{\beta}{\dz}}$ is a $\ch_\fz^\delta$-weight on $\rn$. Moreover, by \eqref{eq918-3} and Remark \ref{48-2}(i), one also find that $w$ is a $\ch_\fz^\delta$-weight on $\rn$.

(i) When $p\in(1,\fz)$, from Lemma \ref{lm918-1} and Remark \ref{rem-HC}, we deduce that, for any cube $Q$ of $\rn$,
\begin{align}\label{eq0710-1}
\frac{1}{\ch_{\fz}^{\beta}(Q)}\int_Qw(x)^{\frac{\beta}{\dz}}\,d\ch_{\fz}^{\beta}\leq \frac{\beta}{\dz} \lf(\frac{1}{\ch_{\fz}^{\dz}(Q)}\int_Qw(x)\,d\ch_{\fz}^{\dz}\r)^{\frac{\beta}{\dz}}
\end{align}
and 
\begin{align}\label{eq918-2}
\frac{1}{\ch_{\fz}^{\beta}(Q)}\int_Qw(x)^{-\frac{\beta}{(p-1)\dz}}\,d\ch_{\fz}^{\beta}\le\frac{\beta}{\dz} \lf(\frac{1}{\ch_{\fz}^{\dz}(Q)}\int_Qw(x)^{-\frac{1}{p-1}}\,d\ch_{\fz}^{\dz}\r)^{\frac{\beta}{\dz}}.
\end{align}
By this and \eqref{eq0710-1}, one immediately gets that, $w^{\frac{\beta}{\dz}}\in \ca_{p, \beta}$ and
$[w^{\frac{\beta}{\dz}}]_{\ca_{p,\beta}}\le (\frac{\beta}{\dz})^{p} [w]_{\ca_{p,\dz}}^{\frac{\beta}{\dz}}.$
Similarly, when $p=1$, we also obtain, for $\ch_{\fz}^{\dz}$-almost every $x\in Q$,
\[
\frac{1}{\ch_{\fz}^{\beta}(Q)}\int_Qw(x)^{\frac{\beta}{\dz}}\,d\ch_{\fz}^{\beta}\le \frac{\beta}{\dz}\left(\frac{1}{\ch_{\fz}^{\dz}(Q)}\int_Qw(x)\,d\ch_{\fz}^{\dz}\right)^{\frac{\beta}{\dz}}\leq \frac{\beta}{\dz} [w]_{\ca_{1,\dz}}^{\frac{\beta}{\dz}}w(x)^{\frac{\beta}{\dz}}.
\]
Additionally, for any subset $E$ of $\rn$, the relation $\ch_{\fz}^{\beta}(E)\le \ch_{\fz}^{\dz}(E)^{\frac{\beta}{\dz}}$ holds. It then follows that for $\ch_{\fz}^{\beta}$-almost every $x\in Q$,
\begin{align}\label{eq918-4}
\frac{1}{\ch_{\fz}^{\beta}(Q)}\int_Qw(x)^{\frac{\beta}{\dz}}\,d\ch_{\fz}^{\beta}\le \frac{\beta}{\dz}[w]_{\ca_{1,\dz}}^{\frac{\beta}{\dz}}w(x)^{\frac{\beta}{\dz}},
\end{align}
which implies that $w^{\frac{\beta}{\dz}}\in \ca_{1,\,\beta}$ and $[w^{\frac{\beta}{\dz}}]_{\ca_{1,\beta}}\le \frac{\beta}{\dz} [w]_{\ca_{1,\dz}}^{\frac{\beta}{\dz}}$.

(ii) By \eqref{eq0710-1}, \eqref{eq918-2} and Remark \ref{48-2}(i), we know that
\begin{align*}
&\lf(\frac{1}{\ch_{\fz}^{\beta}(Q)}\int_Qw(x)\,d\ch_{\fz}^{\beta}\r)\lf(\frac{1}{\ch_{\fz}^{\beta}(Q)}\int_Qw(x)^{-\frac{\beta}{(p-1)\dz}}\,d\ch_{\fz}^{\beta}\r)^{\frac{(p-1)\dz}{\beta}}\\
&\quad \le 2\lf(\frac{1}{\ch_{\fz}^{\beta}(Q)}\int_Qw(x)^{\frac{\beta}{\dz}}\,d\ch_{\fz}^{\beta}\r)^{\frac{\dz}{\beta}}\lf(\frac{1}{\ch_{\fz}^{\beta}(Q)}\int_Qw(x)^{-\frac{\beta}{(p-1)\dz}}\,d\ch_{\fz}^{\beta}\r)^{\frac{(p-1)\dz}{\beta}}\noz\\
&\quad\le 2\lf(\frac{\beta}{\dz}\r)^{\frac{p\dz}{\beta}}\lf(\frac{1}{\ch_{\fz}^{\dz}(Q)}\int_Qw(x)\,d\ch_{\fz}^{\dz}\r)\lf(\frac{1}{\ch_{\fz}^{\dz}(Q)}\int_Qw(x)^{-\frac{1}{p-1}}\,d\ch_{\fz}^{\dz}\r)^{p-1}\le 2\lf(\frac{\beta}{\dz}\r)^{\frac{p\dz}{\beta}} [w]_{\ca_{p,\dz}},\noz
\end{align*}
which implies that $w\in\ca_{\frac{p\dz+\beta-\dz}{\beta},\,\beta}$ and $[w]_{\ca_{\frac{p\dz+\beta-\dz}{\beta},\beta}}\le 2(\frac{\beta}{\dz})^{\frac{p\dz}{\beta}} [w]_{\ca_{p,\dz}}$.
In addition, Proposition \ref{lm-1210-1} tells us that $|x|^{\alpha}\in \ca_{p,\dz}$ if and only if $\alpha\in(-\delta, \delta(p-1))$, and $|x|^{\alpha}\in \ca_{\frac{p\dz+\beta-\dz}{\beta},\,\beta}$ if and only if $\alpha\in(-\beta, \delta(p-1))$. Therefore, we choose $\lambda\in (-\beta,-\dz)$, so that $|x|^{\lambda}\in\ca_{\frac{p\dz+\beta-\dz}{\beta},\,\beta}$, but $|x|^{\lambda}\notin\ca_{p,\dz}$. 
Similarly, we have $\ca_{\frac{p\dz+\beta-\dz}{\beta},\,\beta}\subsetneqq \ca_{p,\,\beta}$
 and $[w]_{\ca_{p,\beta}}\le 2^{p-1} [w]_{\ca_{\frac{p\dz+\beta-\dz}{\beta},\beta}}$.

(iii) For any $w\in\ca_{1,\dz}$ and any cube $Q$ of $\rn$, by Remark \ref{48-2}(i) and \eqref{eq918-4}, for $\ch_{\fz}^{\beta}$-almost every $x\in Q$, one obtains
\begin{align*}
\frac{1}{\ch_{\fz}^{\beta}(Q)}\int_Qw(x)\,d\ch_{\fz}^{\beta}\le 2 \lf(\frac{1}{\ch_{\fz}^{\beta}(Q)}\int_Qw(x)^{\frac{\beta}{\dz}}\,d\ch_{\fz}^{\beta}\r)^{\frac{\dz}{\beta}}\le 2\lf(\frac{\beta}{\dz}\r)^{\frac{\dz}{\beta}} [w]_{\ca_{1,\dz}}w(x),
\end{align*}
which implies that $w\in \ca_{1,\,\beta}$ and $[w]_{\ca_{1,\beta}}\le 2(\frac{\beta}{\dz})^{\frac{\dz}{\beta}} [w]_{\ca_{1,\dz}}$. Similar to (i), take $\lambda\in (-\beta,-\dz)$ such that $|x|^{\lambda}\in\ca_{1,\,\beta}$, but $|x|^{\lambda}\notin\ca_{1,\dz}$. This finishes the proof of Proposition \ref{rem-0710-1}.
\end{proof}

As a consequence of Proposition \ref{rem-0710-1}, we obtain a  connection of the new Hausdorff content $A_p$ weights class $\mathcal A_{p,\delta}$ with the classical Muckenhoupt $A_p$ weights.

\begin{cor}\label{cor25-12-14}
Let $\delta\in (0,n)$ and $p\in[1,\fz)$. Then
$$\mathcal A_{p,\dz} \subsetneqq \mathcal A_{p,\,n} \sim A_p,$$
where $\mathcal A_{p,\,n} \sim A_p$ means that $\mathcal A_{p,n}=A_p$ when limiting to Lebesgue measurable weight class.
\end{cor}

By Proposition \ref{rem-0710-1}, Theorem \ref{them-1013-1} and Theorem \ref{them-0919-1}, it is easy to see the following boundedness.

\begin{cor}\label{cor0710-1}
Let $0<\dz<\beta\le n$, $p\in[1,\fz)$ and $w\in\ca_{p,\dz}$.
\begin{enumerate}
\item[(i)] If $p\in(1,\fz)$, then there exists a positive constant $K$ such that, for any $f\in L^{\frac{p\dz+\beta-\dz}{\beta}}_{w}(\rn, \ch_{\fz}^{\beta})$,
    \[\int_{\rn}[\cm_{\ch_{\fz}^{\beta}}f(x)]^{\frac{p\dz+\beta-\dz}{\beta}}w(x)\,d\ch_{\fz}^{\beta}\le K\int_{\rn}|f(x)|^{\frac{p\dz+\beta-\dz}{\beta}}w(x)\,d\ch_{\fz}^{\beta}.\]
\item[(ii)] For $p=1$, then there exists a positive constant $K$ such that, for any $f\in L^{1}_{w}(\rn, \ch_{\fz}^{\beta})$ and $t\in (0,\fz)$,
    \[w_{\ch_{\fz}^{\beta}}\lf(\lf\{x\in \rn: \cm_{\ch_{\fz}^{\beta}}f(x)>t\r\}\r)\le \frac{K}{t}\int_{\rn}|f(x)|w(x)\,d\ch_{\fz}^{\beta}.\]
\end{enumerate}
\end{cor}

\begin{rem}
Observe that, in Corollary \ref{cor0710-1}(i), ${\frac{p\dz+\beta-\dz}{\beta}}<p$. Thus, Corollary \ref{cor0710-1}(i) provides a slight improvement for the boundedness established in Theorem \ref{them-1013-1} in the integrability index $p$, at the cost of a larger dimension exponent $\delta$. This phenomenon can be viewed as a self-improving property on the integrability index of strong type boundedness of Hardy-Littlewood maximal operators over Hausdorff content.
\end{rem}

\section{Proofs of Theorem \ref{them-1013-1}, Theorem \ref{them-0919-1} and Corollary \ref{Cor-0919-1}\label{s3}}

The following interpolation result is an extension of \cite[Lemma 2.5]{0923-3} from $p=1$ to $p\in[1,\fz)$, and we omit the details here. In what follows, let $C$ be a capacity. A property is said to hold $C$-quasieverywhere if the exceptional set has
zero capacity.

\begin{lem}\label{lm-1012-1}
Let $p\in[1,\fz)$ and $C$ be a capacity.
Suppose that $T$ is a quasi-linear operator defined on $L^p(\rn,C)$, that is,
\[|T(f_1+f_2)|\le K\lf[|T(f_1)|+|T(f_2)|\r]\]
for some constant $K$. If there exist $A_1$, $A_2>0$ such that, for any $t\in(0,\fz)$,
\[C(\{x\in \rn: |Tf(x)|>t\})\le \lf(\frac{A_1}{t}\|f\|_{L^p(\rn,C)}\r)^p\]
for all $C$-quasieverywhere defined functions $f\in L^p(\rn,C)$, and
\[\|Tf\|_{L^{\fz}(\rn,C)}\le A_2\|f\|_{L^{\fz}(\rn,C)}\]
for all $C$-quasieverywhere defined functions $f\in L^{\fz}(\rn,C)$, then, for any $q\in(p,\fz)$, we have the estimate
\[\|Tf\|_{L^q(\rn,C)}\le A_3\|f\|_{L^q(\rn,C)}\]
for all $C$-quasieverywhere defined functions $f\in L^{q}(\rn,C)$,
where $A_3$ depending only on $K$, $A_1$, $A_2$, $p$ and $q$. Moreover,
\[A_3:=\lf(\frac{q}{q-p}\r)^{\frac{1}{q}}\lf(2KA_1^{\frac{p}{q}}A_2^{1-\frac{p}{q}}\r).\]
\end{lem}

\begin{lem}\label{lem-1023-x}
Let $\delta\in(0,n]$, $q\in[1,\fz)$ and $w\in\ca_{q,\delta}$.
\begin{enumerate}
\item[{\rm(i)}]
Then exists a positive constant $K=K(n, \delta, q)$ such that, for any $t\in (0,\fz)$,
\[ w_{\mathcal H^{\delta}_\infty}\lf(\lf\{x\in \mathbb R^n:\mathcal M_{w_{\ch_{\fz}^{\delta}}}(f)(x)>t\r\}\r)
\leq \frac{K[w]^3_{\ca_{q,\delta}}}{t}\int_{\mathbb R^n}|f(x)|\,dw_{\ch_{\fz}^{\delta}}.\]
\item[{\rm(ii)}]
If $p\in(1, \fz)$, then there exists a positive constant $K=K(n, \delta, p, q)$ such that
\[\int_{\mathbb R^n}\lf[\mathcal M_{w_{\ch_{\fz}^{\delta}}}(f)(x)\r]^p \,dw_{\ch_{\fz}^{\delta}}\leq K[w]^3_{\ca_{q,\delta}}\int_{\mathbb R^n}|f(x)|^p\,dw_{\ch_{\fz}^{\delta}}.\]
\end{enumerate}
\end{lem}
To prove Lemma \ref{lem-1023-x}, we define the \emph{dyadic Hardy-Littlewood maximal operator associated with a capacity $C$} by setting,
for any $x\in\rn$,
\begin{align}\label{eq1115-3}
\cm_{C}^{\rm d}(f)(x)=\sup_{{\rm dyadic}\ Q\ni x}\frac{1}{C(Q)}\int_{Q}|f(y)|\,dC,
\end{align}
where the supremum is taken over all dyadic cubes $Q$ of $\rn$ containing $x$.

\begin{proof}[Proof of Lemma \ref{lem-1023-x}]
We first prove (i). For any given $t\in(0,\fz)$, we first find that there exists a sequence $\{Q_j\}_{j\in \nn}$ of dyadic cubes in $\rn$ such that
\begin{equation}\label{eq1101-b}
\lf\{x\in \mathbb R^n:\ \mathcal M^{\rm d}_{w_{\ch_{\fz}^{\delta}}}(f)(x)>\frac{t}{2^{n+q}5^{q\delta}[w]_{\ca_{q,\delta}}}\r\}
=\bigcup_{j\in\nn}Q_j,
\end{equation}
and, for any $Q_j$ with $j\in\nn$,
\begin{equation}\label{eq1101-c}
\frac{1}{w_{\ch_{\fz}^{\delta}}(Q_j)}\int_{Q_j}|f(y)|\,dw_{\ch_{\fz}^{\delta}}>\frac{t}{2^{n+q}5^{q\delta}[w]_{\ca_{q,\delta}}}.
\end{equation}
Let $\{Q^\ast_j\}_{j\in \nn}$ be the maximal dyadic cubes of $\{Q_j\}_{j\in \nn}$.
Then
\begin{equation}\label{eq1101-d}
\lf\{x\in \mathbb R^n:\mathcal M^{\rm d}_{w_{\ch_{\fz}^{\delta}}}(f)(x)>\frac{t}{2^{n+q}5^{q\delta}[w]_{\ca_{q,\delta}}}\r\}
=\bigcup_{j\in\nn}Q^\ast_j.
\end{equation}

Fixed $x\in E_t:=\{x\in \mathbb R^n: \mathcal M_{w_{\ch_{\fz}^{\delta}}}(f)(x)>t\}$,
let $Q^x$ be a cube of $\rn$ such that $x\in Q^x$ and
\[\frac{1}{w_{\ch_{\fz}^{\delta}}(Q^x)}\int_{Q^x}|f(y)|\,dw_{\ch_{\fz}^{\delta}}>t.\]
Let $m\in\zz$ be such that $2^{m}\leq l(Q^x)<2^{m+1}$, where $l(Q^x)$ denotes the side length of the cube $Q^x$.
Then there are at most $2^n$ dyadic cubes $\{Q^x_j\}_{j=1}^{N_x}$, with side length $2^{m+1}$ and $1\leq N_x\leq 2^n$,
such that
\[Q^x\subset \bigcup_{j=1}^{N_x}Q^x_j,\]
and, for any $j\in\{1, \dots, N_x\}$,
$Q^x\cap Q^x_j\neq \emptyset$.
Therefore, $x\in 3Q^x_{j}$ and $Q_j^x\subset 5Q^x$ for any $j\in\{1, \dots, N_x\}$. Additionally, applying Lemma \ref{lm-0919-1}, we find that
\[\begin{aligned}
\sum_{j=1}^{N_x}\frac{1}{w_{\ch^{\delta}_{\fz}}(Q^x_j)}\int_{Q^x_j}|f(y)|\,dw_{\ch_{\fz}^{\delta}}
&\ge\frac{1}{w_{\ch^{\delta}_{\fz}}(5Q^x)}\int_{Q^x}|f(y)|\,dw_{\ch_{\fz}^{\delta}}\\
&\ge\frac{1}{2^q5^{q\delta}[w]_{q,\dz}w_{\ch^{\delta}_{\fz}}(Q^x)}\int_{Q^x}|f(y)|\,dw_{\ch_{\fz}^{\delta}}
>\frac{t}{2^q5^{q\delta}[w]_{q,\dz}}.
\end{aligned}\]
Subsequently, there exists a $j_0\in\{1, \dots, N_x\}$ such that
\[\frac{1}{w_{\ch^{\delta}_{\fz}}(Q^x_{j_0})}\int_{Q^x_{j_0}}|f(y)|\,dw_{\ch^{\delta}_{\fz}}
>\frac{t}{2^q5^{q\delta}N_x[w]_{q,\dz}}\geq\frac{t}{2^{n+q}5^{q\delta}[w]_{q,\dz}}.\]
Therefore, $Q^x_{j_0}$ is one of the above dyadic cubes $\{Q_j\}_{j\in\nn}$ as in \eqref{eq1101-b},
which, combined with \eqref{eq1101-d}, implies that, for any $t\in(0,\fz)$,
\begin{equation*}
E_t=\lf\{x\in \mathbb R^n:\mathcal M_{w_{\ch^{\delta}_{\fz}}}(f)(x)>t\r\}\subset\bigcup_{x\in E_t}3Q_{j_0}^x\subset \bigcup_{j\in\nn}3Q^*_j.
\end{equation*}
For any $j\in\nn$, there exists a sequence of dyadic cubes $\{Q_j^{*,i}\}_{i=1}^{3^n}$ such that $l(Q_j^{*,i})=l(Q_j^*)$ and
$$3Q_j^*=\bigcup_{i=1}^{3^n}Q^{*,i}_j.$$
Hence, for given $t$, we have
\begin{equation}\label{eq1101-f}
w_{\mathcal H^{\delta}_\infty}\lf(\lf\{x\in \mathbb R^n:\ \mathcal M_{w_{\ch^{\delta}_{\fz}}}(f)(x)>t\r\}\r)
\le w _{\mathcal H^{\delta}_\infty}\lf(\bigcup_{j\in\nn}3Q^*_j\r)
\leq\sum_{i=1}^{3^n} w _{\mathcal H^{\delta}_\infty}\lf(\bigcup_{j\in\nn}Q_j^{*,i}\r).
\end{equation}

Now applying Lemma \ref{lm-0919-2} to the dyadic cubes $\{Q_j^*\}_{j\in\nn}$, we deduce that
there exists a subfamily $\{Q_{j_v}^*\}_{v}$ such that, for any dyadic cube $Q$,
$$\sum_{Q_{j_v}^*\subset Q} w _{\mathcal H^{\delta}_\infty}(Q_{j_v}^*)\le  2 w _{\mathcal H^{\delta}_\infty}(Q),$$
and, for any $i\in\{1,\dots,3^n\}$, due to $Q_j^{*,i}\subset 3Q_j^*$, we obtain
$$w _{\mathcal H^{\delta}_\infty} \lf(\bigcup_{j\in\nn} Q_j^{*,i}\r)
\leq K(\delta, q)[w]_{\ca_{q,\delta}}\sum_v w _{\mathcal H^{\delta}_\infty}(Q_{j_v}^*).$$
From this, \eqref{eq1101-f}, \eqref{eq1101-c} and \eqref{eq12-5-y}, we deduce that
\[\begin{aligned}
w_{\mathcal H^{\delta}_\infty}\lf(\lf\{x\in \mathbb R^n:\mathcal M_{w_{\ch^{\delta}_{\fz}}}(f)(x)>t\r\}\r)
&\leq K(n, \delta, q)[w]_{\ca_{q,\delta}}\sum_v w _{\mathcal H^{\delta}_\infty}(Q_{j_v}^*)\\
&\leq K(n, \delta, q)\frac{[w]^2_{\ca_{q,\delta}}}{t}\sum_v\int_{Q_{j_v}^*}|f(y)|\,dw_{\ch_{\fz}^{\delta}}\\
&\leq K(n, \delta, q)\frac{[w]^3_{\ca_{q,\delta}}}{t}\int_{\cup_{v} Q_{j_v}^*}|f(y)|\,dw_{\ch_{\fz}^{\delta}}\\
&\leq K(n, \delta, q)\frac{[w]^3_{\ca_{q,\delta}}}{t}\int_{\rn}|f(y)|\,dw_{\ch_{\fz}^{\delta}},
\end{aligned}\]
which proves Lemma \ref{lem-1023-x}(i).

The conclusion (ii) is a consequence of Lemma \ref{lm-1012-1}, Lemma \ref{lem-1023-x}(i) and the fact that
\begin{equation*}
\lf\|\mathcal M_{w_{\ch_{\fz}^{\delta}}}(f)\r\|_{L^{\fz}(\rn,w_{\ch_{\fz}^{\delta}})}
\le\|f\|_{L^{\fz}(\rn,w_{\ch_{\fz}^{\delta}})}.
\end{equation*}
This finishes the proof
of Lemma \ref{lem-1023-x}.
\end{proof}

Before proving Theorem \ref{them-1013-1}, we require the following lemma.

\begin{lem}\label{lm918-2}
Let $\dz\in(0, n]$ and $p\in(1,\fz)$. Then $w\in \ca_{p,\dz}$ if and only if $w^{1-p'}\in \ca_{p', \dz}$, and $[w^{1-p'}]_{\ca_{p',\dz}}=[w]^{\frac{1}{p-1}}_{\ca_{p, \dz}}$.
\end{lem}
\begin{proof}
If $w\in \ca_{p,\dz}$. Then $w$ is a $\ch_\fz^\delta$-weight on $\rn$ and $[w]_{\ca_{p,\dz}}<\fz$. Moreover, $w^{1-p'}\in(0,\fz)$ $\ch_\fz^\delta$-a. e. on $\rn$, and $w^{1-p'}\in L^{1}_{\loc}(\rn, \ch_{\fz}^{\dz})$. Therefore, $w^{1-p'}$ is a $\ch_\fz^\delta$-weight on $\rn$. Furthermore, by \eqref{eq12-12-b}, we infer that, for any cube $Q$ of $\rn$,
\[
\sup_{\text{cube}\,Q\,\subset \rn}\lf\{\frac{1}{\mathcal H^{\delta}_\infty(Q)}\int_Qw(x)^{-\frac{1}{p-1}}\,d\mathcal H^{\delta}_\infty\r\}\lf\{\frac{1}{\mathcal H^{\delta}_\infty(Q)}\int_Qw(x)\,d\mathcal H^{\delta}_\infty\r\}^{\frac{1}{p-1}}=[w]^{\frac{1}{p-1}}_{\ca_{p,\dz}}< \fz,
\]
which implies that $w^{1-p'}\in \ca_{p',\dz}.$  
Conversely, if $w^{1-p'}\in \ca_{p',\dz}$, then a similar argument yields $w\in \ca_{p,\delta}$. This completes the proof of Lemma \ref{lm918-2}.
\end{proof}

Now we show Theorem \ref{them-1013-1}.

\begin{proof}[Proof of Theorem \ref{them-1013-1}]
We complete the proof by showing (ii)$\Longrightarrow$(iii)$\Longrightarrow$(i)$\Longrightarrow$(ii).
The proof of (ii)$\Longrightarrow$(iii) is obvious.

Now we prove (iii)$\Longrightarrow$(i). For any given cube $Q$ of $\rn$,
let $$f(x):=w(x)^{-\frac{1}{p-1}}\mathbf{1}_Q(x),\quad \forall x\in\rn.$$
Then, for any
$$0<t<\frac{1}{\mathcal H^{\delta}_{\infty}(Q)}\int_Qw(x)^{-\frac{1}{p-1}}d\mathcal H^{\delta}_{\infty},$$
we have $Q\subset \{x\in \mathbb R^n:\mathcal M_{\mathcal H^{\delta}_{\infty}} (f)(x)>t\}$.
Thus, applying the weak-type $(p,p)$
inequality, we know
\[\begin{aligned}
w_{\mathcal H^{\delta}_{\infty}}(Q)
&\leq w_{\mathcal H^{\delta}_{\infty}}\lf(\{x\in \mathbb R^n:\mathcal M_{\mathcal H^{\delta}_{\infty}} (f)(x)>t\}\r)\\
&\ls \frac{1}{t^p}\int_{\mathbb R^n}|f(x)|^pw(x)\,d\mathcal H^{\delta}_{\infty}
\sim \frac{1}{t^p}\int_Qw(x)^{-\frac{1}{p-1}}d\mathcal H^{\delta}_{\infty}.
\end{aligned}\]
Letting $t\to \frac{1}{\ch^{\delta}_{\infty}(Q)}\int_Qw(x)^{-\frac{1}{p-1}}d\mathcal H^{\delta}_{\infty}$, we obtain
\[\lf[\frac{1}{\mathcal H^{\delta}_\infty(Q)}\int_Qw(x)\,d\mathcal H^{\delta}_\infty\r]\lf[\frac{1}{\mathcal H^{\delta}_\infty(Q)}\int_Qw(x)^{-\frac{1}{p-1}}\,d\mathcal H^{\delta}_\infty\r]^{p-1}
\ls 1,\]
namely, $w\in \ca_{p,\delta}$.

Next, we show (i)$\Longrightarrow$(ii).  Let $w\in \ca_{p,\delta}$ with $p\in(1,\fz)$. Define
$u:=w^{1-p'}$. Then, by Lemma \ref{lm918-2}, we have $$u\in\ca_{p',\delta}\quad {\rm with}\quad [u]_{\ca_{p',\delta}}=[w]^{\frac{1}{p-1}}_{\ca_{p,\delta}},$$ and, for any cube $Q$ of $\rn$,
\begin{equation}\label{eq-1023-x}
{w_{\ch_\fz^\delta}(Q)^{\frac{1}{p-1}}u_{\ch_\fz^\delta}(Q)}\leq [w]_{\ca_{p,\delta}}^{\frac{1}{p-1}} {\mathcal H^{\delta}_\infty(Q)^{p'}}.
\end{equation}
For any fixed $x\in \rn$ and for any cube $Q$ containing $x$, by \eqref{eq-1023-x} and
Proposition \ref{them-0919-4}, we know that
\begin{align*}
&\frac{1}{\mathcal H^{\delta}_\infty(Q)}\int_Q|f(y)|\,d\mathcal H^{\delta}_\infty\\
&\hs=\frac{w_{\ch_\fz^\delta}(Q)^{\frac{1}{p-1}}u_{\ch_\fz^\delta}(Q)}{\mathcal H^{\delta}_\infty(Q)^{p'}}\lf\{\frac{\mathcal H^{\delta}_\infty(Q)}{w_{\ch_\fz^\delta}(Q)}\lf(\frac{1}{u_{\ch_\fz^\delta}(Q)}\int_Q|f(y)|\,d\mathcal H^{\delta}_\infty\r)^{p-1}\r\}^{\frac{1}{p-1}}\\
&\hs\leq 4[w]^{\frac{1}{p-1}}_{\ca_{p,\delta}}\lf\{\frac{1}{w_{\ch_\fz^\delta}(Q)}
\int_Q\lf[\lf(\frac{1}{u_{\ch_\fz^\delta}(Q)}
\int_{Q}|f(y)|w(y)^{\frac{1}{p-1}}\,du_{\ch_\fz^\delta}\r)^{p-1}\r]\,d\ch_\fz^\delta\r\}^{\frac{1}{p-1}}\\
&\hs\leq 4[w]^{\frac{1}{p-1}}_{\ca_{p,\delta}}\lf\{\frac{1}{w_{\ch_\fz^\delta}(Q)}\lf(\int_Q
\lf[\mathcal M_{u_{\ch_\fz^\delta}}\lf(|f|w^{\frac{1}{p-1}}\r)(z)\r]^{p-1}\,d\mathcal H^{\delta}_\infty\r)\r\}^{\frac{1}{p-1}}\\
&\hs\leq 4^{p'}[w]^{\frac{1}{p-1}}_{\ca_{p,\delta}}\lf\{\frac{1}{w_{\ch_\fz^\delta}(Q)}\lf(\int_Q
\lf[\mathcal M_{u_{\ch_\fz^\delta}}\lf(|f|w^{\frac{1}{p-1}}\r)(z)\r]^{p-1}w(z)^{-1}\,dw_{\ch_\fz^\delta}\r)\r\}^{\frac{1}{p-1}}\\
&\hs\leq 4^{p'}[w]^{\frac{1}{p-1}}_{\ca_{p,\delta}}\lf[\mathcal M_{w_{\ch_\fz^\delta}}
\lf\{\lf[\mathcal M_{u_{\ch_\fz^\delta}}\lf(|f|w^{\frac{1}{p-1}}\r)\r]^{p-1}w^{-1}\r\}(x)\r]^{\frac{1}{p-1}},
\end{align*}
which further implies that
\[\mathcal M_{\mathcal H^{\delta}_{\fz}}(f)(x)
\leq 4^{p'}[w]^{\frac{1}{p-1}}_{\ca_{p,\delta}} \lf[\mathcal M_{w_{\ch_\fz^\delta}}\lf(\lf[\mathcal M_{u_{\ch_\fz^\delta}}\lf(|f|w^{\frac{1}{p-1}}\r)\r]^{p-1}w^{-1}\r)(x)\r]^{\frac{1}{p-1}}.\]
On the other hand, by Lemmas \ref{lem-1023-x} and \ref{lm918-2}, we know that, for any $q\in(1,\fz)$,
\begin{equation*}
\int_{\rn}\lf[\cm_{u_{\ch_\fz^\delta}}(f)(x)\r]^q\,du_{\ch_\fz^\delta}\leq K(n,\delta, p ,q)[w]^{\frac{3}{p-1}}_{\ca_{p,\delta}}\int_{\rn}|f(x)|^q\,du_{\ch_\fz^\delta}
\end{equation*}
and
\[\int_{\rn}\lf[\cm_{w_{\ch_\fz^\delta}}(f)(x)\r]^q\,dw_{\ch_\fz^\delta}\leq K(n,\delta, p ,q)[w]^{3}_{\ca_{p,\delta}}\int_{\rn}|f(x)|^q\,dw_{\ch_\fz^\delta}.\]
From these and Proposition \ref{them-0919-4}, we deduce that
\[\begin{aligned}
\int_{\rn}\lf[\mathcal M_{\mathcal H^{\delta}_{\fz}}(f)(x)\r]^pw(x)\,d\mathcal H^{\delta}_{\fz}
\leq K(n,\delta, p )[w]^{4p'+\frac{2}{p}}_{\ca_{p,\delta}}\int_{\rn}|f(x)|^pw(x)\,d\mathcal H^{\delta}_{\fz},
\end{aligned}\]
which implies that $\mathcal M_{\mathcal H^{\delta}_{\fz}}$ is bounded on $L^p_w(\rn,\ch_\fz^\delta)$.
 This finishes the proof of Theorem \ref{them-1013-1}.
\end{proof}

We next show Theorem \ref{them-0919-1}.

\begin{proof}[Proof of Theorem \ref{them-0919-1}]
To prove (i)$\Longrightarrow$(ii). we show the following weak-type $(p,p)$ inequality: if $w\in \ca_{p,\dz}$ with $p\in [1,\fz)$, then, for any $f\in L_w^p(\rn,\ch_\fz^\delta)$ and any $t\in(0,\fz)$,
\begin{equation}\label{eq1017-3}
w _{\mathcal H^{\delta}_\infty}\lf(\lf\{x\in \mathbb R^n:\mathcal M_{\mathcal H^{\delta}_\infty}(f)(x)>t\r\}\r)
\le K(n, \delta, p)[w]^{3+\frac{1}{p}}_{\ca_{p,\delta}}\frac1{t^p}\int_{\mathbb R^n}|f(x)|^p w (x)\,d\mathcal H^{\delta}_\infty.
\end{equation}
Then by taking $p=1$ we show
(i)$\Longrightarrow$(ii).

For any given $t\in(0,\fz)$, let $\{Q_j\}_{j\in \nn}$ be a sequence of dyadic cubes of $\rn$ such that
\begin{equation*}
\lf\{x\in \mathbb R^n:\ \mathcal M^{\rm d}_{\mathcal H^{\delta}_\infty}(f)(x)>\frac{t}{2^{n+\delta}}\r\}
=\bigcup_{j\in\nn}Q_j,
\end{equation*}
and, for any $Q_j$ with $j\in\nn$,
\begin{equation}\label{eq1017-c}
\frac{1}{\ch^{\delta}_\infty(Q_j)}\int_{Q_j}|f(x)|\,d\ch^{\delta}_\infty>\frac{t}{2^{n+\delta}}.
\end{equation}
Let $\{Q^\ast_j\}_{j\in \nn}$ be the maximal dyadic cubes of $\{Q_j\}_{j\in \nn}$.
Then
\begin{equation*}
\lf\{x\in \mathbb R^n:\mathcal M^{\rm d}_{\mathcal H^{\delta}_\infty}(f)(x)>\frac{t}{2^{n+\delta}}\r\}
=\bigcup_{j\in\nn}Q^\ast_j.
\end{equation*}

Since $w \in \mathcal A_{p,\delta}$, then by \eqref{2610-1} and \eqref{eq1014-3}, we conclude that, for any dyadic cube $Q$ of $\rn$,
\begin{equation}\label{eq1017-b}
\int_{Q}|f(x)|\,d\mathcal H^{\delta}_\infty\leq 2[w]^{\frac{1}{p}}_{\ca_{p,\delta}}\lf(\int_Q|f(x)|^p w (x)\,d\mathcal H^{\delta}_{\infty}\r)^{\frac{1}{p}} \mathcal H^{\delta}_\infty(Q)\lf(w _{\mathcal H^{\delta}_\infty}(Q)\r)^{-\frac{1}{p}}.
\end{equation}
Combining \eqref{eq1017-c} and \eqref{eq1017-b}, we further conclude that, for any $j\in\nn$,
\begin{equation*}
 w _{\mathcal H^{\delta}_\infty}(Q^*_j)
\leq K(n,\delta, p)\frac{[w]_{\ca_{p,\delta}}}{t^p}\int_{Q^*_j}|f(x)|^p w (x)\,d\mathcal H^{\delta}_\infty.
\end{equation*}
Now by this and an argument similar to that used in the proof of Lemma \ref{lem-1023-x}(i), we find that
\begin{equation*}
w _{\mathcal H^{\delta}_\infty}\lf(\lf\{x\in \mathbb R^n:\ \mathcal M_{\mathcal H^{\delta}_\infty}(f)(x)>t\r\}\r)
\le w _{\mathcal H^{\delta}_\infty}\lf(\bigcup_{j\in\nn}3Q^*_j\r),
 \end{equation*}
and hence \eqref{eq1017-3} holds true.

Next we prove (ii)$\Longrightarrow$(i), we need to show that,  for any cube $Q$,
\[\frac{w_{\ch_\fz^\delta}(Q)}{\ch_\fz^\delta(Q)}\leq  Kw(x)\ {\rm for}\ \ch_\fz^\delta{\rm-almost\ every}\ x\in Q,\]
where $K$ is a positive constant independent of $Q$ and $x$.

Let $Q$ be a given cube of $\rn$. For any $E\subset Q$, let $f:=\mathbf{1}_E$. Then, for any $x\in Q$,
\[\mathcal M_{\mathcal H^{\delta}_\infty}(f)(x)\geq\frac{1}{\mathcal H^{\delta}_\infty(Q)}\int_{Q}|f(y)|\,d\mathcal H^{\delta}_{\infty}=\frac{\mathcal H^{\delta}_\infty(E)}{\mathcal H^{\delta}_\infty(Q)},\]
and hence, for any $t\in(0,\frac{\mathcal H^{\delta}_\infty(E)}{\mathcal H^{\delta}_\infty(Q)})$,
\[Q\subset\lf\{x\in \mathbb R^n:\mathcal M_{\ch^{\delta}_\infty}(f)(x)>t\r\}.\]
This, combined with the weak-type $(1,1)$ inequality, implies that
\[\begin{aligned}
 w _{\mathcal H^{\delta}_\infty}(Q)&\leq w _{\mathcal H^{\delta}_\infty}\lf(\lf\{x\in \mathbb R^n:\mathcal M_{\mathcal H^{\delta}_\infty}(f)(x)>t\r\}\r)
\le\frac {K}{t}{ w _{\mathcal H^{\delta}_\infty} (E)}.
\end{aligned}\]
Letting $t\to \frac{\mathcal H^{\delta}_\infty(E)}{\mathcal H^{\delta}_\infty(Q)}$, we obtain
\begin{equation}\label{eq12-5-z}
\frac{\mathcal H^{\delta}_\infty(E)}{\mathcal H^{\delta}_\infty(Q)}\le K\frac{ w _{\mathcal H^{\delta}_\infty}(E)}{ w _{\mathcal H^{\delta}_\infty}(Q)}.
\end{equation}
Now for any $a\in(0,\fz)$, let $E_a:=\{x\in Q: w (x) <a\}$ and
 $$\esinf_{x\in Q} w:=\inf\lf\{{a\in(0,\fz):\ \mathcal H^{\delta}_\infty}\lf(E_a\r)>0\r\}.$$
Then, when $a>{\rm ess\,inf}_{x \in Q} w$, we have $\mathcal H^{\delta}_{\infty}(E_a)>0$, which, together with \eqref{eq12-5-z}, implies that
\[\frac{w _{\mathcal H^{\delta}_\infty}(Q)}{\mathcal H^{\delta}_\infty(Q)}\le K\frac{ w _{\mathcal H^{\delta}_\infty}(E_a)}{\mathcal H^{\delta}_\infty(E_a)}.\]
By this and the fact that
\[w_{\mathcal H^{\delta}_\infty}(E_a)=\int_{E_a}w(x)\,d\mathcal H^{\delta}_{\infty}\leq a\mathcal H^{\delta}_\infty(E_a),\]
we know that $\frac{w_{\mathcal H^{\delta}_\infty}(Q)}{\mathcal H^{\delta}_\infty(Q)}\leq aK$.
Letting $a\to{\rm ess\,inf}_{x \in Q} w$, we further find that
\[\frac{w_{\mathcal H^{\delta}_\infty}(Q)}{\mathcal H^{\delta}_\infty(Q)}
\leq K\,{\rm ess\,inf}_{x \in Q} w, \]
namely, we have
\[\frac{ w _{\mathcal H^{\delta}_\infty}(Q)}{\mathcal H^{\delta}_\infty(Q)}
\leq  Kw(x) \ {\rm for} \ \mathcal H^{\delta}_\infty{\rm-almost\ every}\  x\in Q.\]
This means $w\in \mathcal A_{1,\delta}$.
 This finishes the proof of Theorem \ref{them-0919-1}.
\end{proof}

We end this section by giving the proof of Corollary \ref{Cor-0919-1}. To do this, we need the following
lemma.
\begin{lem}\label{Cor-0919-3}
Let $\delta\in(0, n]$ and $f\in L_{\loc}^1(\rn)$. Then, for any $x\in\rn$,
\[\lf[\mathcal M(|f|^{\frac{n}{\delta}})(x)\r]^{\frac{\delta}{n}}\le \lf(\frac{n}{\delta}\r)^{\frac{\delta}{n}} \mathcal M_{\mathcal H^{\delta}_\infty}(f)(x). \]
\end{lem}

\begin{proof} This is a  consequence of the inequality
\[\int_{\mathbb R^n}|f(x)|dx\leq\frac{n}{\delta}\lf(\int_{\mathbb R^n}|f(x)|^{\frac{\delta}{n}}\,d\mathcal H^{\delta}_\infty\r)^{\frac{n}{\delta}},\]
which comes from \cite[Lemma 3]{0923-7}.
\end{proof}

We finally end this section by proving Corollary \ref{Cor-0919-1}.
\begin{proof}[Proof of Corollary \ref{Cor-0919-1}]
(i) follows immediately from \eqref{eq919-1} and Lemma \ref{Cor-0919-3}. Now we prove (ii). For $w\in \ca_{p,\delta}$ with $p\in [1,\fz)$, by \eqref{eq1017-3} and Proposition \ref{them-0919-4}, we have
\[w _{\mathcal H^{\delta}_\infty}\lf(\lf\{x\in \mathbb R^n:\mathcal M_{\mathcal H^{\delta}_\infty}(f)(x)>t\r\}\r)
\le K_1[w]^{3+\frac{1}{p}}_{\ca_{p,\delta}}\frac1{t^p}\int_{\mathbb R^n}|f(x)|^p \,dw_{\mathcal H^{\delta}_\infty},\quad t\in(0,\fz),
\]
where $K_1$ depends only on $n$, $\dz$, $p$. From this and Lemma \ref{Cor-0919-3}, we obtain
\begin{align}\label{eq919-2}
w_{\mathcal H^{\delta}_\infty}\lf(\lf\{x\in \mathbb R^n:\mathcal M(f)(x)>t\r\}\r)
\le K(n,\delta,p) [w]^{3+\frac{1}{p}}_{\ca_{p,\delta}}\frac{1}{t^{\frac{p\dz}{n}}}\int_{\mathbb R^n}|f(x)|^{\frac{p\dz}{n}} \,dw_{\mathcal H^{\delta}_\infty},\quad t\in(0,\fz).
\end{align}
On the other hand, from Lemma \ref{Cor-0919-3}, it is not difficult to obtain
\[\|\cm (f)\|_{L^{\fz}(\rn, w_{\ch^{\delta}_{\fz}})}\leq\frac{n}{\dz}\|f\|_{L^{\fz}(\rn, w_{\ch^{\delta}_{\fz}})}.\]
By this, \eqref{eq919-2} and Lemma \ref{lm-1012-1}, we conclude that there exists a positive constant $K=K(n,\dz, p, q)$ such that for any $q\in(\frac{p\delta}{n}, \fz)$,
\begin{align*}
\|\cm (f)\|_{L^{q}(\rn, w_{\ch^{\delta}_{\fz}})}\leq K [w]^{\frac{3}{q}+\frac{1}{pq}}_{\ca_{p,\dz}}\|f\|_{L^{q}(\rn, w_{\ch^{\delta}_{\fz}})}.
\end{align*}
This, together with Proposition \ref{them-0919-4}, implies (ii). Finally, (iii) follows immediately from \eqref{eq919-2} with $p=1$ and Proposition \ref{them-0919-4}.
Hence the proof of Corollary \ref{Cor-0919-1} is completed.
\end{proof}

\section{Proofs of Theorem \ref{lm-411-1} and Theorem \ref{Cor-411-1}\label{s4}}

In this section, we first give the proof of the reverse H\"older inequality, Theorem \ref{lm-411-1},
for Hausdorff content $A_p$ weight
class $\ca_{p,\delta}$, which is inspired by the celebrated work of R. R. Coifman and C. Fefferman \cite{CoFe74}. As an application, we then obtain the self-improving property of $\ca_{p,\dz}$; see Theorem \ref{Cor-411-1}.

The following Calder\'on-Zygmund decomposition with respect to the Hausdorff content
$\ch_\fz^\delta$ plays an important role in the proof of Theorem \ref{lm-411-1}.
The Calder\'on-Zygmund decomposition under the additional quasicontinuity assumption was originally proved in \cite{BC25}, while the generalized version without quasicontinuity was later derived in \cite{BC24}, whose techniques are applicable to our setting and yield analogous conclusions. We only need a simple variant of this result, and hence omit its proof.

\begin{prop}\label{lm-410-5}
Let $\delta\in(0, n]$, $Q\subset \rn$ be a given cube and $w\in L^1(Q,\ch_\fz^\delta)$.
Then, for any $\lambda>\frac{1}{\ch_{\infty}^{\delta}(Q)}\int_{Q}w(x)\,d\ch_{\infty}^{\delta}$,
there exists a collection $\{Q_j\}_{j\in\nn}$ of non-overlapping dyadic subcubes of $Q$ such that
\begin{enumerate}
\item[{\rm(i)}]
for any $j\in\nn$, $\lambda<\frac{1}{\ch_{\infty}^{\delta}(Q_j)}\int_{Q_j}w(x)\,d\ch_{\infty}^{\delta}\leq2^{\delta}\lambda$;
\item[{\rm(ii)}]
  for $\ch_{\infty}^{\delta}$-almost every $x\in Q\backslash\bigcup_{j\in\nn}Q_j$, $w(x)\leq 6\Gamma(n, \dz)\lambda$, where
  $\Gamma(n, \dz)$ is a positive constant depending only on $n$ and $\dz$.
\end{enumerate}
\end{prop}

To prove Theorem \ref{lm-411-1}, we still need the following two lemmas.

\begin{lem}\label{lm-410-6}
Let $\delta\in(0, n]$.
If $w\in\ca_{p,\delta}$ with $p\in[1,\infty)$, then, for any $\beta\in(0,[w]^{-1}_{\ca_{p,\delta}})$,
there exists a positive constant $K(p, \beta, [w]_{\ca_{p,\delta}})$ such that, for any cube $Q$ of $\rn$,
\[\ch_{\infty}^{\delta}(\{x\in Q:\ w(x)>\beta W_Q\})\geq K(p, \beta, [w]_{\ca_{p,\delta}})\ch_{\infty}^{\delta}(Q).\]
Here and thereafter, $W_Q:=\frac{1}{\ch_{\infty}^{\delta}(Q)}\int_{Q}w(x)\,d\ch_{\infty}^{\delta}$ for any cube $Q$ of $\rn$.
\end{lem}

\begin{proof}
When $p=1$, since $\beta\in(0,[w]^{-1}_{\ca_{1,\delta}})$, by the definition of $\ca_{1,\delta}$, it is clear that
\[\ch_{\infty}^{\delta}(\{x\in Q:\ w(x)>\beta W_Q\})=\ch_{\infty}^{\delta}(Q).\]
When $p\in(1,\infty)$, let $E:=\{x\in Q:\ w(x)\leq \beta W_Q\}$. Then
\begin{align*}\label{eq410-9}
\frac{1}{\beta}\lf(\frac{\ch_{\infty}^{\delta}(E)}{\ch_{\infty}^{\delta}(Q)}\r)^{p-1}
&=W_Q\lf[\frac{1}{\ch_{\infty}^{\delta}(Q)}\int_{E}(\beta W_Q)^{-\frac{1}{p-1}}\,d\ch_{\infty}^{\delta}\r]^{p-1}\\
&\leq W_Q\lf[\frac{1}{\ch_{\infty}^{\delta}(Q)}\int_{Q}w(x)^{-\frac{1}{p-1}}\,d\ch_{\infty}^{\delta}\r]^{p-1}\leq [w]_{\ca_{p,\delta}}.
\end{align*}
Therefore, we have
\[\ch_{\infty}^{\delta}(\{x\in Q:\ w(x)>\beta W_Q\})\geq\ch_{\infty}^{\delta}(Q)-\ch_{\infty}^{\delta}(E)
\geq\lf(1-(\beta[w]_{\ca_{p,\delta}})^{\frac{1}{p-1}}\r)\ch_{\infty}^{\delta}(Q).\]
This finishes the proof of Lemma \ref{lm-410-6}.
\end{proof}

\begin{lem}\label{lm-410-7}
Let $\delta\in(0, n]$.
If $w\in\ca_{p,\delta}$ with $p\in[1,\infty)$, then, for any  $\beta\in(0,[w]^{-1}_{\ca_{p,\delta}})$,
there exists a positive constant $K(n, \delta, p, \beta, [w]_{\ca_{p,\delta}})$ such that, for any given cube $Q$ of $\rn$
and any $\lambda>W_Q$,
\[w_{\ch_{\infty}^{\delta}}(\{x\in Q:\ w(x)>6\Gamma(n, \dz)\lambda\})\leq K(n, \delta, p, \beta,
[w]_{\ca_{p,\delta}})\lambda\ch_{\infty}^{\delta}(\{x\in Q: w(x)>\beta\lambda\}).\]
\end{lem}

\begin{proof}
By Proposition \ref{lm-410-5}, we know that there exists a collection $\{Q_j\}_{j\in\nn}$ of non-overlapping dyadic subcubes of $Q$ such that
\begin{align*}
w_{\ch_{\infty}^{\delta}}(\{x\in Q:\ w(x)>6\Gamma(n, \dz)\lambda\})\leq w_{\ch_{\infty}^{\delta}}\lf(\bigcup_{j\in\nn}Q_j\r),
\end{align*}
and, for any $j\in\nn$,
\begin{align*}
\lambda<\frac{1}{\ch_{\infty}^{\delta}(Q_j)}\int_{Q_j}w(x)\,d\ch_{\infty}^{\delta}\leq2^{\delta}\lambda.
\end{align*}
Furthermore, although the ${Q_j}$ here is a dyadic subcube of $Q$, the corresponding conclusions of
Lemmas \ref{lm-0919-2} and \ref{lm-1101-9} still hold true. Therefore, by these and Lemma \ref{lm-410-6}, we find that there exists a subfamily $\{Q_{j_v}\}_{v=1}^N$ of  $\{Q_j\}_{j\in\nn}$ as in Lemma \ref{lm-0919-2} such that
\begin{align*}
w_{\ch_{\infty}^{\delta}}(\{x\in Q: w(x)>6\Gamma(n, \dz)\lambda\})
&\leq 2\sum_{v=1}^Nw_{\ch_{\infty}^{\delta}}(Q_{j_v})
\leq K(\delta)\lambda \sum_{v=1}^N\ch_{\infty}^{\delta}(Q_{j_v})\\
&\leq K(\delta, p, \beta, [w]_{\ca_{p,\delta}})\lambda\sum_{v=1}^N\ch_{\infty}^{\delta}(\{x\in Q_{j_v}: w(x)>\beta W_{Q_{j_v}}\})\\
&\leq K(\delta, p, \beta, [w]_{\ca_{p,\delta}})\lambda\sum_{v=1}^N\int_{Q_{j_v}}
\mathbf{1}_{\{x\in Q:\ w(x)>\beta\lambda\}}(x)\,d\ch_{\infty}^{\delta}\\
&\leq K(n, \delta, p, \beta, [w]_{\ca_{p,\delta}})\lambda\ch_{\infty}^{\delta}(\{x\in Q:\ w(x)>\beta\lambda\}).
\end{align*}
This finishes the proof of Lemma \ref{lm-410-7}.
\end{proof}

Now we are ready to prove Theorem \ref{lm-411-1}.

\begin{proof}[Proof of Theorem \ref{lm-411-1}]
For any given cube $Q$ of $\rn$, let
$$W_Q:=\frac{1}{\ch_{\infty}^{\delta}(Q)}\int_{Q}w(x)\,d\ch_{\infty}^{\delta}\quad{\rm and}\quad
E:=\{x\in Q: w(x)> 6\Gamma(n, \dz)W_Q\}.$$
Note that $W_Q$ is finite due to $w\in L^1_{\rm loc}(\rn,\ch_{\infty}^{\delta})$. We first show that there exists a positive constant $\tilde{K}=\tilde{K}(n, \delta, p, [w]_{\ca_{p,\delta}})$ such that,
for any $\tilde{\gamma}\in(0,1)$,
\begin{equation}\label{eq411-3}
\frac{1}{\tilde{\gamma}}\int_E\lf[\lf(\frac{w(x)}{6\Gamma(n, \dz)}\r)^{\tilde{\gamma}}-(W_Q)^{\tilde{\gamma}}\r]w(x)\,d\ch_{\infty}^{\delta}
\leq \tilde{K}\int_{Q}w(x)^{1+\tilde{\gamma}}\,d\ch^{\delta}_{\infty}.
\end{equation}
By Proposition \ref{them-0919-4} and Lemma \ref{lm-410-7}, we have
\begin{align*}
&\frac{1}{\tilde{\gamma}}\int_E\lf[\lf(\frac{w(x)}{6\Gamma(n, \dz)}\r)^{\tilde{\gamma}}-(W_Q)^{\tilde{\gamma}}\r]w(x)\,d\ch_{\infty}^{\delta}\\
&\hs \leq\frac{4}{\tilde{\gamma}}\int_E\lf(\frac{w(x)}{6\Gamma(n, \dz)}\r)^{\tilde{\gamma}}-(W_Q)^{\tilde{\gamma}}\,dw_{\ch_{\infty}^{\delta}}\\
&\hs =\frac{4}{\tilde{\gamma}}\int_{0}^{\infty}w_{\ch_{\infty}^{\delta}}{\lf(\lf\{x\in E: \lf(\frac{w(x)}{6\Gamma(n, \dz)}\r)^{\tilde{\gamma}}-(W_Q)^{\tilde{\gamma}}>t\r\}\r)}\,dt\\
&\hs =4\int_{W_Q}^{\infty}t^{\tilde{\gamma}-1}w_{\ch_{\infty}^{\delta}}{\lf(\lf\{x\in E: \lf(\frac{w(x)}{6\Gamma(n, \dz)}\r)^{\tilde{\gamma}}-(W_Q)^{\tilde{\gamma}}>t^{\tilde{\gamma}}-(W_Q)^{\tilde{\gamma}}\r\}\r)}\,dt\\
&\hs =4\int_{W_Q}^{\infty}t^{\tilde{\gamma}-1}w_{\ch_{\infty}^{\delta}}\lf(\lf\{x\in Q: w(x)>6\Gamma(n, \dz)t\r\}\r)\,dt\\
&\hs \leq \tilde{K}(n, \delta, p, [w]_{\ca_{p,\delta}})\int_{W_Q}^{\infty}t^{\tilde{\gamma}}\ch_{\infty}^{\delta}\lf(\lf\{x\in Q: w(x)>\frac{1}{2[w]_{\ca_{p,\delta}}}t\r\}\r)\,dt\\
&\hs \leq \tilde{K}(n, \delta, p, [w]_{\ca_{p,\delta}})\int_{Q}w(x)^{1+\tilde{\gamma}}\,d\ch^{\delta}_{\infty},
\end{align*}
which means that \eqref{eq411-3} holds true.

Next we prove that
\begin{align}\label{eq-411-4}
\int_{Q}w(x)^{1+\tilde{\gamma}}\,d\ch^{\delta}_{\infty}
&\leq 12\Gamma(n, \dz)\int_E\lf[\lf(\frac{w(x)}{6\Gamma(n, \dz)}\r)^{\tilde{\gamma}}-(W_Q)^{\tilde{\gamma}}\r]w(x)\,d\ch_{\infty}^{\delta}\\
&\quad+\lf(12\Gamma(n, \dz)+36\Gamma(n, \dz)^2\r)(W_Q)^{1+\tilde{\gamma}}\ch_{\infty}^{\delta}(Q).\noz
\end{align}
Indeed, by Remark \ref{12-4}(ii), we have
\begin{align*}
&\int_{Q}w(x)^{1+\tilde{\gamma}}\,d\ch^{\delta}_{\infty}\\
&\hs \leq \int_{Q\backslash E}w(x)^{1+\tilde{\gamma}}\,d\ch^{\delta}_{\infty}+\int_{E}w(x)^{1+\tilde{\gamma}}\,d\ch^{\delta}_{\infty}\\
&\hs \leq \int_{Q\backslash E}w(x)^{1+\tilde{\gamma}}\,d\ch^{\delta}_{\infty}+12\Gamma(n, \dz)\int_E
\lf[\lf(\frac{w(x)}{6\Gamma(n, \dz)}\r)^{\tilde{\gamma}}-(W_Q)^{\tilde{\gamma}}\r]w(x)\,d\ch_{\infty}^{\delta}
\\
&\quad\hs+12\Gamma(n, \dz)\int_E(W_Q)^{\tilde{\gamma}}w(x)\,d\ch_{\infty}^{\delta}\\
&\hs \leq12\Gamma(n, \dz)\int_E\lf[\lf(\frac{w(x)}{6\Gamma(n, \dz)}\r)^{\tilde{\gamma}}-(W_Q)^{\tilde{\gamma}}\r]w(x)
\,d\ch_{\infty}^{\delta}+\lf(12\Gamma(n, \dz)+36\Gamma(n, \dz)^2\r)(W_Q)^{1+\tilde{\gamma}}\ch_{\infty}^{\delta}(Q),
\end{align*}
which implies \eqref{eq-411-4}.

Finally, by \eqref{eq411-3} and \eqref{eq-411-4}, we conclude that
$$(1-12\Gamma(n, \dz)\tilde{K}\tilde{\gamma})\int_Qw(x)^{1+\tilde{\gamma}}\,d\ch_\fz^\delta
\le \lf(12\Gamma(n, \dz)+36\Gamma(n, \dz)^2\r) (W_Q)^{1+\tilde{\gamma}}\ch_\fz^\delta(Q).$$
Choosing $\gamma:=\tilde{\gamma}\in (0,1)$ such that $12\Gamma(n, \dz)\tilde{K}\tilde{\gamma}<1$, then we obtain the reverse H\"older inequality \eqref{eq-411-1}.
This finishes the proof of Theorem \ref{lm-411-1}.
\end{proof}

\begin{rem}
Recall that the cube $Q$ in Theorem \ref{lm-411-1} is a left-open and right-closed cube of $\rn$. But we point out that Theorem \ref{lm-411-1} holds for any cube $P$ of $\rn$, which is not required to be the left-open and right-closed cube.
Indeed, for any cube $P$ of $\rn$, there exists a left-open and right-closed cubes $Q$ with the same center as $P$ and side length twice that of $P$. Then inequality \eqref{eq-411-1} holds for $Q$, i.e., there exist positive constants $\tilde{K}$ and $\tilde{\gamma}$ such that
\[\lf[\frac{1}{\ch_{\infty}^{\delta}(Q)}\int_{Q}w(x)^{1+\tilde{\gamma}}\,d\ch_{\infty}^{\delta}\r]^{\frac{1}{1+\tilde{\gamma}}}\leq \frac{\tilde{K}}{\ch_{\infty}^{\delta}(Q)}\int_{Q}w(x)\,d\ch_{\infty}^{\delta}.\]
By $\ch_{\fz}^{\dz}(Q)=2^{\dz}\ch_{\fz}^{\dz}(P)$ and Lemma \ref{lm-0919-1}, we deduce that $w_{\ch_{\fz}^{\dz}}(Q)\leq 2^{p(\dz+1)}[w]_{\ca_{p,\dz}}w_{\ch_{\fz}^{\dz}}(P).$ Thus, we have
\begin{align*}
\lf[\frac{1}{\ch_{\infty}^{\delta}(P)}\int_{P}w(x)^{1+\tilde{\gamma}}\,d\ch_{\infty}^{\delta}\r]^{\frac{1}{1+\tilde{\gamma}}}
&\leq 2^{\frac{\dz}{1+\tilde{\gamma}}}
\lf[\frac{1}{\ch_{\infty}^{\delta}(Q)}\int_{Q}w(x)^{1+\tilde{\gamma}}\,d\ch_{\infty}^{\delta}\r]^{\frac{1}{1+\tilde{\gamma}}}\\
&\leq \frac{2^{\frac{\dz}{1+\tilde{\gamma}}}\tilde{K}}{\ch_{\infty}^{\delta}(Q)}\int_{Q}w(x)\,d\ch_{\infty}^{\delta}\leq\frac{2^{\frac{\dz}{1+\tilde{\gamma}}+p(\dz+1)}\tilde{K}[w]_{\ca_{p,\dz}}}{\ch_{\infty}^{\delta}(P)}\int_{P}w(x)\,d\ch_{\infty}^{\delta}.
\end{align*}
Therefore, for the cube $P$, inequality \eqref{eq-411-1} also holds with constants $K:=2^{\frac{\dz}{1+\tilde{\gamma}}+p(\dz+1)}\tilde{K}[w]_{\ca_{p,\dz}}$ and $\gamma:=\tilde{\gamma}$.
\end{rem}

Corollary \ref{cor-0714-1} is an immediate consequence of the reverse H\"older inequality and the definition of $\ca_{p,\delta}$. We omit its proof here. With the help of Theorem \ref{lm-411-1}, we now show the self-improving property of $\ca_{p,\delta}$.

\begin{proof}[Proof of Theorem \ref{Cor-411-1}]
It follows from $w\in \ca_{p,\delta}$ that $w^{1-p'}\in\ca_{p', \dz}$ due to Lemma \ref{lm918-2}. By this and Theorem \ref{lm-411-1},
we know that there exist positive constants $\gamma\in(0,1)$ and $K$ such that, for any cube $Q$ of $\rn$,
\begin{equation}\label{eq6-18}
\lf(\frac{1}{\ch_{\infty}^{\delta}(Q)}\int_{Q}w(x)^{-\frac{1+\gamma}{p-1}}\,d\ch_{\infty}^{\delta}\r)^{\frac{1}{1+\gamma}}
\le \frac{K}{\ch_{\infty}^{\delta}(Q)}\int_{Q}w(x)^{-\frac{1}{p-1}}\,d\ch_{\infty}^{\delta}.
\end{equation}
Choosing $q\in(1,p)$ such that $\frac{1+\gamma}{p-1}=\frac{1}{q-1}$, we then infer that $w\in \ca_{q,\delta}$ from $\eqref{eq6-18}$, which completes the proof of Theorem \ref{Cor-411-1}.
\end{proof}

\section{Applications\label{s5}}

In this section, we present two applications of the main results established in this paper: a Jones-type factorization theorem for $\mathcal{A}_{p,\delta}$ weights as well as the Hausdorff content $A_p$ weight characterization for the boundedness of the maximal operator $\mathcal{M}_{\ch_\fz^\delta}$ on weak weighted Choquet–Lebesgue spaces.

\subsection{Jones' Factorization Theorem within the $\ca_{p,\delta}$ Framework}

Applying the boundedness of the operator $\cm_{\ch_\fz^\delta}$ established in Theorem \ref{them-1013-1},
we prove Jones' factorization theorem within the $\ca_{p,\delta}$ framework stated in Theorem \ref{lm-410-3}, which provides a representation of Hausdorff content $A_p$ weight
class $\ca_{p,\delta}$ for all $p\in [1,\fz)$ and $\delta\in(0, n]$.

Let $\delta\in(0, n]$, $p\in (1,\fz)$ and $w\in\ca_{p,\delta}$. For any $f\in L^{pp'}(\rn, \ch_{\fz}^{\delta})$, define the following three operators by setting for any $x\in\rn$,
$$T_{1}(f)(x):=\lf[\cm_{\ch_{\fz}^{\delta}}\lf(|f|^{p'}w^{-\frac{1}{p}}\r)(x)\r]^{\frac{1}{p'}}w(x)^{\frac1{p{p'}}},$$
$$T_{2}(f)(x):=\lf[\cm_{\ch_{\fz}^{\delta}}\lf(|f|^pw^{\frac{1}{p}}\r)(x)\r]^{\frac{1}{p}}w(x)^{-\frac1{p^2}}$$
and $$T_3:=T_1+T_2.$$ Then we have the following conclusions.

\begin{lem}\label{lm918-5}(Minkowski-type inequality)
Let $\dz\in(0,n ]$, $p\in(1,\fz)$ and $\{f_j\}_{j\in\nn}$ be a sequence of non-negative functions on $\rn$. If for every $j\in\nn$, $f_j\in L^p(\rn, \widetilde{\ch}_\fz^{\delta})$, then we have
\begin{align}\label{eq918-z}
\lf\{\int_{\rn}\lf(\sum_{j\in\nn}f_j(x)\r)^p\,d\widetilde{\ch}_\fz^{\delta}\r\}^{\frac{1}{p}}\leq \sum_{j\in\nn}\lf\{\int_{\rn}f_j(x)^p\,d\widetilde{\ch}_\fz^{\delta}\r\}^{\frac{1}{p}}.
\end{align}
\end{lem}

\begin{proof}
Notice that $\widetilde{\ch}_\fz^{\delta}$ is strongly subadditive (see \cite{Ada15}), that is, for any sets $E,F\subset\rn$,
\[\widetilde{\ch}_\fz^{\delta}(E\cup F)+\widetilde{\ch}_\fz^{\delta}(E\cap F)\leq \widetilde{\ch}_\fz^{\delta}(E)+\widetilde{\ch}_\fz^{\delta}(F).\]
Combining this with \cite[Theorem 2]{cms10}, for any non-negative functions functions $f,g$ on $\rn$ with $f,g\in L^p(\rn, \widetilde{\ch}_\fz^{\delta})$, we obtain
\[\lf\{\int_{\rn}\lf[f(x)+g(x)\r]^p\,d\widetilde{\ch}_\fz^{\delta}\r\}^{\frac{1}{p}}\leq \lf\{\int_{\rn}f(x)^p\,d\widetilde{\ch}_\fz^{\delta}\r\}^{\frac{1}{p}}+\lf\{\int_{\rn}g(x)^p\,d\widetilde{\ch}_\fz^{\delta}\r\}^{\frac{1}{p}}.\]
From this, we immediately get that for any $N\in\nn$,
\begin{align}\label{eq918-a}
\lf\{\int_{\rn}\lf(\sum_{j=1}^{N}f_j(x)\r)^p\,d\widetilde{\ch}_\fz^{\delta}\r\}^{\frac{1}{p}}\leq \sum^{N}_{j=1}\lf\{\int_{\rn}f_j(x)^p\,d\widetilde{\ch}_\fz^{\delta}\r\}^{\frac{1}{p}}\leq  \sum_{j\in\nn}\lf\{\int_{\rn}f_j(x)^p\,d\widetilde{\ch}_\fz^{\delta}\r\}^{\frac{1}{p}} .
\end{align}
By \cite[Proposition 2.6]{BC24}, we have
\[\lim_{N\to\fz}\int_{\rn}\lf(\sum_{j=1}^{N}f_j(x)\r)^p\,d\widetilde{\ch}_\fz^{\delta}=\int_{\rn}\lf(\sum_{j\in\nn}f_j(x)\r)^p\,d\widetilde{\ch}_\fz^{\delta}.\]
Thus, by this and \eqref{eq918-a}, we immediately obtain that \eqref{eq918-z} holds. This completes the proof of Lemma \ref{lm918-5}.
\end{proof}

\begin{lem}\label{lm-410-2}
Let $\dz\in(0,n]$, $p\in(1,\fz)$ and $w\in\ca_{p,\dz}$. Then there exists a positive constant $K$ such that
\begin{enumerate}
\item[{\rm(i)}] for any sequence $\{f_j\}_{j\in\nn}$ of functions in
$L^{p{p'}}(\rn, \ch_{\fz}^{\delta})$ and $i\in\{1,\ 2,\ 3\}$,
$$T_i\lf(\sum_{j\in\nn }f_j\r)(x)\leq K\sum_{j\in\nn}T_{i}(f_j)(x),\quad \forall\,x\in\rn;$$

\item[{\rm(ii)}]
for any $f\in L^{p{p'}}(\rn, \ch_{\fz}^{\delta})$ and $i\in\{1,\ 2,\ 3\}$,
$$\int_{\rn}|T_i(f)(x)|^{p{p'}}\,d\ch_{\fz}^{\delta}\leq K\int_{\rn}|f(x)|^{p{p'}}\,d\ch_{\fz}^{\delta}.$$
\end{enumerate}
\end{lem}

\begin{proof}
(i) By \eqref{eq918-z} and \eqref{eq918-5}, we obtain
\begin{align*}
\lf\{\int_{\rn}\lf(\sum_{j\in\nn}|f_j(x)|\r)^p\,d\ch^{\delta}_{\infty}\r\}^{\frac1{p}}
\leq K(n,\dz)^{\frac{1}{p}}\sum_{j\in\nn}\lf\{\int_{\rn}|f_j(x)|^p\,d\ch^{\delta}_{\infty}\r\}^{\frac1{p}},
\end{align*}
where $K(n, \dz)$ comes from \eqref{eq918-5}. By this, we find that, for any $x\in\rn$ and any cube $Q$ containing $x$,
\begin{align*}
\lf(\frac{1}{\ch_{\infty}^{\delta}(Q)}\int_{Q}\lf(\sum_{j\in\nn}|f_j(y)| w(y)^{-\frac{1}{p{p'}}}\r)^{p'}\,d\ch_{\fz}^{\delta}\r)^{\frac{1}{p'}}
&\ls\sum_{j\in\nn}\lf(\frac{1}{\ch_{\infty}^{\delta}(Q)}\int_{Q}
|f_j(y)|^{p'}w(y)^{-\frac{1}{p}}\,d\ch_{\fz}^{\delta}\r)^{\frac{1}{p'}}\\
&\ls\sum_{j\in\nn}\lf[\cm_{\ch_{\fz}^{\delta}}(|f_j|^{p'}w^{-\frac{1}{p}})(x)\r]^{\frac{1}{p'}}.
\end{align*}
Therefore, (i) is true for $T_1$ and, similarly for $T_2$ and $T_3$.

(ii) Since $w\in \ca_{p,\delta}$ and hence $w^{1-p'}\in \ca_{{p'},\delta}$  due to Lemma \ref{lm918-2}, it follows from Theorem \ref{them-1013-1} that, for $i=1,\ 2$,
$$\int_{\rn}|T_i(f)(x)|^{p{p'}}\,d\ch_{\fz}^{\delta}\ls \int_{\rn}|f(x)|^{p{p'}}\,d\ch_{\fz}^{\delta},$$
which further implies that (ii) is also true for $T_3$.
This finishes the proof of Lemma \ref{lm-410-2}.
\end{proof}

We now give the proof of Theorem \ref{lm-410-3}.

\begin{proof}[Proof of Theorem \ref{lm-410-3}]
To prove the necessity, we only need to consider the case of $p\in(1,\fz)$,
since it is easy to see that, for $p=1$, $w_0:=w$ and $w_1:\equiv1$ satisfy the requirement.

Let $w\in\ca_{p,\delta}$ and fixed a function $g\in L^{p{p'}}(\rn, \ch_{\fz}^{\delta})$.
We define $$\varphi(x):=\sum_{k=0}^{\infty}\frac{1}{A^k}T_3^kg(x),\quad\forall\,x\in\rn,$$
where $T_3$ is defined above and $A>K^{\frac{1}{p{p'}}}$ with $K$ being the constant in Lemma \ref{lm-410-2}(i).
Then, from Lemmas \ref{lm918-5} and \ref{lm-410-2}, it follows that
\begin{align*}
\lf[\int_{\rn}|\varphi(x)|^{p{p'}}\,d\ch_{\fz}^{\delta}\r]^{\frac{1}{p{p'}}}
&\ls\sum_{k=0}^{\infty}\frac{1}{A^k}\lf[\int_{\rn}|T^k_3(g)(x)|^{p{p'}}\,d\ch_{\fz}^{\delta}\r]^{\frac{1}{p{p'}}}\\
&\ls\sum_{k=0}^{\infty}\frac{K^{\frac{k}{p{p'}}}}{A^k}\lf[\int_{\rn}|g(x)|^{p{p'}}\,d\ch_{\fz}^{\delta}\r]^{\frac{1}{p{p'}}}
\sim \|g\|_{L^{p{p'}}(\rn, \ch_{\fz}^{\dz})},
\end{align*}
namely, $\varphi\in L^{p{p'}}(\rn, \ch_{\fz}^{\delta})$.
From this and Lemma \ref{lm-410-2}, we infer that, for any $x\in\rn$,
\begin{align*}
T_1\varphi(x)&\leq T_3\lf(\sum_{k=0}^{\infty}\frac{1}{A^k}T_3^kg(x)\r)\ls
\sum_{k=0}^{\infty}\frac{1}{A^k}T^{k+1}_3(g)(x)\\
&\sim A\sum_{k=1}^{\infty}\frac{1}{A^k}T^{k}_3(g)(x)\ls \varphi(x),
\end{align*}
and, similarly $T_2\varphi(x)\ls\varphi(x)$.
Therefore, by the definitions of $T_1$ and $T_2$, we further find that
\[\cm_{\ch_{\infty}^{\delta}}(\varphi^{p'}w^{-\frac{1}{p}})\ls\varphi^{p'}w^{-\frac{1}{p}}\ \ \ {\rm and}\ \ \ \cm_{\ch_{\infty}^{\delta}}(\varphi^pw^{\frac{1}{p}})\ls\varphi^pw^{\frac{1}{p}}.\]
Let $w_0:=\varphi^pw^{\frac{1}{p}}$ and $w_1:=\varphi^{p'}w^{-\frac{1}{p}}$.
Then $w_0,\, w_1\in \ca_{1,\delta}$ and
 $$w_0w_1^{1-p}=\varphi^pw^{\frac{1}{p}}\varphi^{{p'}(1-p)}w^{-\frac{1-p}{p}}=w.$$

Conversely, let $w_0,\, w_1\in \ca_{1,\delta}$ and $w:=w_0w_1^{1-p}$. Then, for any cube $Q$ of $\rn$, we have
\[\frac{1}{\ch_{\fz}^{\delta}(Q)}\int_{Q}w_i(x)\,d\ch_{\fz}^{\delta}\| w_i ^{-1}\|_{L^{\infty}(Q,{\mathcal H^{\delta}_{\infty}})}\leq [w_i]_{\ca_{1,\delta}},\quad i\in\{0,1\}.\]
By this, we know that
\begin{align*}
&\frac{1}{\ch_{\fz}^{\delta}(Q)}\int_{Q}w(x)\,d\ch_{\fz}^{\delta}\lf(\frac{1}{\ch_{\fz}^{\delta}(Q)}\int_{Q}w(x)^{-\frac{1}{p-1}}\,d\ch_{\fz}^{\delta}\r)^{p-1}\\
&\hs \leq \frac{1}{\ch_{\fz}^{\delta}(Q)}\int_{Q}w_0(x)\,d\ch_{\fz}^{\delta}\|w_1^{-(p-1)} \|_{L^{\infty}(Q,{\mathcal H^{\delta}_{\infty}})}\lf(\frac{1}{\ch_{\fz}^{\delta}(Q)}\int_{Q}w_1(x)\,d\ch_{\fz}^{\delta}\|w_0^{-\frac{1}{p-1}} \|_{L^{\infty}(Q,{\mathcal H^{\delta}_{\infty}})}\r)^{p-1}\\
&\hs = \frac{1}{\ch_{\fz}^{\delta}(Q)}\int_{Q}w_0(x)\,d\ch_{\fz}^{\delta}\| w_0 ^{-1}\|_{L^{\infty}(Q,{\mathcal H^{\delta}_{\infty}})}\lf(\frac{1}{\ch_{\fz}^{\delta}(Q)}\int_{Q}w_1(x)\,d\ch_{\fz}^{\delta}\| w_1 ^{-1}\|_{L^{\infty}(Q,{\mathcal H^{\delta}_{\infty}})}\r)^{p-1}\\
&\hs \leq [w_0]_{\ca_{1,\delta}}[w_1]_{\ca_{1,\delta}}^{p-1}.
\end{align*}
Therefore, $w\in \ca_{p,\delta}$. This finishes the proof of Theorem \ref{lm-410-3}.
\end{proof}

\subsection{Characterization for the Boundedness of $\cm_{\ch_{\infty}^{\delta}}$ on Weak-Type Spaces}
Finally, combining Theorems \ref{them-1013-1} and \ref{lm-411-1}, we characterize the boundedness of the operator $\cm_{\ch_{\fz}^{\dz}}$ on the weak weighted Choquet-Lebesgue space $L^{p,\fz}_{w}(\rn, \ch_{\fz}^{\dz})$ via Hausdorff content $A_p$ weight class $\ca_{p,\dz}$ as follows. In the special case $\delta=n$ under the restriction of Lebesgue measurable functions, it reduces to a known classical equivalence for Muckenhoupt $A_p$ weight setting \cite{Chung1982} (see also \cite{Accomazzo2023}).
\begin{thm}\label{them-411-1}
Let $\delta\in (0, n]$ and $p\in(1, \infty)$.
Then the maximal operator $$\cm_{\ch_\fz^\delta}:\, L^{p, \infty}_w(\rn,\ch_\fz^\delta)\to L_w^{p,\infty}(\rn,\ch_\fz^\delta)$$ is bounded, i.e.,
there exists a positive constant $K$ such that, for any $f\in L^{p, \infty}_w(\rn,\ch_\fz^\delta)$,
\begin{align}\label{eq-411-8}
\lf\|\cm_{\ch_\fz^\delta}(f)\r\|_{L^{p, \infty}_w(\rn,\ch_\fz^\delta)}\leq K\lf\|f\r\|_{L^{p, \infty}_w(\rn,\ch_\fz^\delta)}
\end{align}
if and only if $w\in \ca_{p,\delta}$.
\end{thm}

\begin{proof}
If \eqref{eq-411-8} holds true, then by Chebyshev's inequality, we know that
the maximal operator $\cm_{\ch_\fz^\delta}$ is of weak-type $(p,p)$. Therefore, by Theorem \ref{them-1013-1}, we obtain $w\in\ca_{p,\delta}$.

Conversely, let $w\in \ca_{p,\delta}$. We only need to show that, for any $\lambda \in(0,\fz)$,
\begin{align}\label{eq-411-9}
w_{\ch_{\infty}^{\delta}}\lf(\lf\{x\in\rn: \cm_{\ch_\fz^\delta}(f)(x)>4\lambda\r\}\r)
\ls \frac{1}{\lambda^p}\|f\|^p_{L^{p, \infty}_w(\rn,\ch_\fz^\delta)}.
\end{align}
For any $x\in\rn$, define
$$f^{\lambda}(x):=\left\{\begin{matrix}
f(x),  \ \ \ \  {\rm when}\ \ |f(x)|>\lambda,\\
\ \ \ \ 0,\ \ \ \  \ \  {\rm when}\ \ |f(x)|\leq\lambda,\end{matrix}\right. $$
and $f_{\lambda}(x):=f(x)-f^{\lambda}(x)$.
Then $\cm_{\ch_\fz^\delta}(f_\lz)\le \lz$ and
$$\cm_{\ch_\fz^\delta}(f)\leq 2\lf[\cm_{\ch_\fz^\delta}(f^{\lambda})+\cm_{\ch_\fz^\delta}(f_{\lambda})\r].$$
Thus,
\begin{align*}
w_{\ch_{\infty}^{\delta}}\lf(\lf\{x\in\rn:\ \cm_{\ch_\fz^\delta}(f)(x)>4\lambda\r\}\r)
\leq w_{\ch_{\infty}^{\delta}}\lf(\lf\{x\in\rn:\ \cm_{\ch_\fz^\delta}(f^{\lambda})(x)>\lambda\r\}\r).
\end{align*}
Since $w\in\ca_{p,\dz}$, it follows from Theorem \ref{Cor-411-1}, that we can choose a $p_1\in (1,p)$ such that $w\in\ca_{p_1,\dz}$, which, combined with Theorem \ref{them-1013-1},
implies
\begin{align*}
w_{\ch_{\infty}^{\delta}}\lf(\lf\{x\in\rn:\ \cm_{\ch_\fz^\delta}(f)(x)>4\lambda\r\}\r)
&\ls\frac{1}{\lambda^{p_1}}\int_{\rn}|f^{\lambda}(x)|^{p_1}w(x)\,d\ch_{\infty}^{\delta}.
\end{align*}
In addition,
by Proposition \ref{them-0919-4}, we know that
\begin{align*}
\int_{\rn}|f^{\lambda}(x)|^{p_1}w(x)\,d\ch_{\infty}^{\delta}
&\leq 4\int_{\rn}|f^{\lambda}(x)|^{p_1}\,dw_{\ch_{\infty}^{\delta}}\\
&= 4p_1\int_{0}^{\infty}t^{p_1-1}w_{\ch_{\infty}^{\delta}}(\{x\in\rn:\ |f(x)|>\max(t, \lambda)\})\,dt\\
&=4p_1\int_{0}^{\lambda}t^{p_1-1}w_{\ch_{\infty}^{\delta}}(\{x\in\rn:\ |f(x)|> \lambda\})\,dt\\
&\quad\ +4p_1\int_{\lambda}^{\infty}t^{p_1-1}w_{\ch_{\infty}^{\delta}}(\{x\in\rn:\ |f(x)|>t\})\,dt\\
&\leq4p_1\|f\|^p_{L^{p, \infty}_w(\rn,\ch_\fz^\delta)}\lf(\int_{0}^{\lambda}\frac{t^{p_1-1}}{\lambda^p}\,dt+\int_{\lambda}^{\infty}t^{p_1-p-1}\,dt\r)\\
&=\frac{4p}{p-p_1}\lambda^{p_1-p}\|f\|^p_{L^{p, \infty}_w(\rn,\ch_\fz^\delta)}.
\end{align*}
Therefore,
\begin{align*}
w_{\ch_{\infty}^{\delta}}\lf(\lf\{x\in\rn:\ \cm_{\ch_\fz^\delta}(f)(x)>4\lambda\r\}\r)
\ls \frac{4p}{p-p_1}\frac{1}{\lambda^p}\|f\|^p_{L^{p, \infty}_w(\rn,\ch_\fz^\delta)}.
\end{align*}
Thus, \eqref{eq-411-9} is proved, and hence completes the proof of Theorem \ref{them-411-1}.
\end{proof}

\medskip

\noindent\textbf{Acknowledgements}
~The authors would like to sincerely thank Professor Andrei Lerner for his valuable suggestions on the early version of this paper and also wish to thank Professor Sibei Yang for helpful comments.
This project is partly supported by the Guangdong Basic and Applied Basic Research Foundation (Grant No. 2026A1515012216), the National Natural
Science Foundations of China (12201139), and the Natural Science Foundation of Hunan province (2024JJ3023).

%\medskip

%\noindent\textbf{Data availability}

%No data was used for the research described in the paper.

\medskip

\noindent Long Huang

\medskip

\noindent School of Mathematics and Information Science,
Guangzhou University, Guangzhou, 510006, The People's Republic of China

\smallskip

\noindent {\it E-mail}: \texttt{longhuang@gzhu.edu.cn}

\medskip
\medskip

\noindent Yangzhi Zhang and Ciqiang Zhuo (Corresponding author)
	
\smallskip
	
\noindent Key Laboratory of Computing and Stochastic Mathematics
(Ministry of Education), School of Mathematics and Statistics,
Hunan Normal University,
Changsha, Hunan 410081, The People's Republic of China
	
\smallskip

\noindent{\it E-mails:}
\texttt{yzzhang@hunnu.edu.cn}

\noindent\phantom{ {\it E-mails}}
\texttt{cqzhuo87@hunnu.edu.cn}

\end{document}